\documentclass[twoside,11pt]{article}

\usepackage{jmlr2e}

\usepackage{amsmath}\allowdisplaybreaks
\usepackage{amsfonts,bm}
\usepackage{amssymb}
\usepackage{algorithm}
\usepackage{algorithmic}
\usepackage{multirow}
\usepackage{booktabs}
\usepackage{xcolor}

\def\eqref#1{equation~(\ref{#1})}
\def\Eqref#1{Equation~\ref{#1}}

\def\1{\bf{1}}

\def\vone{{\bf{1}}}

\def\vf{{\bf{f}}}
\def\vg{{\bf{g}}}

\def\fB{{\mathcal{B}}}

\def\fD{{\mathcal{D}}}

\def\fL{{\mathcal{L}}}

\def\fO{{\mathcal{O}}}

\def\fT{{\mathcal{T}}}

\def\fX{{\mathcal{X}}}
\def\fY{{\mathcal{Y}}}

\def\sB{{\mathbb{B}}}

\def\BE{{\mathbb{E}}}

\def\BI{{\mathbb{I}}}

\def\BP{{\mathbb{P}}}

\def\BR{{\mathbb{R}}}

\newtheorem{thm}{Theorem}[section]
\newtheorem{dfn}{Definition}[section]

\newtheorem{lem}{Lemma}[section]
\newtheorem{asm}{Assumption}[section]

\newtheorem{cor}{Corollary}[section]
\newtheorem{prop}{Proposition}[section]

\makeatletter
\def\Ddots{\mathinner{\mkern1mu\raise\p@
\vbox{\kern7\p@\hbox{.}}\mkern2mu
\raise4\p@\hbox{.}\mkern2mu\raise7\p@\hbox{.}\mkern1mu}}
\makeatother

\makeatletter
\newcommand*{\rom}[1]{\expandafter\@slowromancap\romannumeral #1@}
\makeatother

\usepackage{hyperref,url}
\usepackage{thmtools}
\usepackage{thm-restate}
\usepackage{nicefrac}
\usepackage{enumitem}
\usepackage{threeparttable}
\usepackage{makecell}

\usepackage{lastpage}
\jmlrheading{26}{2025}{1-\pageref{LastPage}}{8/23; Revised
2/25}{5/25}{23-1104}{Lesi Chen, Yaohua Ma and Jingzhao Zhang}
\ShortHeadings{Near-Optimal Bilevel Optimization with Fully First-Order Oracles}{Chen, Ma and Zhang}

\begin{document}

\title{ {Near-Optimal} Nonconvex-Strongly-Convex
Bilevel Optimization with Fully First-Order Oracles }

\author{\name Lesi Chen\textsuperscript{*} \email chenlc23@mails.tsinghua.edu.cn \\
       \addr  IIIS, Tsinghua University\\
       Shanghai Qi Zhi Institute \\
       Beijing, China
       \AND
       \name Yaohua Ma\textsuperscript{*} \email ma-yh21@mails.tsinghua.edu.cn \\
       \addr IIIS, Tsinghua University \\
       Beijing, China
       \AND
       \name Jingzhao Zhang\textsuperscript{$\dagger$} \email jingzhaoz@mail.tsinghua.edu.cn \\
       \addr 
       IIIS, Tsinghua University\\
       Shanghai AI Lab \\
       Shanghai Qi Zhi Institute\\
       Beijing, China
       }
\editor{Dan Alistarh}

\maketitle
\begingroup
\begin{NoHyper}
\renewcommand\thefootnote{*}
\footnotetext{Equal contributions.}
\end{NoHyper}
\begin{NoHyper}
\renewcommand\thefootnote{$\dagger$}
\footnotetext{The corresponding author.}
\end{NoHyper}
\endgroup


\begin{abstract}
In this work, we consider bilevel optimization when the lower-level problem is strongly convex.
Recent works show that with a Hessian-vector product (HVP) oracle, one can provably find an $\epsilon$-stationary point within ${\mathcal{O}}(\epsilon^{-2})$ oracle calls. However, the HVP oracle may be inaccessible or expensive in practice.
Kwon et al. (ICML 2023) addressed this issue by proposing a first-order method that can achieve the same goal at a slower rate of $\tilde{\mathcal{O}}(\epsilon^{-3})$. 
In this paper, we incorporate a two-time-scale update to improve their method
to achieve the near-optimal $\tilde {\mathcal{O}}(\epsilon^{-2})$ first-order oracle complexity.
Our analysis is highly extensible. In the stochastic setting,
our algorithm can achieve the stochastic first-order oracle complexity of $\tilde {\mathcal{O}}(\epsilon^{-4})$ and $\tilde {\mathcal{O}}(\epsilon^{-6})$ when the stochastic noises are only in the
 upper-level objective and in both level objectives, respectively.
 When the objectives have higher-order smoothness conditions, our deterministic method can escape saddle points by injecting noise, and can be accelerated to 
achieve a faster rate of $\tilde {\mathcal{O}}(\epsilon^{-1.75})$ using Nesterov's momentum.
%
\end{abstract}


\section{Introduction}

In this paper, we consider the following bilevel optimization problem:
\begin{align} \label{hyper-refo}
    \min_{x \in \BR^{d_x}} \varphi(x) := f(x,y^*(x)) , \quad {\rm where}~ y^*(x) = \arg \min_{y \in \BR^{d_y}} g(x,y).
\end{align}
We call $f$ the upper-level problem, $g$ the lower-level problem, and $\varphi$ the hyper-objective.
This formulation covers many applications, including but not limited to hyperparameter tuning~\citep{franceschi2018bilevel,pedregosa2016hyperparameter,bae2020delta,mackay2018self}, neural architecture search~\citep{liu2018darts,wang2022zarts,zoph2016neural,zhang2021idarts}, meta-learning~\citep{finn2017model,fallah2020convergence,andrychowicz2016learning,mishchenko2023convergence,rajeswaran2019meta,chandra2022gradient,ji2022theoretical,zhou2019efficient}, adversarial training~\citep{zhang2022revisiting,bruckner2011stackelberg,wang2021fast,bishop2020optimal,wang2022solving,robey2023adversarial} and data hyper-cleaning~\citep{gao2022self,zhou2022model,ren2018learning,yong2022holistic,shu2019meta,li2022ood}.

This work focuses on the complexity of solving Problem \ref{hyper-refo}.
Since the hyper-objective $\varphi(x)$ is usually a nonconvex function, 
a common goal for non-asymptotic analysis is to find an approximate
stationary point~\citep{carmon2020lower,carmon2021lower}. 
{When $g(x,\,\cdot\,)$ is convex but not strongly convex,  this goal is intractable as there exists a hard instance such that any first-order algorithm would always get stuck and have no progress in $x$ \citep[Theorem 3.2]{chen2024finding}.}
Therefore, it is common to impose the lower-level strong convexity assumption in the literature~\citep{dagreou2022framework,ji2021bilevel,chen2021closing,khanduri2021near,hong2023two}:
In this case,
the hyper-gradient $\nabla \varphi(x)$ can be expressed by:
\begin{align} \label{hyper-grad} 
\begin{split}
    \nabla \varphi(x) 
    &= 
    \nabla_x f(x,y^*(x)) + \nabla y^*(x) \nabla_y f(x,y^*(x)) \\
    &=
    \nabla_x f(x,y^*(x)) - \nabla_{xy}^2 g(x,y^*(x)) [ \nabla_{yy}^2 g(x,y^*(x))]^{-1} \nabla_y f(x,y^*(x)).
\end{split}
\end{align}
This equation enables one to implement
Hessian-vector-product(HVP)-based methods \citep{ji2021bilevel,ghadimi2018approximation} for nonconvex-strongly-convex bilevel optimization. 
These methods query 
HVP oracles to 
estimate $\nabla \varphi(x)$ via \Eqref{hyper-grad}, and then
perform the so-called hyper-gradient descent on $\varphi(x)$. 
By using known convergence results of gradient descent for smooth nonconvex objectives, these methods can find an $\epsilon$-first-order stationary point of $\varphi(x)$ with $\tilde \fO(\epsilon^{-2})$ HVP oracle calls.
However, in many applications~\citep{sow2022convergence,song2019maml,finn2017model,nichol2018first}, calling the HVP oracle may be very costly and become the computational bottleneck in bilevel optimization.

Very recently,  \citet{kwon2023fully}  proposed a novel gradient-based algorithm that can find an $\epsilon$-first-order stationary point of $\varphi(x)$ without resorting to second-order information.  
The key idea is to exploit the following value-function reformulation~\citep{outrata1990numerical,ye1995optimality,ye2010new,dempe2002foundations,lin2014solving} for \Eqref{hyper-refo}:
\begin{align} \label{value-refo}
    \min_{x \in \BR^{d_x}, y \in \BR^{d_y}} f(x,y) , \quad {\rm subject ~to}~~ g(x,y) - g^*(x) \le 0.
\end{align}
Here $g^*(x) = g(x,y^*(x))$ is usually called the value-function. The 
Lagrangian with a multiplier $\lambda >0$ for this inequality-constrained problem takes the form of:
\begin{align} \label{eq:L}
    \fL_{\lambda}(x,y):= f(x,y)+ \lambda ( g(x,y) - g^*(x)).
\end{align}
\citet{kwon2023fully} further defines the following auxiliary function:
\begin{align}\label{eq:L-lambda}
\begin{split}
     \fL_{\lambda}^*(x) := \fL_{\lambda}(x,y_{\lambda}^*(x)),~
     {\rm where }~ y_{\lambda}^*(x) = \arg \min_{y \in \BR^{d_y}} \fL_{\lambda}(x,y),
\end{split}
\end{align}
and showed that $\fL_{\lambda}^*(x)$ is a good proxy of $\varphi(x)$ with
\begin{align*}
    \Vert \nabla \fL_{\lambda}^*(x) - \nabla \varphi(x)  \Vert = \fO \left( \kappa^3 /\lambda \right).
\end{align*}

\begin{table*}[t]
    \centering
    \caption{
    We present the oracle complexities of different \textbf{deterministic} methods for finding an $\epsilon$-first-order stationary point of the hyper-objective under Assumption \ref{asm:sc}.}
    \label{tab:res-1st}
    \begin{tabular}{c c c}
    \hline 
    Oracle     & Method   & Oracle Calls    \\
    \hline 
    \addlinespace
         & BA \citep{ghadimi2018approximation}  & $\fO( \epsilon^{-2.5} )$ 
    \\
     \addlinespace
    HVP & AID \citep{ji2021bilevel}  & $\fO ( \epsilon^{-2}   )$  \\
    \addlinespace
    & ITD \citep{ji2021bilevel}  & $ \tilde \fO ( \epsilon^{-2} )$  \\
    \addlinespace
    \hline
    \addlinespace
    & PZOBO~\citep{sow2022convergence}  & $\tilde \fO( d_{x}^2\epsilon^{-4} )$
    \\
    \addlinespace
    & BOME~\citep{liu2022bome}~\textsuperscript{{\color{blue}(a)}}  &$ \tilde \fO(  \epsilon^{-6} )$   \\ \addlinespace
    Gradient & PBGD~\citep{shen2023penalty}~\textsuperscript{{\color{blue}(a)}}  & $ \tilde \fO(  \epsilon^{-3} )$ \\ \addlinespace
    & F${}^2$SA~\citep{kwon2023fully}~\textsuperscript{{\color{blue}(b)}}  &$ \tilde \fO(  \epsilon^{-3} )$  \\ 
    \addlinespace
     & 
F${}^2$BA (Algorithm \ref{alg:F2BA}) & $ \tilde \fO( \epsilon^{-2})$  \\ 
     \addlinespace
     \hline
    \end{tabular}
    \begin{tablenotes}
		{\scriptsize   
  \item {\color{blue} (a)} {\citet{shen2023penalty}, \citet{liu2022bome} use a weaker notation of stationary, they proved the convergence to an $\epsilon$-stationary point of $\fL_{\lambda}(x,y)$. We present the complexity of \citet{shen2023penalty}, \citet{liu2022bome} for $\lambda \asymp \epsilon^{-1}$, which also implies a $\fO(\kappa \epsilon)$-stationary point of $\varphi(x)$ as we prove in Appendix \ref{apx:diss-shen}.}
  \item {\color{blue} (b)} {\citet{kwon2023fully} additionally assume $\nabla^2 f$ is Lipschitz and both $\Vert \nabla_x f(x,y)\Vert$, $\Vert \nabla_x g(x,y) \Vert $ are upper bounded. But it is unnecessary if we use a fixed penalty $\lambda$ as discussed in Appendix \ref{apx:diss-kwon}.}
   }
    \end{tablenotes}
\end{table*}

It indicates that when we set $\lambda \asymp \epsilon^{-1}$, an $\epsilon$-first-order stationary point of $\fL_{\lambda}^*(x)$ is also an $\fO(\epsilon)$-first-order stationary point of $\varphi(x)$. Therefore, we
can reduce the problem of finding an $\epsilon$-first-order stationary point of $\varphi(x)$ to finding that of $\fL_{\lambda}^*(x)$.
The advantage of this reduction is that 
 $\nabla \fL_{\lambda}^*(x)$ can be evaluated with only first-order information of $f$ and $g$:
 \begin{align} \label{eq:nabla-Lag}
 \begin{split}
     \nabla \fL_{\lambda}^*(x) &= 
     \nabla_x \fL_{\lambda}^*(x,y_{\lambda}^*(x)) + \nabla y_{\lambda}^*(x) \underbrace{\nabla_y \fL_{\lambda}^*(x,y_{\lambda}^*(x))}_{=0} 
     \\
     &=\nabla_x f(x, y_{\lambda}^*(x)) + \lambda ( \nabla_x g(x,y_{\lambda}^*(x)) - \nabla_x g(x,y^*(x))).
 \end{split}
 \end{align}
Since $\fL_{\lambda}(x,y)$ in \Eqref{eq:L} is $\fO(\lambda)$-gradient Lipschitz, {the single-time-scale $(\eta_x = \eta_y)$ approach for penalty methods~\citep{kwon2023fully,shen2023penalty}} requires a complexity of $\tilde \fO(\lambda \epsilon^{-2}) = \tilde \fO(\epsilon^{-3} )$
to find an $\epsilon$-first-order stationary point, which is slower than the $\tilde \fO(\epsilon^{-2})$ complexity of HVP-based methods. \citet{kwon2023fully} then conjectured that a fundamental gap may exist between gradient-based and HVP-based methods.

However, this paper refutes this conjecture by showing that gradient-based methods can also achieve the near-optimal $\tilde \fO(\epsilon^{-2})$ rate as HVP-based methods.
We prove that for a sufficiently large $\lambda$, the proxy $\fL_{\lambda}^*(x)$ satisfies:
\begin{align*}
    \Vert \nabla^2 \fL_{\lambda}^*(x) \Vert \asymp \Vert \nabla^2 \varphi(x) \Vert \asymp \kappa^3,
\end{align*}
where $\kappa$ is the condition number that will be formally defined later. It indicates that although $\lambda$ is large, the gradient Lipschitz coefficient of $\fL_{\lambda}^*(x)$ would remain constant and not depend on $\lambda$.
{Therefore, although the largest possible step size in $y$ is still bounded by $\fO(1/\lambda)$, \textit{i.e.} $\eta_y = \fO(\epsilon)$, we can use a larger step size in $x$  because the landscape in $x$ is much smoother, \textit{i.e.} $\eta_x = \fO(1)$. It motivates us to propose the Fully First-order Bilevel Approximation (F${}^2$BA) which uses the two-time-scale step size $(\eta_x \ne \eta_y)$ in the method by \citet{kwon2023fully}. Our proposed method can find an $\epsilon$-stationary point with $\tilde \fO(\epsilon^{-2})$ complexity.} As nonconvex-strongly-concave bilevel optimization problems subsume standard nonconvex optimization problems, the $\tilde \fO(\epsilon^{-2})$ upper bound is optimal up to logarithmic factors according to the lower bound provided by ~\cite{carmon2020lower}. We compare our complexity result with prior works in Table \ref{tab:res-1st}. In the stochastic setting, 
we prove that our algorithm enjoys a complexity of $\tilde \fO(\epsilon^{-4})$ and $\tilde \fO(\epsilon^{-6})$
 for partially and fully stochastic cases, respectively, which also improves the best-known results \citet{kwon2023fully}. We compare our result in the stochastic case with prior works in Table \ref{tab:res-1st-stoc}.


\begin{table*}[t]
    \centering
    \caption{
    {We present the oracle complexities of different \textbf{stochastic} methods for finding an $\epsilon$-first-order stationary point of the hyper-objective under Assumption \ref{asm:sc}, \ref{asm:stochastic}.
    }} 
    \label{tab:res-1st-stoc}
    \begin{tabular}{c c c c}
    \hline 
    Oracle     & Method  & {\makecell{Partially \\Stochastic}} & {\makecell{Fully \\ Stochastic}}  \\
    \hline 
    \addlinespace
         & BSA \citep{ghadimi2018approximation}  & $\fO( \epsilon^{-6} )$ & $\fO(\epsilon^{-6})$
    \\
     \addlinespace
    HVP* & TTSA \citep{hong2023two}  &  $\fO ( \epsilon^{-5}   )$ & $\fO(\epsilon^{-5})$  \\
    \addlinespace
    & StocBiO \citep{ji2021bilevel}  &  $ \tilde \fO ( \epsilon^{-4} )$ & $\fO(\epsilon^{-4})$  \\
    \addlinespace
    \hline
    \addlinespace
   & PZOBO-S~\citep{sow2022convergence} &  $\tilde \fO(d_x^2 \epsilon^{-8})$ & $\tilde \fO(d_x^2 \epsilon^{-8})$ \\ \addlinespace
Gradient & F${}^2$SA~\citep{kwon2023fully}  & $\tilde \fO(\epsilon^{-5})$ & $\tilde \fO(\epsilon^{-7})$  
     \\ \addlinespace
& F${}^2$BSA (Algorithm \ref{alg:F2BSA}) & $\tilde \fO( \epsilon^{-4})$ 
    & $\tilde \fO(\epsilon^{-6})$ \\ \addlinespace
     \hline
    \end{tabular}
    \begin{tablenotes}
    {\scriptsize   
   \item {We remark in ``*'' that HVP-based methods \textbf{additionally} assume the stochastic estimators of 
   second-order derivatives $\nabla_{xy}^2 g / \nabla_{yy}^2 g$ are unbiased and have bounded variance. Therefore, it is reasonable that gradient-based methods has worse complexities than HVP-based method in the fully stochastic case. Such a gap also exists in nonconvex single-level optimization~\citep{arjevani2020second}.}}
    \end{tablenotes}
\end{table*}


We also study the problem under additional smoothness conditions.
If we additionally assume the Hessian Lipschitz continuity of $f$ and the third-order derivative Lipschitz continuity of $g$, we prove that
\begin{align*}
    \Vert \nabla^2 \fL_{\lambda}^*(x) - \nabla^2 \varphi(x) \Vert = \fO(\kappa^6 /\lambda) ~~{\rm and}~~  \Vert \nabla^3 \fL_{\lambda}^*(x) \Vert \asymp \Vert \nabla^3 \varphi(x) \Vert \asymp \kappa^5.
\end{align*} 
Based on this observation, we propose the Perturbed F${}^2$BA ( or PF${}^2$BA for short, see Algorithm \ref{alg:PF2BA}), which can provably find an $\epsilon$-second-order stationary point of $\varphi(x)$ within $\tilde \fO(\epsilon^{-2})$ gradient oracle calls.
Our result shows that gradient-based methods can also escape saddle points in bilevel optimization like HVP-based methods~\citep{huang2022efficiently}. {Additionally, we further exploit the Hessian Lipschitz continuity of $\fL_{\lambda}^*(x)$ to achieve a better complexity of $\tilde \fO(\epsilon^{-1.75})$ by incorporating Nesterov's acceleration. We name this new method Accelerated F${}^2$BA (or AccF${}^2$BA for short, see Algorithm \ref{alg:AccF2BA}). We compare the results for finding second-order stationary points in Table \ref{tab:res-2nd}.}


\begin{table*}[t]
    \centering
    \caption{
    {We present the oracle complexities of different \textbf{deterministic} methods for finding an $\epsilon$-\textbf{second-order} stationary point of the hyper-objective under Asmp. \ref{asm:sc}, \ref{asm:third}.}} 
    \label{tab:res-2nd}
    \begin{tabular}{c c c}
    \hline 
    Oracle     & Method  & Oracle Calls    \\
    \hline 
    \addlinespace
    & PAID \citep{huang2022efficiently}  & $\tilde \fO( \epsilon^{-2} )$  \\ [-0.12cm] 
    HVP \\[-0.12cm]
    & PRAHGD \citep{yang2023accelerating} & $\tilde \fO(\epsilon^{-1.75})$ \\ \addlinespace
    \hline 
    \addlinespace
    & PF${}^2$BA (Algorithm \ref{alg:PF2BA}) &$ \tilde \fO(  \epsilon^{-2})$  \\[-0.12cm]
    Gradient \\[-0.12cm]
& AccF${}^2$BA (Algorithm \ref{alg:AccF2BA})
    & $ \tilde \fO( \epsilon^{-1.75})$  \\ 
     \addlinespace
     \hline
    \end{tabular}
\end{table*}

\paragraph{Notations.}
We use notations $\fO(\,\cdot\,),~ \tilde \fO(\,\cdot\,),~ \Omega(\,\cdot\,),~ \asymp$ as follows: given two functions $p: \BR^+ \rightarrow \BR^+$ and $q: \BR^+ \rightarrow \BR^+$,  $p(x) = \fO(q(x))$ means $ \lim \sup_{x \rightarrow +\infty} {p(x)}/{q(x)} < +\infty $; $p(x) = \tilde \fO(p(x)) $ means there exists some positive integer $k \ge 0$ such that $p(x) = \fO(q(x) \log^k (q(x)))$, $p(x) = \Omega(q(x))$ means $ \lim \sup_{x \rightarrow + \infty} p(x) / q(x) >0$, and $p(x) \asymp q(x)$ means we both have $p(x) = \fO(q(x))$ and $p(x) = \Omega(q(x))$.
We use $I_d \in \BR^{d\times d}$ to denote a $d$=dimensional identity matrix. For two symmetric matrices $A$ and $B$, we use $A \succeq B$ to indicate that $A-B$ is positive semidefinite.
We use $\sB(r)$ to denote the Euclidean ball centered at the origin and radius $r$.
For a function $h: \BR^d \rightarrow \BR$, we use $\nabla h \in \BR^d, \nabla^2 h \in \BR^{d \times d},  \nabla^3 h \in \BR^{d \times d \times d}$ to denote its gradient, Hessian, and third-order derivative, and use $h^*$ to denote the global minimum of $h(\,\cdot\,)$. We
denote $\Vert \cdot\Vert$ to be the operator norm of a tensor and more details about the notations of tensors can be found in Appendix \ref{apx:tensor}.

\section{Related Works}

We review the related works on HVP-based and gradient-based methods.

\paragraph{HVP-Based Methods for Bilevel Optimization.
} 

Most existing HVP-based methods for nonconvex-strongly-convex bilevel optimization can be categorized into
the approximate implicit differentiation (AID) methods  and the iterative differentiation
(ITD) methods. The AID approach~\citep{liao2018reviving,ji2021bilevel,ghadimi2018approximation,lorraine2020optimizing} constructs hyper-gradients explicitly according to \Eqref{hyper-grad} that $\nabla \varphi(x) = \nabla_x f(x,y^*(x)) - \nabla_{xy}^2 g(x,y^*(x)) v^*$, where $v^* $ is the solution to the linear system $\nabla_{yy}^2 g(x,y^*(x)) v^* =\nabla_y f(x,y^*(x)) $. Then one can use
iterative algorithms such as fix point iteration or conjugate gradient method to solve this linear system, avoiding the computation of the Hessian inverse \citep{grazzi2020iteration,ji2021bilevel,hong2023two}. 
The ITD approach~\citep{maclaurin2015gradient,shaban2019truncated,domke2012generic,franceschi2017forward,franceschi2018bilevel}  takes advantage of the fact that backpropagation can be efficiently implemented via modern
automatic differentiation frameworks such as PyTorch~\citep{paszke2019pytorch}. These methods approximate hyper-gradients by $\partial f(x,y^K(x)) / \partial x $, where $y^K(x)$ is the output from $K$-steps of gradient descent on $g(x,\,\cdot\,)$. Although the ITD approach does not query second-order information explicitly, the analytical form of $\partial f(x,y^K(x)) / \partial x $ involves second-order derivatives \citep{ji2021bilevel}, which also requires HVP oracle implicitly. Both AID and ITD methods require $\tilde \fO(\epsilon^{-2})$ HVP oracle calls to find an $\epsilon$-first-order stationary point of $\varphi(x)$. Recently, \citet{huang2022efficiently} proposed the perturbed AID to find an $\epsilon$-second-order stationary point with $\tilde \fO(\epsilon^{-2})$ complexity under additional smoothness conditions. \citet{yang2023accelerating} proposed the Perturbed Restarted Accelerated HyperGradient Descent (PRAHGD) algorithm with an improved complexity of $\tilde \fO(\epsilon^{-1.75})$. \citet[Section 4.2]{wang2024efficient} also mentioned that the $\tilde \fO(\epsilon^{-1.75})$ complexity can also be achieved by applying the Inexact APPA Until Nonconvexity (IAPUN) algorithm to bilevel problems. 

\paragraph{Gradient-Based Methods for Bilevel Optimization.}

The basic idea of gradient-based methods for bilevel optimization is to approximate the hyper-gradient in \Eqref{hyper-grad} using gradient information. \citet{sow2022convergence} proposed the 
Partial Zeroth-Order based Bilevel Optimizer (PZOBO) that
applies a zeroth-order-like estimator to approximate the response Jacobian matrix $\nabla y^*(x) $ in \Eqref{hyper-grad}, and hence 
the complexity has a polynomial 
dependency on the dimension of the problem like the standard results in zeroth-order optimization~\citep{duchi2015optimal,ghadimi2013stochastic,kornowski2023algorithm}. 
\citet{liu2022bome} first observed \Eqref{eq:nabla-Lag} that $\nabla \fL_{\lambda}^*(x)$ only involves first-order information and proposed the method named Bilevel Optimization Made Easy (BOME).
{
\citet{shen2023penalty} studied the relationship of global and local solutions of the penalty function and the original bilevel problem, and proposed the penalty-based bilevel gradient descent (PBGD) that can converge to the stationary point of $\fL_{\lambda}(x,y)$ in \Eqref{eq:L} with $\tilde{\fO}(\lambda \epsilon^{-2})$ oracle calls.
But \citet{liu2022bome}, \citet{shen2023penalty} did not provide any convergence result of $\varphi(x)$.}
Remarkably, \citet{kwon2023fully}
established
the relationship between $\fL_{\lambda}^*(x)$  and $\varphi(x)$, and proposed the Fully First-order Stochastic Approximation (F${}^2$SA) that can provably find an $\epsilon$-first-order stationary point of $\varphi(x)$ within $\tilde \fO(\epsilon^{-3})$ oracle complexity. 
As a by-product of their analysis, one can also show that BOME~\citep{liu2022bome} and PBGD~\citep{shen2023penalty} converge to an $\epsilon$-stationary point of $\varphi(x)$ at the rate of $\tilde \fO(\epsilon^{-6})$ and $\tilde \fO(\epsilon^{-3})$, respectively.
However, before our work, we did not know whether gradient-based methods could have comparable theoretical guarantees to HVP-based methods.



\section{Preliminaries}

In this section, we introduce the different setups studied in this work. We focus on stating the assumptions and definitions used later, while delaying a more comprehensive description to future sections where we state and describe our main results.

\paragraph{First-Order Stationary Points} 

We first discuss the assumptions for finding first-order stationary points of the hyper-objective $\varphi(x)$, detailed below.

\begin{asm} \label{asm:sc}
Suppose that
\begin{enumerate}[label=\alph*.]
    \item $g(x,y)$ is $\mu$-strongly convex in $y$;
     \item $g(x,y)$ is $L_g$-gradient Lipschitz;
    \item $g(x,y)$ is $\rho_g$-Hessian Lipschitz;
     \item $f(x,y)$ is $C_f$-Lipschitz in $y$;
    \item $f(x,y)$ is $L_f$-gradient Lipschitz;
    \item $f(x,y)$ is two-times continuous differentiable;
    \item $\varphi(x)$ is lower bounded, \emph{i.e.} $\inf_{x \in \BR^{d_x}} \varphi(x) > - \infty$;
\end{enumerate}
\end{asm}

The above assumptions are common and necessary for non-asymptotic analyses. {According to Theorem 3.2 in \citep{chen2024finding}, bilevel problems are intractable in general  when $g(x,\,\cdot\,)$ is not strongly convex. }
For this reason, existing non-asymptotic analyses for bilevel optimization commonly make the 
lower-level strong convexity assumption (Assumption \ref{asm:sc}a). In this case,  $\nabla \varphi(x)$ can be expressed jointly by $ \nabla_x f(x,y)$, $\nabla_y f(x,y)$, $\nabla_{xy}^2 g(x,y)$ and $ \nabla_{yy}^2 g(x,y)$ as \Eqref{hyper-grad}.
This expression indicates that we need the smoothness conditions for $f$ and $g$ (Assumption \ref{asm:sc}b - \ref{asm:sc}e) to guarantee the gradient  Lipschitz continuity of $ \varphi(x)$. 
Besides these, 
we adopt Assumption \ref{asm:sc}f to ensure that $ \fL_{\lambda}(x,y)$ (\Eqref{eq:L-lambda}) is two-times continuous differentiable, and
Assumption \ref{asm:sc}g to ensure the bilevel optimization problem (\Eqref{hyper-refo}) is well-defined. 

\begin{dfn}  \label{dfn:kappa-sc}
Under Assumption \ref{asm:sc}, we define the largest smoothness constant $\ell: =\max \{ C_f,L_f,L_g,\rho_g  \}$ and the condition number 
$\kappa:= \ell / \mu$. 
\end{dfn}

We can derive 
from the above assumptions that $\nabla \varphi(x)$ is uniquely defined and Lipschitz continuous, formally stated as follows.

\begin{restatable}[{\citet[Lemma 2.2]{ghadimi2018approximation}}]{prop}{propFsmooth}
    \label{prop:F-smooth}
Under Assumption \ref{asm:sc},
the hyper-gradient $\nabla \varphi(x)$ is uniquely defined by \Eqref{hyper-grad}, and
the hyper-objective $\varphi(x)$ is $L_\varphi$-gradient Lipschitz, where $L_\varphi = \fO(\ell \kappa^3)$.
 \end{restatable}

As the above proposition ensures that the hyper-objective $\varphi(x)$ is differentiable, we can define the $\epsilon$-first-order stationary points as follows.
\begin{dfn} \label{dfn:sta-hyper}
Given a differentiable function $\varphi(x): \BR^d \rightarrow \BR $,
we call $\hat x$ an $\epsilon$-first-order stationary point of $\varphi(x)$ if $ \Vert \nabla \varphi(\hat x) \Vert \le \epsilon$.
\end{dfn}

\paragraph{Second-Order Stationary Points}

Algorithms that pursue first-order stationary points may get stuck in saddle points and have poor performances \citep{,dauphin2014identifying}.
For
single-level optimization problems, there have been many researchers studying how to escape saddle points  
~\citep{jin2017escape,ge2015escaping,fang2019sharp,tripuraneni2018stochastic,agarwal2017finding,lee2016gradient,allen2018neon2,carmon2017convex,zhou2020stochastic,allen2018natasha,xu2018first,zhang2021escape}.
A common assumption in these works is to suppose that the objective is Hessian Lipschitz. 
When generalizing to bilevel optimization, we also expect $\varphi(x)$ to be Hessian Lipschitz, which can be proved if 
we further assume the following higher-order smoothness condition of $f$ and $g$.


\begin{asm} \label{asm:third}
Suppose that
\begin{enumerate}[label=\alph*.]
    \item $f(x,y)$ is three-times continuous differentiable;
    \item $f(x,y)$ is $\rho_f$-Hessian Lipschitz;
    \item  $ g(x,y)$ is $\nu_g$-third-order derivative Lipschitz.
\end{enumerate}
\end{asm}

\begin{dfn} \label{dfn:kappa-third}
Under Assumption \ref{asm:sc} and \ref{asm:third}, we define the largest smoothness constant $\ell: =\max \{ C_f,L_f,L_g,\rho_g,\rho_f,\nu_g  \}$ and the condition number 
$\kappa:= \ell / \mu$. 
\end{dfn}

\begin{prop}[{\citet[Lemma 3.4]{huang2022efficiently}}] \label{prop:Hess-smooth}
    Under Assumption \ref{asm:sc} and \ref{asm:third}, the hyper-objective $\varphi(x)$ is two-times continuously differentiable and $\rho_\varphi$-Hessian Lipschitz, where $\rho_\varphi = \fO(\ell \kappa^5)$.
\end{prop}

We can then formally define the
 approximate second-order stationary point as follows.
\begin{dfn}[\citet{nesterov2006cubic}] \label{dfn:sta-hyper-second}
Given a two-times continuously differentiable function $\varphi(x): \BR^d \rightarrow \BR $ with $\rho$-Lipschitz Hessian,
we call $\hat x$ an $\epsilon$-second-order stationary point of $\varphi(x)$ if 
\begin{align*}
    \Vert \nabla \varphi(\hat x) \Vert \le \epsilon,\quad  \nabla^2 \varphi(\hat x)  \succeq -\sqrt{\rho \epsilon}  I_{d}.
\end{align*}
\end{dfn}
We will discuss finding second-order stationary points for bilevel problems later in Section~\ref{sec:second}.


\section{Finding First-Order Stationary Points} \label{sec:first}

In bilevel optimization, the hyper-objective $\varphi(x)$ is usually a nonconvex function. Since finding the global minimum of a nonconvex function in the worst case requires an exponential number of
queries, a common compromise is to find a local minimum~\citep{ge2016matrix}. First-order stationary points (Definition \ref{dfn:sta-hyper}) are the points that satisfy the first-order necessary condition of a local minimum, which turns out to be a valid optimality criterion for nonconvex optimization.

\subsection{Near-Optimal Rate in the Deterministic Case}

\begin{algorithm*}[t]  
\caption{F${}^2$BA $(x_0,y_0)$} \label{alg:F2BA}
\begin{algorithmic}[1] 
\STATE $ z_0 = y_0$ \\[1mm]
\STATE \textbf{for} $ t =0,1,\cdots,T-1 $ \\[1mm]
\STATE \quad $ y_t^0 =  y_{t}, ~ z_t^0 = z_{t}$ \\[1mm]
\STATE \quad \textbf{for} $ k =0,1,\cdots,K-1$ \\[1mm]
\STATE \quad \quad $ z_t^{k+1} = z_t^{k}- \eta_z \lambda \nabla_y g(x_t, z_t^k)   $ \\[1mm]
\STATE \quad \quad $ y_t^{k+1} = y_t^k - \eta_y \left( \nabla_y f(x_t,y_t^k) +  \lambda  \nabla_y g(x_t,y_t^k) \right)$ \\[1mm]
\STATE \quad \textbf{end for} \\[1mm]
\STATE \quad $z_{t+1} = z_t^K,~ y_{t+1} = y_t^K $ \\[1mm]
\STATE \quad $ \hat \nabla \fL_{\lambda}^*(x_t)= \nabla_x f(x_t,y_{t+1}) + \lambda ( \nabla_x g(x_t,y_{t+1}) - \nabla_x g(x_t,z_{t+1}) )$ \\[1mm]
\STATE \quad  $x_{t+1} = x_t -  \eta_x  \hat \nabla \fL_{\lambda}^*(x_t)$ \\[1mm]
\STATE \textbf{end for} \\[1mm]
\end{algorithmic}
\end{algorithm*}


In this section, we propose our method, namely Fully First-order Bilevel Approximation (F${}^2$BA). The detailed procedure of is presented in Algorithm \ref{alg:F2BA}. The algorithm introduces an auxiliary variable $z \in \BR^{d_y}$, and performs gradient descent jointly in $x,y,z$ to solve the following optimization problem:
\begin{align}  \label{eq:opt-xyz}
     \min_{x \in \BR^{d_x}, y \in \BR^{d_y}} \left\{f(x,y) + \lambda \left(g(x,y) - \min_{z \in \BR^{d_y}} g(x,z)\right)\right\} \overset{\rm Eq. \ref{eq:L-lambda}}{=} \min_{x \in \BR^{d_x}} \fL_{\lambda}^*(x).
\end{align}
The intuition behind the algorithm is that optimizing $\fL_{\lambda}^*(x)$ is almost equivalent to optimizing $\varphi(x)$ when $\lambda$ is large.
Therefore, to analyze the convergence of the algorithm, we first characterize the relationship between 
$\fL_{\lambda}^*(x)$ and $\varphi(x)$ in the following lemmas.
 

\begin{lem} \label{lem:1st}
Suppose Assumption \ref{asm:sc} holds. Define $\ell$,  $\kappa$ according to Definition \ref{dfn:kappa-sc}, and $\fL_{\lambda}^*(x)$ according to \Eqref{eq:L-lambda}. Set $\lambda \ge 2L_f /\mu $, then it holds that 
\begin{enumerate}[label=\alph*.]
    \item $\Vert \nabla \fL_{\lambda}^*(x) - \nabla \varphi(x) \Vert = \fO(\ell \kappa^3 /\lambda),~\forall x \in \BR^{d_x}$ {\rm (Lemma \ref{lem:key-fully})}.
    \item $\vert \fL_{\lambda}^*(x) - \varphi(x) \vert =\fO(\ell \kappa^2/ \lambda),~\forall x \in \BR^{d_x}$ {\rm (Lemma \ref{lem:F-Lag-close})}.
    \item $\fL_{\lambda}^*(x) $ is $\fO(\ell \kappa^3)$-gradient Lipschitz {\rm (Lemma \ref{lem:nabla2-bound})}.
\end{enumerate}
\end{lem}

All the formal versions of these lemmas and the corresponding proofs can be found in Appendix \ref{apx:lem-1st}.
Lemma \ref{lem:1st}a is a restatement of Lemma 3.1 by \citet{kwon2023fully}, which demonstrates that when $\lambda \asymp \epsilon^{-1}$, an $\epsilon$-first-order stationary point of $\fL_{\lambda}^*(x)$ is also an $\fO(\epsilon)$-first-order stationary of $\varphi(x)$.
Lemma \ref{lem:1st}c is a new result proved in this paper. It means although $\fL_{\lambda}^*(x)$ depends on $\lambda$, when $\lambda$ exceeds a certain threshold, the gradient Lipschitz coefficient of $\fL_{\lambda}^*(x)$ only depends on that of $\varphi(x)$ and does not depend on $\lambda$.  {
The high-level intuition is that since we know $\fL_{\lambda}^*(x)$ would converge to a fixed objective $\nabla \varphi(x)$ when $\lambda \rightarrow \infty$, it is possible that the  Lipschitz constant of $\fL^\ast_\lambda(x)$
also converges to a fixed value. 
Below, we sketch the proof of Lemma \ref{lem:1st}c. The non-trivial part is the show that $ \lambda (g(x,y_{\lambda}^*(x)) - g(x,y^*(x)) $ is $\fO(1)$-gradient Lipschitz independent of $\lambda$. We prove this by bounding the operator norm of its second-order derivative, which by calculation is
\begin{align*}
    &\quad \lambda (\nabla_{xx}^2 g(x,y_{\lambda}^*(x))-\nabla_{xx}^2 g(x,y^*(x))) \\
    &\quad + \lambda \left(\nabla y_{\lambda}^*(x)  \nabla_{yx}^2 g(x,y_{\lambda}^*(x)) - \nabla y^*(x) \nabla_{yx}^2 g(x,y^*(x)) \right),
\end{align*}
which should be $\fO(1)$ since we can show that
\begin{align*}
    \Vert y_{\lambda}^*(x) - y^*(x) \Vert = \fO(1/ \lambda), \quad {\rm and} \quad \Vert \nabla y_{\lambda}^*(x) - \nabla y^*(x) \Vert = \fO(1/ \lambda).
\end{align*}}
Since the convergence rate of gradient descent depends on the gradient Lipschitz coefficient of the objective, Lemma \ref{lem:1st}c indicates that optimizing $\fL_{\lambda}^*(x)$ is as easy as optimizing $\varphi(x)$. 
Note that
$\nabla \fL_{\lambda}^*(x)$ only involves first-order information (\Eqref{eq:nabla-Lag}), Lemma \ref{lem:1st}c then suggests that first-order methods can have the same convergence rate as HVP-based methods, as stated in the following theorem.

\begin{restatable}{thm}{thmFBA}\label{thm:F2BA}
Suppose Assumption \ref{asm:sc} holds. Define
$\Delta:=\varphi(x_0) - \inf_{x \in \BR^{d_x}} \varphi(x)$ and $R := \Vert y_0 - y^*(x_0) \Vert^2 $.
Let $\eta_x \asymp \ell^{-1}\kappa^{-3}$, $\lambda \asymp \max\left\{\kappa/R,~ \ell \kappa^2/ \Delta, ~ \ell \kappa^3 / \epsilon \right\} $
and set other parameters 
in Algorithm \ref{alg:F2BA} as
\begin{align*}
     \eta_z = \eta_y = \frac{1}{2 \lambda L_g},~ K = \fO\left( \frac{L_g }{\mu} \log \left( \frac{\lambda L_g}{\mu} \right) \right),
\end{align*}
then it
can find an $\epsilon$-first-order stationary point of $\varphi(x)$ within $\fO( \ell \kappa^4 \epsilon^{-2} \log(\nicefrac{\ell \kappa}{\epsilon})) $ first-order oracle calls, where $\ell$,  $\kappa$ are defined in Definition \ref{dfn:kappa-sc}.
\end{restatable}



\begin{remark} 
When the upper-level function only depends on $x$, \textit{i.e.}
we have $f(x,y) \equiv h(x)$ for some function $h(\,\cdot\,)$, the bilevel problem reduces to a single-level problem, for which \citet{carmon2020lower} proved a lower complexity bound of $\Omega(\epsilon^{-2})$. Therefore, we can conclude that the
first-order oracle complexity of {F${}^2$BA} we proved is near-optimal. 
\end{remark}

We defer the proof of Theorem \ref{thm:F2BA} to Appendix \ref{apx:1st}. The complexity of F${}^2$BA in Theorem \ref{thm:F2BA}
achieves the near-optimal rate in the dependency on $\epsilon$, and matches the state-of-the-art second-order methods AID and ITD ~\citep{ji2021bilevel} in the dependency of $\kappa$. Our result, for the first time, closes the gap between gradient-based and HVP-based methods for nonconvex-strongly-convex bilevel optimization.
In Appendix \ref{apx:dist}, we also discuss the advantage of gradient-based methods compared to HVP-based methods in the distributed scenarios.
The distributed F${}^2$BA is much more easy to implement than HVP-based methods.



{\subsection{Extension to the Stochastic Case }}

\begin{algorithm*}[t]  
\caption{F${}^2$BSA $(x_0,y_0)$} \label{alg:F2BSA}
\begin{algorithmic}[1] 
\STATE $ z_0 = y_0$ \\[1mm]
\STATE \textbf{for} $ t =0,1,\cdots,T-1 $ \\[1mm]
\STATE \quad $ y_t^0 =  y_{t}, ~ z_t^0 = z_{t}$ \\[1mm]
\STATE \quad \textbf{for} $ k =0,1,\cdots,K_t-1$ \\[1mm]
\STATE \quad \quad $ z_t^{k+1} = z_t^{k}- \eta_z \lambda \nabla_y g(x_t, z_t^k;B_{\rm in})   $ \\[1mm]
\STATE \quad \quad $ y_t^{k+1} = y_t^k - \eta_y \left( \nabla_y f(x_t,y_t^k;B_{\rm in}) +  \lambda  \nabla_y g(x_t,y_t^k;B_{\rm in}) \right)$ \\[1mm]
\STATE \quad \textbf{end for} \\[1mm]
\STATE \quad $z_{t+1} = z_t^K,~ y_{t+1} = y_t^K $ \\[1mm]
\STATE \quad $ G_t= \nabla_x f(x_t,y_{t}^{K_t};B_{\rm out}) + \lambda ( \nabla_x g(x_t,y_{t}^{K_t};B_{\rm out}) - \nabla_x g(x_t,z_{t}^K;B_{\rm out}) )$ \\[1mm]
\STATE \quad  $x_{t+1} = x_t -  \eta_x G_t$ \\[1mm]
\STATE \textbf{end for} \\[1mm]
\end{algorithmic}
\end{algorithm*}

In this section, we study the case when the algorithms only have access to stochastic gradient oracles that satisfy the following assumptions.

\begin{asm} \label{asm:stochastic}
We access the gradients of objective functions via unbiased estimators $ \nabla f(x,y; \phi_f)$ and $\nabla g(x,y; \phi_g)$ such that
   \begin{align*}
        \BE_{\phi_f} \left[ \nabla f(x,y;\phi_f) \right] = \nabla f(x,y), \quad \BE_{\phi_g} \left[ \nabla g(x,y;\phi_g) \right] = \nabla g(x,y).
    \end{align*}
And the variance of stochastic gradients is bounded:
\begin{align*}
    \BE_{\phi_f} \left[ \Vert \nabla f(x,y;\phi_f ) - \nabla f(x,y) \Vert^2 \right] \le \sigma_f^2, \quad \BE_{\phi_g} \left[ \Vert \nabla g(x,y;\phi_g) - \nabla g(x,y) \Vert^2 \right] \le \sigma_g^2.
\end{align*}
    
\end{asm}

We propose the Fully First-order Stochastic Bilevel Approximation (F${}^2$BSA) in Algorithm~\ref{alg:F2BSA}. At each iteration, our algorithm samples a mini-batch to estimate the true gradient: 
\begin{align*}
    \nabla f(x,y; B) = \frac{1}{B} \sum_{i=1}^B \nabla f(x,y;\phi_f^{(i)}), \quad 
    \nabla g(x,y;B) = \frac{1}{B} \sum_{i=1}^B \nabla g(x,y;\phi_g^{(i)}),
\end{align*}
where both $\phi_f^{(i)}$ and $\phi_g^{(i)}$ are sampled i.i.d. ate each iteration, and $B$ denotes the batch size. We use $B_{\rm out}$ and $B_{\rm in}$ to denote the batch size of outer and inner loop, respectively. By properly setting the value of $B_{\rm out}$ and $B_{\rm in}$, F${}^2$BSA can track the deterministic method F${}^2$BA up to $\fO(\epsilon)$ error, and therefore also converges to an $\epsilon$-stationary point of $\varphi(x)$.

\begin{restatable}{thm}{thmFBSA}\label{thm:F2BSA}
Suppose Assumption \ref{asm:sc} and \ref{asm:stochastic} hold. 
Let $\ell$,  $\kappa$ defined as Definition \ref{dfn:kappa-sc}.
Define
$\Delta:=\varphi(x_0) - \inf_{x \in \BR^{d_x}} \varphi(x)$ and $R := \Vert y_0 - y^*(x_0) \Vert^2 $.
Let $\eta_x \asymp \ell^{-1}\kappa^{-3}$, $\lambda \asymp \max\left\{\kappa/R,~ \ell \kappa^2/ \Delta, ~ \ell \kappa^3 / \epsilon \right\} $
and set other parameters 
in Algorithm \ref{alg:F2BSA} as
\begin{align*}
     & \eta_z = \eta_y = \frac{1}{2 \lambda L_g},~ K_t = \tilde \fO\left(\frac{L_g  \log \delta_t}{\mu}   \right), \\
     & B_{\rm out} \asymp  \frac{\sigma_f^2 + \lambda^2 \sigma_g^2}{\epsilon^2} , ~~ B_{\rm in} \asymp  \frac{L_f^2 + \lambda^2 L_g^2}{\lambda^2 \mu^2} \cdot B_{\rm out}. 
\end{align*}
where $\delta_t$ is defined via the recursion
\begin{align*} 
    \delta_{t+1} = \frac{1}{2}\delta_t + \frac{34 L_g^2}{\mu^2}  \Vert x_{t+1} - x_t \Vert^2 + \frac{ \sigma_g^2}{ 2 L_g B_{\rm in}}, \quad \delta_0 = \fO(R),
\end{align*}
then it
can output a point such that $\BE \Vert \nabla \varphi(x) \Vert \le \epsilon$ within $ T = \fO( \ell \kappa^3 \epsilon^{-2})$ iterations. The total number of stochastic first-order oracle calls is bounded by
\begin{align*}
\begin{cases}
    \fO(\ell \kappa^6 \epsilon^{-4} \log (\nicefrac{ \ell \kappa}{\epsilon})), & \sigma_f >0 , \sigma_g = 0;\\
    {\fO(\ell^3 \kappa^{12} \epsilon^{-6} \log (\nicefrac{ \ell \kappa}{\epsilon})) }, & \sigma_f>0, \sigma_g >0.
\end{cases}
\end{align*}
\end{restatable}

Our results improve that of \citet{kwon2023fully} by a factor of $\fO(\epsilon^{-1})$ in both cases.
In the partially stochastic case $(\sigma_g >0, \sigma_g=0)$, the $\tilde \fO(\epsilon^{-4})$ upper bound is near-optimal due to the lower bound by \citet{arjevani2023lower}.
However, the $\tilde \fO(\epsilon^{-6})$ complexity in the fully stochastic case $(\sigma_f >0 , \sigma_g >0)$ is worse than the $\tilde{\fO}(\epsilon^{-4})$ upper bound by the current best stochastic HVP-based methods~\citep{ji2021bilevel}. It is reasonable as the stochastic HVP-based methods
rely on stronger assumptions that $\nabla^2 g(x,y; \phi) $ is unbiased and has bounded variance. Similar separation between HVP-based and gradient-based methods in the stochastic setting also exists in single-level optimization~\citep{arjevani2020second}. 

\begin{remark}
A drawback of Theorem \ref{thm:F2BSA} is that it requires a large batch size  $B \asymp (\sigma_f^2 + \lambda^2 \sigma_g^2) \epsilon^{-2}$ to track the deterministic algorithm. The large batch size ensures that the hyper-gradient estimator $G_t$ is nearly unbiased, \textit{i.e.}, $ \Vert \BE G_t - \nabla \fL_{\lambda}^*(x_t) \Vert = \fO(\epsilon)$. In contrast, the algorithms in \citep{kwon2023fully} use \textit{single-batch SGD update} in both inner and outer loops. In Appendix \ref{apx:single-batch}, we propose a modified F${}^2$BSA algorithm that requires only a $\tilde \fO(\kappa)$ batch size, but it remains open whether a single-batch algorithm can achieve the same strong theoretical guarantees as our large-batch algorithm. 

Moreover, as noted by subsequent works \citep{liu2025stochastic,chen2026faster}, the upper bounds in Theorem \ref{thm:F2BSA} can be improved by a factor of $\kappa$ via a tighter analysis in the lower-level SGD sub-solver. The same complexity is also achieved by the algorithm in Appendix \ref{apx:single-batch}.
\end{remark}

\section{Finding Second-Order Stationary Points} \label{sec:second}


\begin{algorithm*}[t]  
\caption{Perturbed F${}^2$BA $(x_0,y_0)$} \label{alg:PF2BA}
\begin{algorithmic}[1] 
\STATE $ z_0 = y_0$ \\[1mm]
\STATE \textbf{for} $ t =0,1,\cdots,T-1 $ \\[1mm]
\STATE \quad $ y_t^0 =  y_{t}, ~ z_t^0 = z_{t}$ \\[1mm]
\STATE \quad \textbf{for} $ k =0,1,\cdots,K_t-1$ \\[1mm]
\STATE \quad \quad $ z_t^{k+1} = z_t^{k}- \eta_z \lambda \nabla_y g(x_t, z_t^k)   $ \\[1mm]
\STATE \quad \quad $ y_t^{k+1} = y_t^k - \eta_y \left( \nabla_y f(x_t,y_t^k) +  \lambda  \nabla_y g(x_t,y_t^k) \right)$ \\[1mm]
\STATE \quad \textbf{end for} \\[1mm]
\STATE \quad $z_{t+1} = z_t^{K_t},~ y_{t+1} = y_t^{K_t} $ \\[1mm]
\STATE \quad $ \hat \nabla \fL_{\lambda}^*(x_t)= \nabla_x f(x_t,y_{t+1}) + \lambda ( \nabla_x g(x_t,y_{t+1}) - \nabla_x g(x_t,z_{t+1}) )$ \\[1mm]
\STATE \quad \textbf{if} $ \Vert \hat \nabla \fL_{\lambda}^*(x_t) \Vert \le \frac{4}{5} \epsilon$ \textbf{and} no perturbation added in the last $\fT  $ steps \\[1mm]
\STATE \quad \quad $x_{t} = x_t - \eta_x \xi_t $, where $\xi_t \sim \sB(r)$  \\[1mm]
\STATE \quad \textbf{end if} \\[1mm]
\STATE \quad $x_{t+1} = x_t -  \eta_x  \hat \nabla \fL_{\lambda}^*(x_t)$ \\[1mm]
\STATE \textbf{end for} \\[1mm]
\end{algorithmic}
\end{algorithm*}

We have shown in the previous section that the F${}^2$BA is near-optimal for finding first-order stationary points. However, a first-order stationary point may be a saddle point or a local maximizer, which needs to be escaped from for an effective optimizer.
For this reason, many works aim to find a second-order stationary point (Definition \ref{dfn:sta-hyper-second}).

\subsection{Perturbed F${}^2$BA}
In this section, we propose a simple variant of F${}^2$BA (Algorithm \ref{alg:PF2BA}) that can achieve this higher goal. 
The only difference to Algorithm \ref{alg:F2BA} is the additional Line 9-10 in Algorithm \ref{alg:PF2BA}, which is motivated by the perturbed strategy for escaping saddle points~\citep{jin2017escape}.


To prove the desired conclusion, we need to extend the analysis in Lemma \ref{lem:1st}
to higher-order derivatives. Below, we show that once $\lambda$ is sufficiently large, $\fL_{\lambda}^*(x)$ and $\varphi(x)$ have not only very close gradients (Lemma \ref{lem:1st}a) but also very close Hessian matrices. 


\begin{lem} \label{lem:2rd}
Suppose both Assumption \ref{asm:sc} and \ref{asm:third} hold. Define $\ell$,  $\kappa$ according to Definition \ref{dfn:kappa-third}, and $\fL_{\lambda}^*(x)$ according to \Eqref{eq:L-lambda}. Set $\lambda \ge 2L_f /\mu $, then it holds that 
\begin{enumerate}[label=\alph*.]
    \item $\Vert \nabla^2 \fL_{\lambda}^*(x) - \nabla^2 \varphi(x) \Vert = \fO(\ell \kappa^6 /\lambda),~\forall x \in \BR^{d_x}$. {\rm (Lemma \ref{lem:nabla2-phi})}
    \item $\fL_{\lambda}^*(x) $ is $\fO(\ell \kappa^5)$-Hessian Lipschitz. {\rm (Lemma \ref{lem:nabla3-bound})}
\end{enumerate}
\end{lem}
These lemmas are the higher-order version of Lemma \ref{lem:1st}, but the proof is much more difficult because $\nabla^2 \varphi(x)$ is very complex and contains third-order derivatives. All the complete proof can be found in Appendix \ref{apx:lem-2nd}.
Based on these lemmas, we can prove the convergence of perturbed F${}^2$BA in the following theorem. 

\begin{restatable}{thm}{thmPFBA}\label{thm:PF2BA}
Suppose both Assumption \ref{asm:sc} and \ref{asm:third} hold.
Define $\Delta:=\varphi(x_0) - \inf_{x \in \BR^{d_x}} \varphi(x)$ and $R := \Vert y_0 - y^*(x_0) \Vert^2 $.
Let $\eta_x \asymp \ell^{-1} \kappa^{-3} $, $\lambda \asymp \max\left\{ \kappa/R,~\ell \kappa^2/ \Delta, ~ \ell \kappa^3 / \epsilon,~\kappa^{3.5} \sqrt{\ell / \epsilon} \right\} $ and set other parameters 
in Algorithm \ref{alg:F2BA} as 
\begin{align*}
    \eta_z = \frac{1}{L_g}, ~ \eta_y = \frac{1}{2 \lambda L_g}, ~r = \fO(\epsilon),  ~ K_t = \tilde \fO\left(\frac{L_g  \log \delta_t}{\mu}   \right),
\end{align*}
where $\delta_t$ is defined via the recursion
\begin{align} \label{eq:delta-t}
    \delta_{t+1} = \frac{1}{2}\delta_t + \frac{34 L_g^2}{\mu^2}  \Vert x_{t+1} - x_t \Vert^2, \quad \delta_0 = \fO(R),
\end{align}
then 
it can find an $\epsilon$-second-order stationary point of $\varphi(x)$ with probability at least $1-\delta$
within $\tilde \fO\left(\ell \kappa^4 \epsilon^{-2} \right) $ first-order oracle calls, where $\ell$,  $\kappa$ are defined in Definition \ref{dfn:kappa-third} and the notation $\tilde \fO(\,\cdot\,)$ hides logarithmic factors of $d_x,\kappa,\ell$, and $\delta,\epsilon$.
\end{restatable}
We defer the proof to Appendix \ref{apx:2nd}.
The above complexity 
for finding $\epsilon$-second-order stationary points matches that  
for finding $\epsilon$-first-order stationary points (Theorem~\ref{thm:F2BA}), up to logarithmic factors. Therefore, we conclude that F${}^2$BA 
can escape saddle points almost for free by simply adding some small perturbation in each step.

{
We remark that it is also possible to design the perturbed version of F${}^2$BSA to escape saddle points using only stochastic gradient oracles via existing techniques ~\citep{xu2018first,allen2018neon2}. We leave them as potential future extensions.}

{\subsection{Accelerated F${}^2$BA}

\begin{algorithm*}[t]  
\caption{Accelerated F${}^2$BA$(x_0,y_0)$} \label{alg:AccF2BA}
\begin{algorithmic}[1] 
\STATE $ z_0 = y_0$, $x_{-1} = x_0$ \\[1mm]
\STATE \textbf{while} $ t < T $ \\[1mm]
\STATE \quad $x_{t+1/2} = x_t + (1-\theta) (x_t -x_{t-1})$ \\[1mm]
\STATE \quad $ y_t^0 =  y_{t}, ~ z_t^0 = z_{t}$ \\[1mm]
\STATE \quad \textbf{for} $ k =0,1,\cdots,K_t-1$ \\[1mm]
\STATE \quad \quad $ z_t^{k+1} = z_t^{k}- \eta_z \lambda \nabla_y g(x_{t+1/2}, z_t^k)   $ \\[1mm]
\STATE \quad \quad $ y_t^{k+1} = y_t^k - \eta_y \left( \nabla_y f(x_{t+1/2},y_t^k) +  \lambda  \nabla_y g(x_{t+1/2},y_t^k) \right)$ \\[1mm]
\STATE \quad \textbf{end for} \\[1mm] 
\STATE \quad $z_{t+1} = z_t^{K_t},~ y_{t+1} = y_t^{K_t} $ \\[1mm]
\STATE \quad $ \hat \nabla \fL_{\lambda}^*(x_{t+1/2})= \nabla_x f(x_{t+1/2},y_{t+1}) + \lambda ( \nabla_x g(x_{t+1/2},y_{t+1}) - \nabla_x g(x_{t+1/2},z_{t+1}) )$ \\[1mm]
\STATE\quad  $x_{t+1} = x_{t+1/2} - \eta_x \hat \nabla \fL_{\lambda}^*(x_{t+1/2}) $ \\[1mm]
\STATE \quad $ t = t+1 $ \\ [1mm]
\STATE \quad \textbf{if} $ t \sum_{j=0}^{t-1} \Vert x_{t+1} - x_t \Vert^2 \ge B^2$ \\[1mm]
\STATE \quad \quad $t=0$, $x_{-1} = x_0 = x_t + \xi_t \vone_{ \Vert \hat \nabla \fL_{\lambda}^*(x_{t+1/2}) \Vert \le \frac{B}{2 \eta_x} }$, {\rm where} $\xi_t \sim \sB(r)$  \\[1mm] 
\STATE \quad \textbf{end if} \\[1mm]
\STATE \textbf{end while} \\[1mm]
\STATE $T_0 = \arg \min_{\lfloor \frac{T}{2} \rfloor \le t \le T-1} \Vert x_{t+1} - x_t \Vert$ \\[1mm]
\STATE \textbf{return} $x_{\rm out} = \frac{1}{T_0+1} \sum_{t=0}^{T_0} x_{t+1/2}$ \\[1mm]
\end{algorithmic}
\end{algorithm*}
}
Compared to F${}^2$BA, the perturbed version relies on the additional Assumption \ref{asm:third} to ensure the Hessian Lipschitz continuity of $\varphi(x)$. This addition assumption not only allows escaping saddle points, but also makes further acceleration being possible.

In this section, we combine our F${}^2$BA with the recently proposed acceleration technique for nonconvex optimization~\citep{li2023restarted} to achieve a faster rate of $\tilde \fO(\epsilon^{-1.75})$ for Hessian Lipschitz objectives. This accelerated F${}^2$BA algorithm is presented in Algorithm \ref{alg:AccF2BA}. The difference to Algorithm \ref{alg:F2BA} is the uses of Nesterov's momentum~\citep{nesterov1983method} in $x$ (Line 3), and the restart strategy by \citep{li2023restarted} in Line 13-15.

\begin{restatable}{thm}{thmAccFBA}\label{thm:AccF2BA}
Suppose both Assumption \ref{asm:sc} and \ref{asm:third} hold.
Define $\Delta:=\varphi(x_0) - \inf_{x \in \BR^{d_x}} \varphi(x)$ and $R := \Vert y_0 - y^*(x_0) \Vert^2 $.
Let $\eta_x \asymp \ell^{-1} \kappa^{-3} $, $\lambda \asymp \max\left\{ \kappa/R,~\ell \kappa^2/ \Delta, ~ \ell \kappa^3 / \epsilon,~\kappa^{3.5} \sqrt{\ell / \epsilon} \right\} $ and set other parameters 
in Algorithm \ref{alg:F2BA} as 
\begin{align*}
    & \eta_z = \eta_y = \frac{1}{2 \lambda L_g},~~ K_t = \tilde \fO\left(\frac{L_g  \log \delta_t}{\mu}   \right), \\
    &  T \asymp \frac{\chi}{\theta},~  B \asymp \frac{1}{\chi^2} \sqrt{\frac{\epsilon}{\ell \kappa^3}}, ~~ \theta \asymp \left( \frac{\ell \epsilon}{\kappa} \right)^{1/4}, ~ r = \fO(\epsilon),
\end{align*}
where $\chi = \fO(\log (\nicefrac{d_x}{\delta \epsilon}))$, 
where $\delta_t$ is defined via the recursion
\begin{align} \label{eq:delta-t-acc}
    \delta_{t+1} = \frac{1}{2}\delta_t + \frac{34 L_g^2}{\mu^2}  \Vert x_{t+1/2} - x_{t-1/2} \Vert^2, \quad \delta_0 = \fO(R),
\end{align}
then 
it can find an $\epsilon$-second-order stationary point of $\varphi(x)$ with probability at least $1-\delta$
within $ \tilde \fO\left( \kappa \ell^{1/2} \rho^{1/4} \epsilon^{-1.75}  \right) = \tilde \fO \left( \kappa^{3.75} \epsilon^{-1.75} \right)  $ first-order oracle calls, where $\ell$,  $\kappa$ are defined in Definition \ref{dfn:kappa-third} and the notation $\tilde \fO(\,\cdot\,)$ hides logarithmic factors of $d_x,\kappa,\ell$, and $\delta,\epsilon$.
\end{restatable}

The $\tilde \fO(\epsilon^{-1.75})$ complexity matches the state-of-the-art methods for Hessian Lipschitz objectives~\citep{jin2018accelerated,carmon2017convex,li2023restarted}. However, it is still an open problem whether this rate is optimal because there still exists a gap between this upper bound and the current best $\Omega(\epsilon^{-1.714})$ lower bound by \citet{carmon2021lower}.

\section{Experiments}

We conduct experiments to showcase the superiority of our proposed methods. We implement the algorithms using PyTorch~\citep{paszke2019pytorch}. In all the experiments, we tune 
tune the step size in the grid $\{10^{-5}, 10^{-4},10^{-3},10^{-2},10^{-1},10^0, 10^1, 10^2,10^3 \}$ and present the best result of each algorithm.

{\subsection{Tuning a Single Regularizer on Linear Regression}


Let $\fD^{\rm tr} = (A^{\rm tr}, b^{\rm tr})$ and $\fD^{\rm val} = (A^{\rm val}, b^{\rm val})$ are the training and validation sets respectively,
We first validate the convergence rate
suggested by the theory on a simple problem:
\begin{align} \label{eq:simple-l2reg}
\begin{split}
    &\min_{x \in \BR} \frac{1}{2} \Vert A^{\rm val} y^*(x) - b^{\rm val} \Vert^2, \\
    \text{where } & y^*(x) = \arg \min_{y \in \BR^p} \frac{1}{2} \Vert A^{\rm tr} y - b^{\rm tr} \Vert^2 + \frac{\sigma(x)}{2} \Vert y \Vert^2,    
\end{split}
\end{align}
where $\sigma(x) = \exp(x)$. We use the    ``abalone'' dataset\footnote{Dataset available at \url{https://www.csie.ntu.edu.tw/~cjlin/libsvmtools/datasets/}}, which contains 4,177 samples and each sample has $8$ features ($p=8$). We split the dataset into the training set and the validation set in a 7:3 ratio.  For this problem, we can  explicitly solve $y^*(x)$ and calculate $\nabla \varphi(x)$, $\nabla_y \fL_{\lambda}(x,y)$, $\nabla_y g(x,z)$ to measure the convergence the algorithm.
We run F${}^2$BA (Algorithm \ref{alg:F2BA}) with $K=10$, $\lambda = 10^3$, and 
the results are shown in Figure \ref{fig:simple-l2reg}.
It can be seen from the figure that the upper-level variable $x$ converges to a stationary point of the hyper-objective $\varphi$, and the lower-level variables $y$ and $z$ converges to the minimizers $y_{\lambda}^*(x)$ and $y^*(x)$, respectively.}

\begin{figure*}[t] \label{fig:simple-l2reg}
\centering
\begin{tabular}{c c c }\includegraphics[scale=0.28]{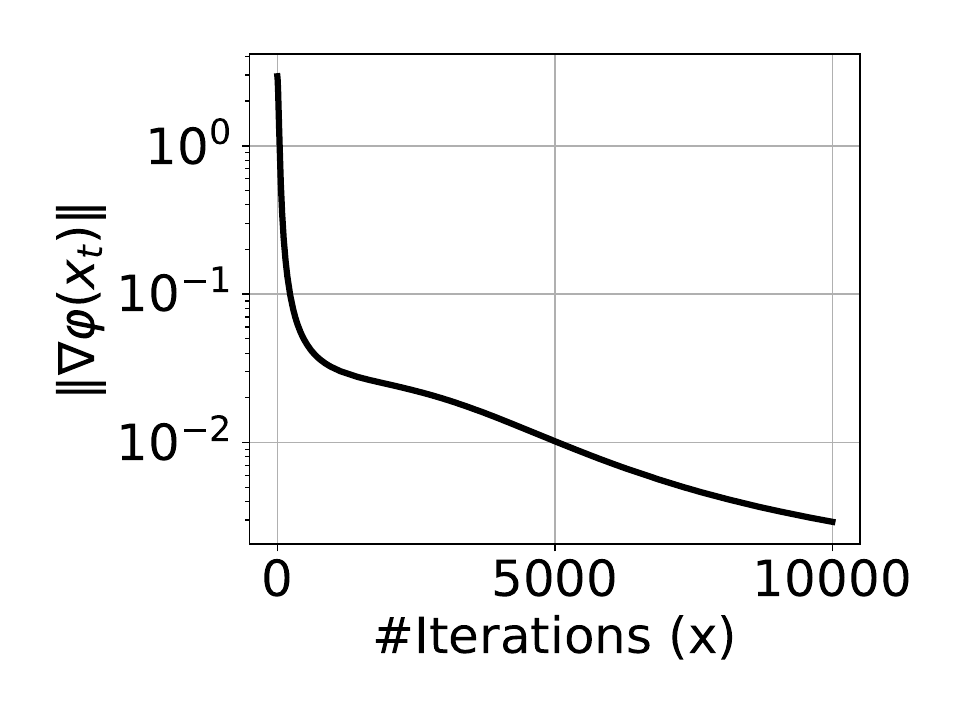} & \includegraphics[scale=0.28]{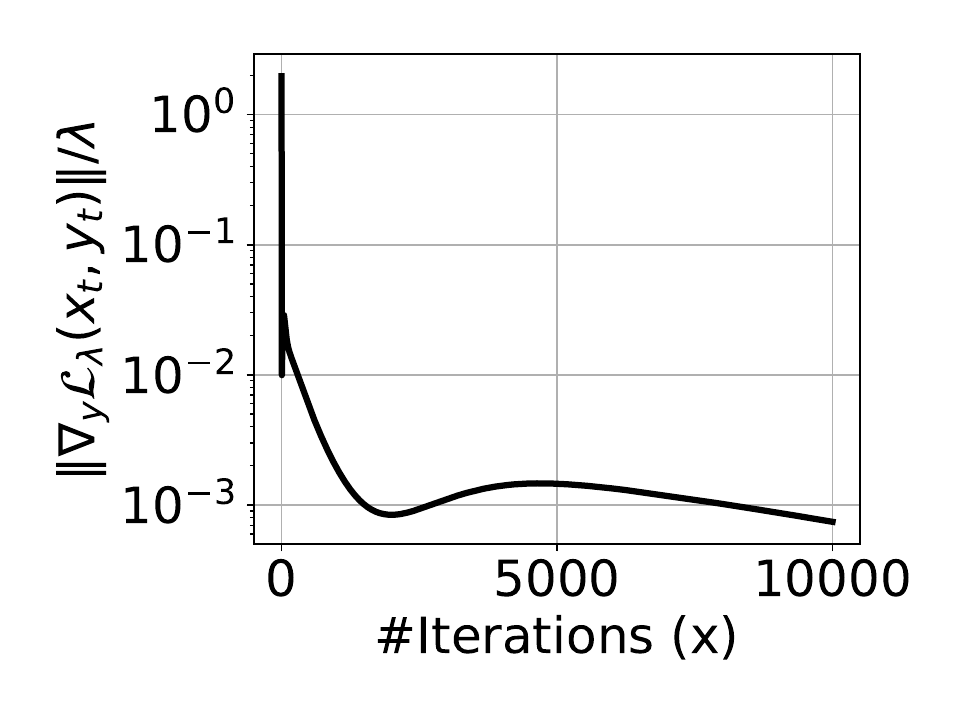} & \includegraphics[scale=0.28]{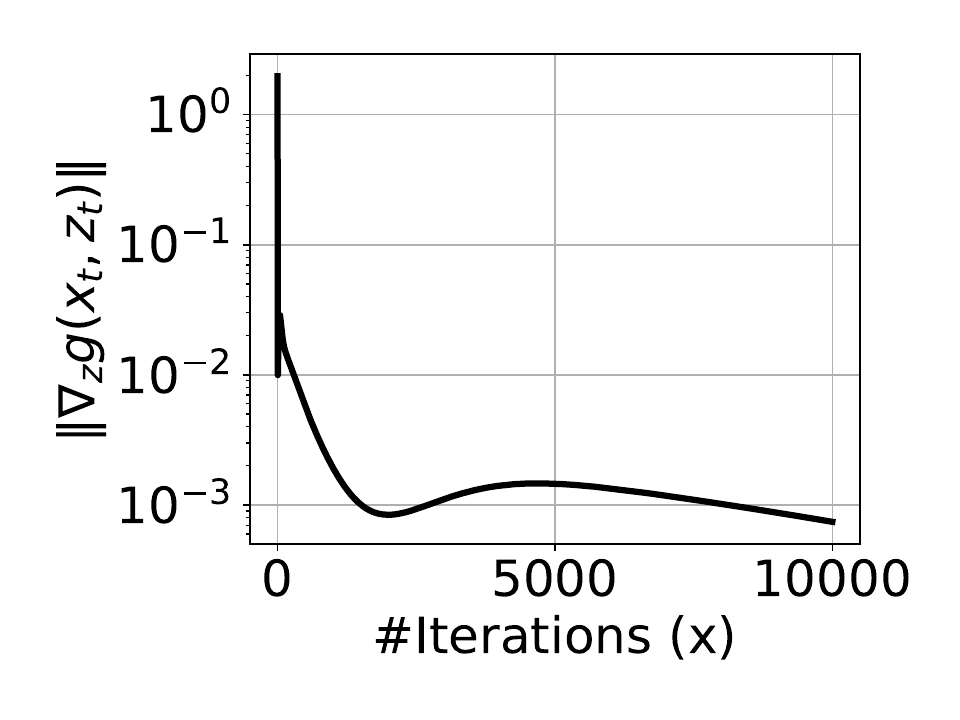} \\
(a) Convergence of $x$     &  (b) Convergence of $y$ & (c) Convergence of $z$
\end{tabular}
\caption{Numerical verification of F${}^2$BA on Problem (\ref{eq:simple-l2reg}).}
\end{figure*}

{\subsection{Tuning 100,000 Regularizers on Logistic Regression}

We then compare our deterministic methods on the ``learnable regularization'' problem of  logistic regression. The aim is to find the optimal regularizer for each feature separately. The corresponding bilevel problem formulation is:
\begin{align} \label{eq:l2reg}
\begin{split}
    &\min_{x \in \BR^p}  \left\{ \frac{1}{\vert \fD^{\rm val} \vert} \sum_{(a_i,b_i) \in \fD^{\rm val}} \ell(\, \langle a_i,y^*(x) \rangle\,,b_i) \right\}, \\
    \text{where } & y^*(x) = \arg \min_{y \in \BR^p}  \left\{ \frac{1}{\vert \fD^{\rm tr} \vert} \sum_{(a_i,b_i) \in \fD^{\rm tr}} 
\ell \left( \,\langle a_i,y \rangle\,,b_i \right) + y^\top \Sigma(x) \,y \right\},
\end{split}
\end{align}
where $\ell(\,\cdot\,, \,\cdot\,)$ is the cross-entropy loss and $\Sigma(x) := {\rm diag}(\exp(x))$ determines the regularizers on each feature. We conduct the experiments on the ``20 newsgroup'' dataset which is commonly used by previous works ~\citep{grazzi2020iteration,ji2021bilevel}. This dataset 
is a collection of 18,000 newsgroup documents from 20 different newsgroups. Each document is represented by a $p$-dimensional vector ($p= 101,631$), where each dimension represents the TF-IDF (Term Frequency-Inverse Document Frequency) of a word. 

We compare F${}^2$BA (Algorithm \ref{alg:F2BA}) as well as its acceleration AccF${}^2$BA (Algorithm \ref{alg:AccF2BA}) with a HVP-based method AID~\citep{grazzi2020iteration,ji2021bilevel} and recently proposed gradient-based methods including F${}^2$SA~\citep{kwon2023fully}, PBGD~\citep{shen2023penalty}. 
We set the number of inner loops as 10 for all the algorithms, and additionally tune $\lambda$ from $\{10^1, 10^2, 10^3, 10^4\}$ for penalty based methods F${}^2$BA, AccF${}^2$BA, F${}^2$SA, and PBGD. 
For AccF${}^2$BA, we set $\theta = 0.1$ as commonly used in momentum-based gradient methods. Instead of using a fixed $B$ in the restart mechanism of AccF${}^2$BA that may lead to very frequent restarts, we follow the practical implementation suggestion by \citet{li2023restarted} to use $B= \max\{B,B_0\}$ with $B_0$ decaying by a constant factor $\gamma$ after each restart. According to Theorem 2 \citep{li2023restarted}, this modified algorithm would have the same convergence rate as the original algorithm except for an additional $\fO(\log (\nicefrac{B_0}{\epsilon}))$ factor. We take $\gamma = 0.99$ in our experiments.

We present the results in Figure \ref{fig:l2reg}, where the dashed line (labeled with ``w/o Reg'') represents the result without tuning any regularizers. 
It can be seen from Figure \ref{fig:l2reg} that our proposed F${}^2$BA outperforms all existing baselines, and can ever achieve a better performance after acceleration (AccF${}^2$BA). 
It is also interesting that the HVP-based method AID performs badly in this experiment. We think the reason is that $\nabla_{yy}^2 g(x,y)$ in our problem can be very close to singular, which makes it very difficult to calculate its inverse numerically. In contrast, penalty-based methods do not suffer from the numerical instability.}

\begin{figure*}
    \centering
    \includegraphics[scale= 0.3]{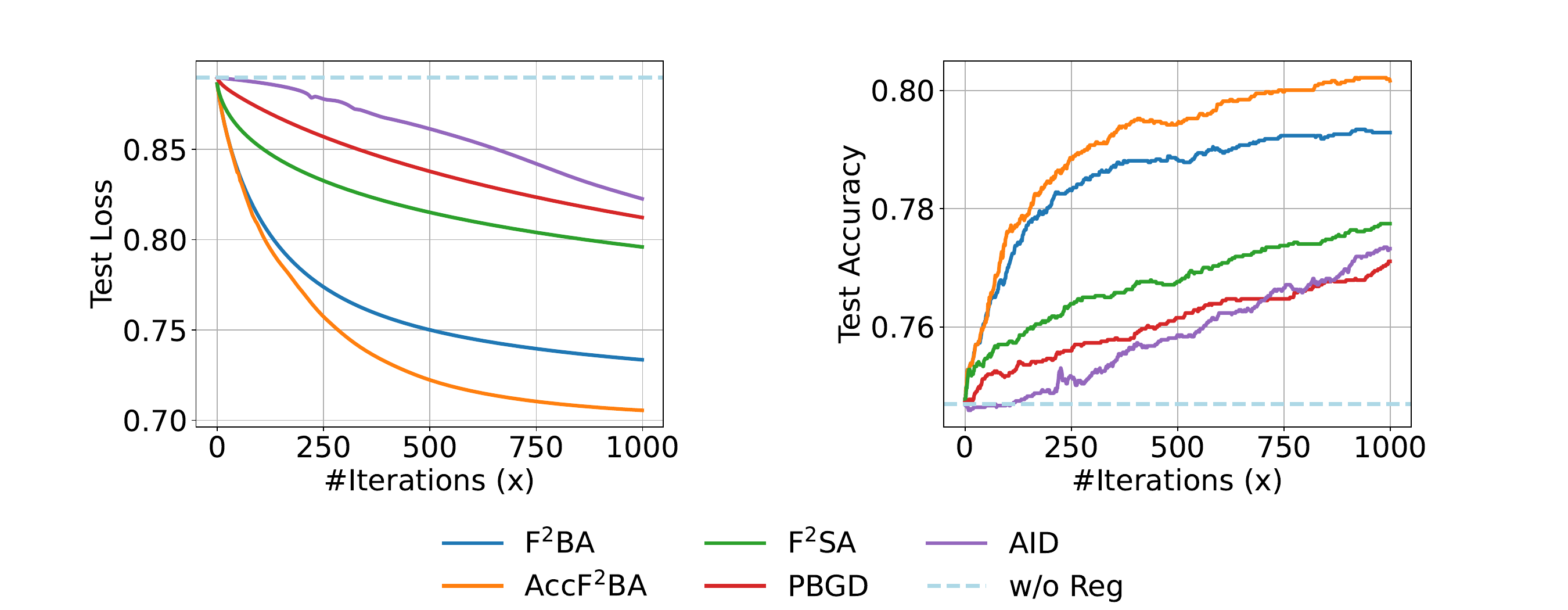} 
\caption{Comparison of deterministic algorithms on Problem (\ref{eq:l2reg}). Our proposed algorithms F${}^2$BA as well as its acceleration AccF${}^2$BA outperform baselines.}
\label{fig:l2reg}
\end{figure*}


{\subsection{Data Hyper-Cleaning for GPT-2} \label{sec:exp-GPT2}

\begin{figure*}[t]
\centering
\begin{tabular}{c c c }\includegraphics[scale=0.28]{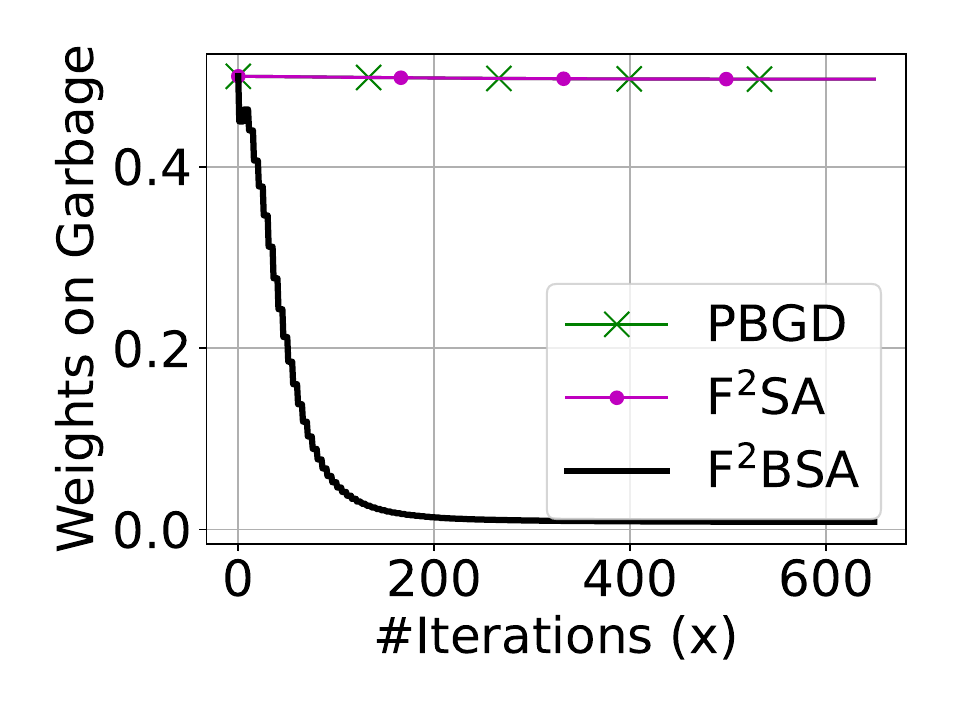} & \includegraphics[scale=0.28]{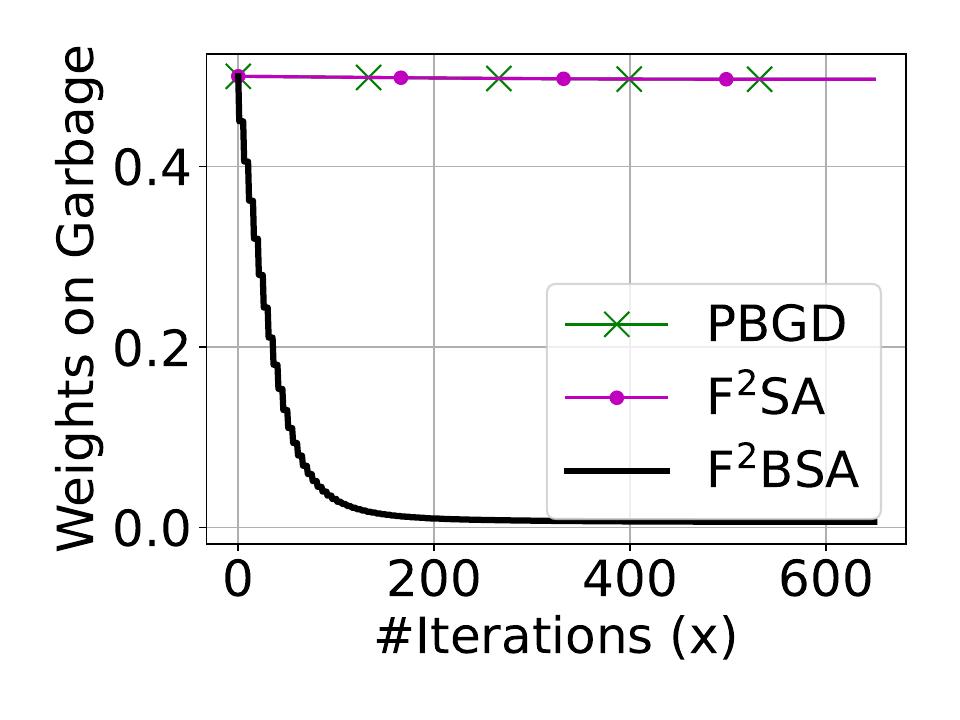} & \includegraphics[scale=0.28]{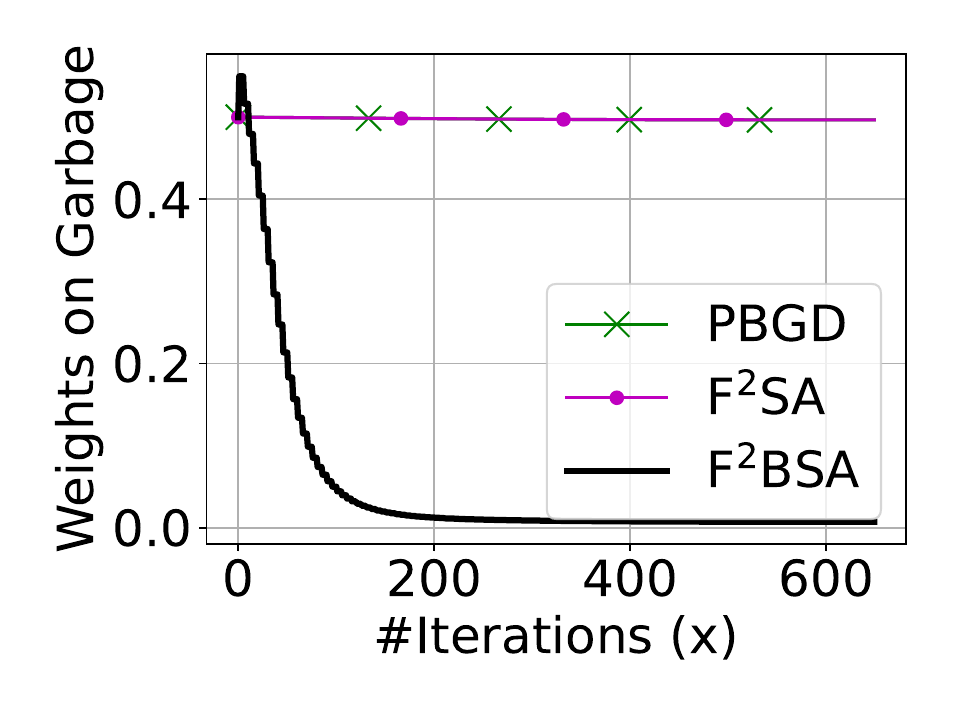} \\
(a)  $p=0.5$     &  (b) $p=0.9$ & (c) $p=0.99$
\label{fig:cleaning}
\end{tabular}
\caption{Comparison of stochastic algorithms on Problem (\ref{eq:cleaning}) under different corruption ratios $p$. Due to the use of two-time-scale step size, our proposed F${}^2$BSA significantly outperforms single-time-scale baselines PBGD and F${}^2$SA.}
\end{figure*}

To demonstrate the scalability of our proposed method, we consider 
a large-scale data hyper-cleaning problem similar to the setting in \citep{pan2024scalebio}. Let $y \in \BR^p$ be the parameters of a neural network, which in our experiment is a GPT-2 model with 124M parameters~\citep{radford2019language}.
The model is trained on a dataset from multiple sources. The loss on each individual data source is denoted as $\ell_{\rm tr}^i(\,\cdot\,)$, where $ i=1,\cdots,m$ and $m$ is the total number of data sources. The goal is to find the optimal weights of different data sources (parameterized by $x \in \BR^m$), such that the trained model can have low validation loss $\ell_{\rm val}(\,\cdot\,)$
This problem can be formulated as a bilevel optimization problem with upper-level and lower-level functions given by:
\begin{align}
\begin{split} \label{eq:cleaning}
    f(x,y) &:= \ell_{\rm val}(y), \\
    g(x,y) &:=  \sum_{i=1}^m \sigma(x_i) \ell_{\rm tr}^i(y).
\end{split}
\end{align}
where $\sigma(x_i)$ is the Softmax function such that $\sigma(x_i) = {\exp(x_i)}/{\sum_{j=1}^m \exp(x_j)}$. 

In our experiment, we use the ``Alpaca'' dataset~\citep{alpaca}, which contains 52k  instructions and demonstration outputs generated by OpenAI's ``text-davinci-003'' engine. 
We split the dataset into a training set and a validation set in an $8:2$ ratio, and then corrupt the training set with a proportion of $p$. The corruption is done by replacing the demonstration outputs with an empty string ``''. This naturally divides the dataset into two sources: useful data and garbage data. We expect the algorithm can learn to assign zero weights to the garbage data and use it to measure the performance of the algorithm. 

We run the experiments on 8 $\times$ A40 GPUs. We compared our proposed stochastic method F${}^2$BSA with related works F${}^2$SA~\citep{kwon2023fully} and PBGD~\citep{shen2023penalty}.
We do not compare any HVP-based methods as we encounter difficulties in obtaining HVP oracles in multi-GPU training (also see Appendix \ref{apx:dist}).
Although the original PBGD~\citep{shen2023penalty} only focuses on the deterministic case, we also implemented a stochastic version of PBGD. For all the algorithms, we set batch size as $64$, and a fixed penalty with $\lambda = 10^3$. The results of different corruption ratios $p$ is shown in Figure \ref{fig:cleaning}. It can be seen from the figure that our proposed F${}^2$BSA converges much faster than baselines F${}^2$SA~\citep{kwon2023fully} and PBGD~\citep{shen2023penalty}. The reason is that the unreasonable choice of single-time-scale step sizes in baselines (the setting of $\zeta=1$ in (5) \citep{kwon2023fully} and the joint update of $(x,y)$ in Line 4, Algorithm 1~\citep{shen2023penalty}) may 
significantly slow down the optimization process.
This also reaffirms the importance of using the two-time-scale step size. Therefore, whether from a theoretical or a practical standpoint, we recommend always using the two-time-scale approach for penalty methods.}

\section{Conclusions and Future Directions}

This paper proposes a {two-time-scale} gradient-based method which achieves the the near-optimal complexity in nonconvex-strongly-convex bilevel optimization. Our result closes the gap between gradient-based and HVP-based methods. {We also study stochastic extension of this algorithm, its acceleration, and its ability to escape saddle points. The algorithms we proposed improve the current best-known guarantees under various settings.} We conclude this paper with 
several potential directions for future research.
\paragraph{Logarithmic Factors in the Deterministic Case.}
The complexity of F${}^2$BA has an additional $\log(\nicefrac{1}{\epsilon})$ factor compared with the lower bound for finding first-order stationary points~\citep{carmon2020lower}. 
Although this factor can be shaved for HVP-based methods using tighter analysis~\citep{ji2021bilevel}, it remains open whether it is avoidable for penalty methods considered in this paper.

{\paragraph{Optimality in the Stochastic Case.} The $\tilde \fO(\epsilon^{-6})$ upper bound by F${}^2$BSA improves the previous $\tilde \fO(\epsilon^{-7})$ result by \citet{kwon2023fully}. However, it is unknown whether this complexity is optimal for gradient-based methods. Recently, \citet{kwoncomplexity} proved an $\Omega(\epsilon^{-6})$ lower bound for stochastic optimization with an $\fO(\epsilon)$-$y^*$-aware oracle, which is related to but different from stochastic bilevel problems. Tight lower bounds for stochastic bilevel problems remains open, as also pointed out by \citet{kwoncomplexity}.
}


\paragraph{Constrained Bilevel Problems.} In this work, we only consider the unconstrained case, \textit{i.e.} we have $x \in \BR^{d_x}$ and $y \in \BR^{d_y}$. It would also be important to 
study the convergence of gradient-based algorithms for constrained problems~\citep{khanduri2023linearly,tsaknakis2022implicit,xiao2023alternating,kwon2024penalty,kornowski2024first}.

\paragraph{Single-Loop Methods.} 
Our method adopts a double-loop structure. Designing single-loop methods could be more efficient in certain scenarios, but the analysis would also be more challenging~\citep{hong2023two,khanduri2021near,chen2022single,chen2021closing,chen2023optimal,dagreou2022framework}. We also leave it in future works.

\section*{Acknowledgments}

Lesi Chen thanks Jeongyeol Kwon for helpful discussions.
Jingzhao Zhang is supported by National Key R\&D Program of China 2024YFA1015800 and Shanghai Qi Zhi Institute Innovation Program. 


\appendix

\newpage

\section{Notations for Tensors and Derivatives} \label{apx:tensor}
We follow the notations of tensors used in \cite{kolda2009tensor}.
For a three-way tensor $\fX \in \BR^{d_1 \times d_2 \times d_3}$,
we use $ [\fX]_{i_1,i_2,i_3}$ to represent its $(i_1,i_2,i_3)$-th element.
The inner product of two three-way tensors $\fX,\fY$ is defined by $\langle \fX, \fY \rangle := \sum_{i_1,i_2,i_3} [\fX]_{i_1,i_2,i_3} \cdot [\fY]_{i_1,i_2,i_3}$. The operator norm of three-way tensor $\fX \in \BR^{d_1 \times d_2 \times d_3}$ is defined by $\Vert \fX \Vert:= \sup_{\Vert x_i \Vert =1} \langle \fX, x_1 \circ x_2 \circ x_3 \rangle $, where the elements of  $x_1 \circ x_2 \circ x_3 \in \BR^{d_1 \times d_2 \times d_3}$ is $ [x_1 \circ x_2 \circ x_3]_{i_1,i_2,i_3} := [x_1]_{i_1} \cdot  [x_2]_{i_2} \cdot [x_3]_{i_3} $. 
It can be verified that such a definition generalizes the Euclidean norm of a vector and the spectral norm of a matrix to tensors.
For a three-way tensor  $ \fX \in \BR^{d_1 \times d_2 \times d_3}$ and a vector $v \in \BR^{d_1}$, their mode-1 product, denoted by $\fX \bar \times_1 v$, gives a matrix in $\BR^{d_2 \times d_3}$ with elements $ [\fX \bar \times_1 v]_{i_2,i_3} := \sum_{i_1} [\fX]_{i_1,i_2,i_3} \cdot [v]_{i_1} $.
We define $\bar \times_2$ and $\bar \times_3$ in a similar way, and it can  be verified that 
$ \Vert \fX \bar \times_i v \Vert \le \Vert \fX \Vert \Vert v \Vert$.
For a three-way tensor $\fX \in \BR^{d_1 \times d_2 \times d_3}$ and a matrix $A \in \BR^{d_1' \times d_1}$, their mode-1 product, denoted by $ \fX \times_1 A$, gives a tensor in $\BR^{d_1' \times d_2 \times d_3}$ with elements  $ [\fX \times_1 A]_{i_1',i_2,i_3} := \sum_{i_1} [\fX]_{i_1,i_2,i_3} \cdot [A]_{i_1',i_1} $. We define $ \times_2$ and $ \times_3$ in a similar way, and it can also be verified that 
$ \Vert \fX \times_i A \Vert \le \Vert \fX \Vert \Vert A \Vert$.

For a function $h(x,y) : \BR^{d_x} \times \BR^{d_y} \rightarrow \BR$,
we denote $\nabla_x h(x,y) \in \BR^{d_x}$ to be the partial gradient with respect to $x$, with elements given by $ [\nabla_x h(x,y) ]_{i} := \partial f(x,y) / \partial [x]_{i}$. And we define $\nabla_y h(x,y) \in \BR^{d_y}$ in a similar way.  We denote $ \nabla_{xx}^2 h(x,y) \in \BR^{d_x} \times \BR^{d_x}$ to be the partial Hessian with respect to $x$, with elements given by $ [\nabla_{xx}^2 h(x,y)]_{i_1,i_2} := \partial^2 h(x,y) / (\partial [x]_{i_1} \partial [x]_{i_2})$. And we define $ \nabla_{xy}^2 h(x,y) \in \BR^{d_x} \times \BR^{d_y} $, $\nabla_{yx}^2 h(x,y) \in \BR^{d_y} \times \BR^{d_x}$ and  $ \nabla_{yy}^2 h(x,y) \in \BR^{d_y} \times \BR^{d_y}$ in a similar way.  We denote $\nabla_{xxx}^3 h(x,y) \in \BR^{d_x} \times \BR^{d_x} \times \BR^{d_x}$ to be the partial third-order derivative with respect to $x$, with elements given by 
$ [\nabla_{xxx}^3 h(x,y)]_{i_1,i_2,i_3} := \partial^3 h(x,y) / (\partial [x]_{i_1} \partial [x]_{i_2} \partial [x]_{i_3})$. And we define $\nabla_{xxy}^3 h(x,y) \in \BR^{d_x} \times \BR^{d_x} \times \BR^{d_y}$, $\nabla_{xyx}^3 h(x,y) \in \BR^{d_x} \times \BR^{d_y} \times \BR^{d_x}$,
$\nabla_{xyy}^3 h(x,y) \in \BR^{d_x} \times \BR^{d_y} \times \BR^{d_y}$,
$\nabla_{yyy}^3 h(x,y) \in \BR^{d_y} \times \BR^{d_y} \times \BR^{d_y}$,
$\nabla_{yyx}^3 h(x,y) \in \BR^{d_y} \times \BR^{d_y} \times \BR^{d_x}$,
$\nabla_{yxy}^3 h(x,y) \in \BR^{d_y} \times \BR^{d_x} \times \BR^{d_y}$
and$\nabla_{yxx}^3 h(x,y) \in \BR^{d_y} \times \BR^{d_x} \times \BR^{d_x}$
in a similar way.
We denote $\nabla y^*(x) \in \BR^{d_x} \times \BR^{d_y} $ to be the matrix with elements given by $[\nabla y^*(x) ]_{i_1,i_2} :=  \partial [y^*(x)]_{i_2}/\partial [x]_{i_1}$. We denote $\nabla^2 y^*(x) \in \BR^{d_x} \times \BR^{d_x} \times \BR^{d_y} $ to be the three-way tensor with elements given by $[\nabla^2 y^*(x)]_{i_1,i_2,i_3}: = \partial^2 [y^*(x)]_{i_3} / (\partial [x]_{i_1} \partial [x]_{i_2})$. And we define $\nabla y_{\lambda}^*(x)$ and $\nabla^2 y_{\lambda}^*(x)$ in a similar way.

\section{Lemmas for Finding First-Order Stationarity} \label{apx:lem-1st}

\begin{lem}[\citet{kwon2023fully}] \label{lem:Lag-sc}
Under Assumption \ref{asm:sc}, for $\lambda \ge 2L_f /\mu$,   $\fL_{\lambda}(x,\,\cdot\,)$ is  $(\lambda \mu/2)$-strongly convex.  
\end{lem}


\begin{lem} \label{lem:yy-lambda}
    Under Assumption \ref{asm:sc}, for $\lambda \ge 2L_f/ \mu$, it holds that 
    \begin{align*}
        \Vert y_{\lambda}^*(x) - y^*(x) \Vert \le \frac{C_f}{\lambda \mu}
    \end{align*}
\end{lem}

\begin{proof}
By the first-order optimality condition, we know that
\begin{align*}
    \nabla_y f(x, y_{\lambda}^*(x)) + \lambda \nabla_y g(x,y_{\lambda}^*(x))  = 0.
\end{align*}
Then we have
\begin{align*}
    \Vert y_{\lambda}^*(x) - y^*(x) \Vert 
    \le \frac{1}{\mu} \Vert \nabla_y g(x,y_{\lambda}^*(x)) \Vert = \frac{1}{\lambda \mu} \Vert \nabla_y f(x,y_{\lambda}^*(x)) \Vert \le \frac{C_f}{\lambda \mu}
\end{align*}
\end{proof}

As a direct consequence, we can show that $\fL_{x}^*$ and $\varphi(x)$ are close.
\begin{lem} \label{lem:F-Lag-close}
 Under Assumption \ref{asm:sc}, for $\lambda \ge 2L_f / \mu$, it holds that $\vert \fL_{\lambda}^*(x) - \varphi(x) \vert \le D_0/ \lambda$, where
 \begin{align*}
     D_0 : = \left(C_f + \frac{C_f L_g}{2 \mu} \right) \frac{C_f }{\mu} = \fO(\ell \kappa^2).
 \end{align*}
\end{lem}

\begin{proof}
A simple calculus shows that
    \begin{align*}
    &\quad \vert \fL_{\lambda}^*(x) - \varphi(x) \vert \\
    &\le   \vert f(x,y_{\lambda}^*(x)) - f(x,y^*(x)) \vert + \lambda \vert g(x,y_{\lambda}^*(x)) - g(x,y^*(x)) \vert \\
    &\le C_f  \Vert y_{\lambda}^*(x) - y^*(x) \Vert + \frac{\lambda L_g}{2} \Vert y_{\lambda}^*(x) - y^*(x) \Vert^2 \\
    &\le \left(C_f + \frac{C_f L_g}{2 \mu} \right) \Vert y_{\lambda}^*(x) - y^*(x) \Vert \\
    &\le \left(C_f + \frac{C_f L_g}{2 \mu} \right) \frac{C_f }{\lambda \mu}.
    \end{align*}
\end{proof}
The following result is the key to designing fully first-order methods for bilevel optimization, first proved in Lemma 3.1 in \cite{kwon2023fully}. We provide the proof here for completeness.
\begin{lem} \label{lem:key-fully}
Under Assumption \ref{asm:sc}, for $\lambda \ge 2L_f/ \mu$,  it holds that $\Vert \nabla \fL_{\lambda}^*(x) - \nabla \varphi(x) \Vert \le D_1 / \lambda$, where
\begin{align} \label{eq:D-lambda}
    D_1 :=  \left( L_f 
   + \frac{\rho_g L_g}{\mu} + \frac{C_f L_g \rho_g}{2\mu^2} + \frac{C_f \rho_g}{2\mu} 
    \right) \frac{C_f}{\mu} = \fO(\ell \kappa^3).
\end{align}
\end{lem}

\begin{proof}
Taking total derivative on $\varphi(x) = f(x,y^*(x))$, we obtain the following result:
\begin{align} \label{eq:hyper-grad}
    \nabla \varphi(x) = \nabla_x f(x,y^*(x)) - \nabla_{xy}^2 g(x,y^*(x)) [ \nabla_{yy} g(x,y^*(x))]^{-1} \nabla_y f(x,y^*(x)). 
\end{align}
Also, we know that
\begin{align*}
    \nabla \fL_{\lambda}^*(x) = \nabla_x f(x,y_{\lambda}^*(x)) + \lambda (\nabla_x g(x,y_{\lambda}^*(x)) - \nabla_x g(x,y^*(x))).
\end{align*}
By simple calculus, we have
\begin{align*}
    &\quad \nabla \fL_{\lambda}^*(x) - \nabla \varphi(x) \\
    &= \nabla_x f(x,y_{\lambda}^*(x)) - \nabla_x f(x,y^*(x)) \\
    &\quad + \nabla_{xy}^2 g(x,y^*(x)) [ \nabla_{yy} g(x,y^*(x))]^{-1} (\nabla_y f(x,y^*(x)) - \nabla_y f(x,y_{\lambda}^*(x)) \\
    &\quad + \lambda \nabla_{xy}^2 g(x,y^*(x)) [ \nabla_{yy} g(x,y^*(x))]^{-1} \times \\
    &~~~~~~~~~~~~~~~( \nabla_{yy}^2 g(x,y^*(x)) (y_{\lambda}^*(x) - y^*(x)) + \nabla_y g(x,y^*(x)) - \nabla_y g(x,y_{\lambda}^*(x)) ) \\
    &\quad + \lambda ( \nabla_x g(x,y_{\lambda}^*(x)) - \nabla_x  g(x,y^*(x)) - \nabla_{xy}^2 g(x,y^*(x)) (y_{\lambda}^*(x) - y^*(x)) ).
\end{align*}
Taking norm on both sides,
\begin{align*}
    \Vert \nabla \fL_{\lambda}^*(x) - \nabla \varphi(x) \Vert &\le L_f \Vert y_{\lambda}^*(x) - y^*(x) \Vert + \frac{\rho_g L_g}{\mu} \Vert y_{\lambda}^*(x) - y^*(x) \Vert  \\
    &\quad + \frac{\lambda L_g \rho_g}{2\mu} \Vert y_{\lambda}^*(x) - y^*(x) \Vert^2  + \frac{\lambda \rho_g}{2} \Vert y_{\lambda}^*(x) - y^*(x) \Vert^2 \\
    &\le \left( L_f 
   + \frac{\rho_g L_g}{\mu} + \frac{C_f L_g \rho_g}{2\mu^2} + \frac{C_f \rho_g}{2\mu} 
    \right) \Vert y_{\lambda}^*(x) - y^*(x) \Vert \\
    &\le \left( L_f 
   + \frac{\rho_g L_g}{\mu} + \frac{C_f L_g \rho_g}{2\mu^2} + \frac{C_f \rho_g}{2\mu} 
    \right) \frac{C_f}{\lambda \mu}
\end{align*}

\end{proof}

\begin{lem} \label{lem:yxyx-lambda}
Under Assumption \ref{asm:sc}, for $\lambda \ge 2 L_f / \mu$, it holds that $ \Vert \nabla y^*(x) - \nabla y_{\lambda}^*(x) \Vert \le D_2/\lambda $, where
\begin{align*}
    D_2 := \left( \frac{1}{\mu} + \frac{2L_g}{\mu^2} \right) \left( L_f + \frac{C_f \rho_g}{\mu} \right) = \fO(\kappa^3).
\end{align*}

\end{lem}

\begin{proof}
 Taking derivative on both sides of $\nabla_y g (x,y^*(x)) = 0$ yields
 \begin{align} \label{eq:yxyy}
     \nabla_{xy}^2 g(x,y^*(x)) +  \nabla y^*(x)\nabla_{yy}^2 g(x,y^*(x)) = 0.
 \end{align}
 Rearranging, we can obtain
\begin{align} \label{eq:y-x}
  \nabla y^*(x) = - \nabla_{xy}^2 g(x,y^*(x)) [\nabla_{yy}^2 g(x,y^*(x)) ]^{-1} . 
  \end{align}
Similarly, we also have
\begin{align} \label{eq:y-lambda-x}
    \nabla y_{\lambda}^*(x) = -  \nabla_{xy}^2 \fL_{\lambda}(x,y_{\lambda}^*(x)) [\nabla_{yy}^2 \fL_{\lambda}(x,y_{\lambda}^*(x)) ]^{-1}.
\end{align}
Using the matrix identity $ A^{-1} - B^{-1} = A^{-1} (B -A) B^{-1}$, we have
\begin{align} \label{eq:diff-inv}
    \begin{split}
    &\quad  \left \Vert [\nabla_{yy}^2 g(x,y^*(x))]^{-1} - \left [  \frac{\nabla_{yy}^2 \fL_{\lambda}(x,y_{\lambda}^*(x))}{\lambda} \right]^{-1}  \right \Vert \\
    &\le  \left \Vert [\nabla_{yy}^2 g(x,y_{\lambda}^*(x))]^{-1} \right \Vert  
    \left \Vert  \frac{\nabla_{yy}^2 \fL_{\lambda}(x,y_{\lambda}^*(x))}{\lambda} -   \nabla_{yy}^2 g(x,y^*(x))  \right \Vert \left \Vert \left[\frac{\nabla_{yy}^2 \fL_{\lambda} (x,y_{\lambda}^*(x))}{\lambda}\right]^{-1}  \right \Vert \\
    &\le \frac{2}{\mu^2} \left \Vert 
    \frac{\nabla_{yy}^2 f(x,y_{\lambda}^*(x))}{\lambda}
     + \nabla_{yy}^2 g(x,y_{\lambda}^*(x)) - \nabla_{yy}^2 g(x,y^*(x)) \right \Vert \\
     &\le \frac{2}{\mu^2} \left( \frac{L_f}{\lambda} + \rho_g \Vert y_{\lambda}^*(x) - y^*(x) \Vert  \right)\\
     &\le \frac{2}{\lambda \mu^2} \left(L_f + 
     \frac{C_f \rho_g}{\mu}
     \right). 
    \end{split}
\end{align}
Note that the setting of $\lambda \ge 2L_f / \mu  $ implies $\Vert \nabla_{xy}^2 \fL(\,\cdot\,,\,\cdot\,) \Vert \le 2 \lambda L_g$, then we further have
\begin{align*}
    &\quad \Vert \nabla y_{\lambda}^*(x) - \nabla y^*(x) \Vert \\
    & \le\left  \Vert \nabla_{xy}^2 g(x,y^*(x)) - \frac{\nabla_{xy}^2 \fL_{\lambda}(x,y_{\lambda}^*(x))}{ \lambda} \right \Vert \left \Vert [\nabla_{yy}^2 g(x,y^*(x))]^{-1}  \right \Vert  \\
    &\quad + \left \Vert \frac{\nabla_{xy}^2 \fL_{\lambda}(x,y_{\lambda}^*(x))}{\lambda}  \right \Vert  \left \Vert [\nabla_{yy}^2 g(x,y^*(x))]^{-1} - \left [  \frac{\nabla_{yy}^2 \fL_{\lambda}(x,y_{\lambda}^*(x))}{\lambda} \right]^{-1}  \right \Vert  \\
    &\le \frac{1}{\mu} \left \Vert \nabla_{xy}^2 g(x,y^*(x)) - \nabla_{xy}^2 g(x,y_{\lambda}^*(x)) - \frac{\nabla_{xy}^2 f(x,y_{\lambda}^*(x))}{\lambda}  \right \Vert  +  \frac{2L_g}{\lambda \mu^2} \left(L_f + 
     \frac{C_f \rho_g}{\mu}
     \right)  \\
    &\le \left( \frac{1}{\lambda \mu} + \frac{2L_g}{\lambda \mu^2} \right) \left( L_f + \frac{C_f \rho_g}{\mu} \right).
    \end{align*}
\end{proof}

It is clear that $\Vert \nabla y^*(x) \Vert \le L_g/\mu$, therefore $y^*(x)$ is $(L_g/\mu)$-Lipschitz. Below, we show that a similar result also holds for $y_{\lambda}^*(x)$.

\begin{lem} \label{lem:nabla-y-lambda-bound}
Under Assumption \ref{asm:sc}, for $\lambda \ge 2L_f /\mu$, it holds that $\Vert \nabla y_{\lambda}^*(x) \Vert \le 4L_g/\mu$.    
\end{lem}

\begin{proof}
    Recall \Eqref{eq:y-lambda-x} that 
\begin{align*}
    \nabla y_{\lambda}^*(x) = -  \nabla_{xy}^2 \fL_{\lambda}(x,y_{\lambda}^*(x)) [\nabla_{yy}^2 \fL_{\lambda}(x,y_{\lambda}^*(x)) ]^{-1}.
\end{align*}
We arrive at the conclusion by noting 
that $  \nabla_{xy}^2 \fL_{\lambda}(\,\cdot\,,\,\cdot\,) \preceq 2 \lambda L_g$ and  $ \nabla_{yy}^2 \fL_{\lambda}(\,\cdot\, , \,\cdot\,) \succeq \lambda \mu /2 $ by Lemma \ref{lem:Lag-sc}.

\end{proof}
This implies that $y_{\lambda}^*(x)$ is $(4L_g/\mu)$-Lipschitz.

\begin{lem} \label{lem:nabla2-bound}
Under Assumption \ref{asm:sc}, for $\lambda \ge 2L_f /\mu$, $ \nabla \fL_{\lambda}^*(x)$ is $D_3$-Lipschitz, where
\begin{align*}
    D_3 &:= L_f + \frac{4L_f L_g}{\mu} +\frac{C_f \rho_g}{\mu} +  \frac{C_f L_g \rho_g}{\mu^2} + L_g D_2 = \fO(\kappa^3).
\end{align*}
\end{lem}

\begin{proof}
Note that
\begin{align} \label{eq:nabla-AB}
    \nabla \fL_{\lambda}^*(x) = \underbrace{\nabla_x f(x,y_{\lambda}^*(x))}_{A(x)} + \underbrace{\lambda ( \nabla_x g(x,y_{\lambda}^*(x)) - \nabla_x g(x,y^*(x)))}_{B(x)}. 
\end{align}
where $A(x)$ and $B(x)$ are both mappings  $\BR^{d_x} \rightarrow \BR^{d_x}$. 
From Lemma \ref{lem:nabla-y-lambda-bound} we know that $y_{\lambda}^*(x)$ is $(4L_g/\mu)$-Lipschitz. This further implies that $A(x)$ is  $(1+4L_g/\mu)L_f$-Lipschitz. Next, we bound the Lipschitz  coefficient of $B(x)$ via its derivative $\nabla B(x): \BR^{d_x} \rightarrow \BR^{d_x} \times \BR^{d_x} $, which has the following form by 
taking total derivative on $B(x)$:
\begin{align} \label{eq:nabla-Bx}
\begin{split}
    \nabla B(x) &= 
    \lambda (\nabla_{xx}^2 g(x,y_{\lambda}^*(x))-\nabla_{xx}^2 g(x,y^*(x))) \\
    &\quad + \lambda \left(\nabla y_{\lambda}^*(x)  \nabla_{yx}^2 g(x,y_{\lambda}^*(x)) - \nabla y^*(x) \nabla_{yx}^2 g(x,y^*(x)) \right).
\end{split}
\end{align}
And we can bound the operator norm of $\nabla B(x)$ by:
\begin{align*}
    \Vert \nabla B(x) \Vert &\le  \lambda \Vert \nabla_{xx}^2 g(x,y_{\lambda}^*(x)) - \nabla_{xx}^2 g(x,y^*(x)) \Vert \\
    &\quad + \lambda \Vert \nabla y^*(x) \Vert \Vert \nabla_{yx}^2 g(x,y_{\lambda}^*(x)) - \nabla_{yx}^2 g(x,y^*(x)) \Vert \\
    &\quad  + \lambda \Vert \nabla y_{\lambda}^*(x) - \nabla y^*(x) \Vert \Vert \nabla_{yx}^2 g(x,y_{\lambda}^*(x)) \Vert\\
    &\le  \lambda \rho_g \left(1+ \frac{L_g}{\mu} \right) \Vert y_{\lambda}^*(x) - y^*(x) \Vert +\lambda L_g \Vert \nabla y_{\lambda}^*(x) - \nabla y^*(x) \Vert  \\
    &\le  \left(1+ \frac{L_g}{\mu} \right) \frac{C_f \rho_g}{\mu} +L_g D_2,
\end{align*}
where we use Lemma \ref{lem:nabla-y-lambda-bound}
in the second inequality; Lemma \ref{lem:yy-lambda} and \ref{lem:yxyx-lambda} in the third one.
\end{proof}

\section{Lemmas for Finding Second-Order Stationarity} \label{apx:lem-2nd}

\begin{lem} \label{lem:nabla2-yx}
Under Assumption \ref{asm:sc} and \ref{asm:third}, 
for $\lambda \ge  2L_f / \mu$,
we have $ \Vert \nabla^2 y^*(x) -\nabla^2 y_{\lambda}^*(x)  \Vert \le D_4 /\lambda$, where 
\begin{align*}
    D_4 := \frac{2\rho_g}{\mu^2}  \left( 1 + \frac{L_g}{\mu}\right)^2  \left(L_f + 
     \frac{C_f \rho_g}{\mu}
     \right) +   \frac{14 L_g \rho_g D_2}{\mu^2} + \frac{50 L_g^2 }{\mu^3} \left( \frac{C_f \nu_g}{\mu} + \rho_f \right) = \fO(\kappa^5).
\end{align*}
\end{lem}

\begin{proof}
{First of all, we calculate the explicit form of $\nabla^2 y^*(x)$ and $\nabla^2 y_{\lambda}^*(x)$.}

    By taking the derivative with respect to $x$ on
    \begin{align*}
        \nabla_{xy}^2 g(x,y^*(x)) + \nabla y^*(x) \nabla_{yy}^2 g(x,y^*(x)) =0,
    \end{align*}
    we obtain
    \begin{align*}
        & \nabla_{xxy}^3 g(x,y^*(x)) + \nabla_{yxy}^3 g(x,y^*(x)) \times_{1}  \nabla y^*(x) + \nabla^2 y^*(x) \times_3 \nabla_{yy}^2 g(x,y^*(x)) \\
        &\quad + \nabla_{xyy}^3 g(x,y^*(x)) \times_2 \nabla y^*(x)  +\nabla_{yyy}^3 g(x,y^*(x)) \times_1 \nabla y^*(x) \times_2 \nabla y^*(x) = 0.
    \end{align*}
    Rearranging to get
    {\begin{align} \label{eq:nabla2-yx}
    \begin{split}
         \nabla^2 y^*(x) 
        &= -\left( \nabla_{xxy}^3 g(x,y^*(x)) + \nabla_{yxy}^3 g(x,y^*(x)) \times_1 \nabla y^*(x) 
        \right)  \times_3 [\nabla_{yy}^2 g(x,y^*(x))]^{-1} \\
        &\quad -
        \nabla_{xyy}^3 g(x,y^*(x)) \times_2 \nabla y^*(x) \times_3 [\nabla_{yy}^2 g(x,y^*(x))]^{-1} \\
        &\quad-  \nabla_{yyy}^3 g(x,y^*(x)) \times_1 \nabla y^*(x) \times_2 \nabla y^*(x)    \times_3 [\nabla_{yy}^2 g(x,y^*(x))]^{-1}.
    \end{split}
    \end{align}}
    Similarly,
    {
    \begin{align} \label{eq:nabla2-yx-lambda}
    \begin{split}
        \nabla^2 y_{\lambda}^*(x)  &= -\left( \nabla_{xxy}^3 \fL_{\lambda} (x, y_{\lambda}^*(x)) + \nabla_{yxy}^3 \fL_\lambda(x, y_{\lambda}^*(x)) \times_1 \nabla y_{\lambda}^*(x)  \right) \times_3
        [ \nabla_{yy}^2 \fL_{\lambda}(x,y_{\lambda}^*(x))  ]^{-1} \\
        &\quad - 
        \nabla_{xyy}^3 \fL_{\lambda}(x,y_{\lambda}^*(x))
        \times_2 \times_3 [ \nabla_{yy}^2 \fL_{\lambda}(x,y_{\lambda}^*(x))]^{-1} \\
        &\quad - 
        \nabla y_{\lambda}^*(x)   +  
         \nabla_{yyy}^3 \fL_{\lambda} (x,y_{\lambda}^*(x))
        \times_1 \nabla y_{\lambda}^*(x)
        \times_2 \nabla y_{\lambda}^*(x) 
         \times_3 [ \nabla_{yy}^2 \fL_{\lambda}(x,y_{\lambda}^*(x))]^{-1}.
    \end{split}
    \end{align}}
{We then prove $\Vert \nabla^2 y^*(x) - \nabla^2 y_{\lambda}^*(x) \Vert = \fO(1/\lambda)$. We prove this by showing that the difference between each corresponding term of $\nabla^2 y^*(x)$ and $\nabla^2 y_{\lambda}^*(x)$ is $\fO(1/\lambda)$.}
Note that 
\begin{align*}
    \left \Vert  \nabla_{xxy}^3  g(x,y^*(x)) - \frac{\nabla_{xxy}^3 \fL_{\lambda} (x,y_{\lambda}^*(x))  }{\lambda}   \right \Vert \le \nu_g \Vert y_{\lambda}^*(x) - y^*(x) \Vert + \frac{\rho_f}{\lambda} \le \frac{1}{\lambda} \left(\frac{C_f \nu_g}{\mu} +\rho_f \right),
\end{align*}
and
\begin{align*}
    &\quad \left \Vert 
    \nabla_{yxy}^3 g(x,y^*(x)) \times_1 \nabla y^*(x)  - \frac{\nabla_{yxy}^3 \fL_{\lambda} (x,y_{\lambda}^*(x)) \times_1 \nabla y_{\lambda}^*(x) }{\lambda}
    \right \Vert \\
    &\le \Vert \nabla y^*(x) - \nabla y_{\lambda}^*(x) \Vert \Vert \nabla_{yxy}^3 g(x,y^*(x)) \Vert + \Vert \nabla y_{\lambda}^*(x) \Vert \left\Vert \nabla_{yxy}^3 g(x,y^*(x)) - \frac{\nabla_{yxy}^3 \fL_{\lambda}(x, y_{\lambda}^*(x))}{\lambda} \right \Vert \\
    &\le \frac{\rho_g D_2}{\lambda} + \frac{4L_g}{\lambda \mu} \left( \frac{C_f \nu_g}{\mu} + \rho_f \right),
\end{align*}
and
{ 
\begin{align*}
    &\quad  \left \Vert
   \nabla_{yyy}^3 g(x,y^*(x)) \times_1  \nabla y^*(x)  \times_2 \nabla y^*(x)  - \frac{\nabla_{yyy}^3 \fL_{\lambda} (x,y_{\lambda}^*(x)) \times_1 \nabla y_{\lambda}^*(x) \times_2  \nabla y_{\lambda}^*(x)} { \lambda}
    \right \Vert \\
    &\le \Vert \nabla y^*(x) \Vert \Vert \nabla_{yyy}^3 g(x,y^*(x))\Vert \Vert \nabla y^*(x) - \nabla y_{\lambda}^*(x) \Vert  \\
    &\quad + \Vert \nabla y_{\lambda}^*(x) \Vert \Vert \nabla_{yyy}^3 g(x,y^*(x))\Vert \Vert \nabla y^*(x) - \nabla y_{\lambda}^*(x) \Vert  \\
    &\quad +  \Vert \nabla y_{\lambda}^*(x) \Vert^2  \left \Vert  \nabla_{xxy}^3  g(x,y^*(x)) - \frac{\nabla_{xxy}^3 \fL_{\lambda} (x,y_{\lambda}^*(x))  }{\lambda}   \right \Vert \\
    &\le  \frac{5L_g \rho_g D_2}{\lambda \mu} + \frac{16L_g^2 }{\lambda \mu^2} \left( \frac{C_f \nu_g}{\mu} + \rho_f \right),
\end{align*}}
we can obtain that
{
\begin{align*}
    &\quad \Vert \nabla^2 y^*(x) - \nabla^2 y_{\lambda}^*(x) \Vert  \\
    &\le \rho_g \left( 1 + \frac{L_g}{\mu}\right)^2 \left \Vert [\nabla_{yy}^2 g(x,y^*(x))]^{-1} - \left [  \frac{\nabla_{yy}^2 \fL_{\lambda}(x,y_{\lambda}^*(x))}{\lambda} \right]^{-1}  \right \Vert \\
    &\quad + \left( \frac{7L_g \rho_g D_2}{\lambda \mu} + \frac{25L_g^2 }{\lambda \mu^2} \left( \frac{C_f \nu_g}{\mu} + \rho_f \right) \right) \left \Vert \left [  \frac{\nabla_{yy}^2 \fL_{\lambda}(x,y_{\lambda}^*(x))}{\lambda} \right]^{-1} \right \Vert \\
    &\le \frac{2\rho_g}{\lambda \mu^2}  \left( 1 + \frac{L_g}{\mu}\right)^2  \left(L_f + 
     \frac{C_f \rho_g}{\mu}
     \right) +   \frac{14 L_g \rho_g D_2}{\lambda \mu^2} + \frac{50 L_g^2 }{\lambda \mu^3} \left( \frac{C_f \nu_g}{\mu} + \rho_f \right) ,
\end{align*}}
where we use \Eqref{eq:diff-inv} in the second inequality.
\end{proof}

\begin{lem} \label{lem:nabla3-bound}
Under Assumption \ref{asm:sc} and \ref{asm:third}, for $ \lambda \ge 2 L_f /\mu$, 
$\nabla^2 \fL_{\lambda}^*(x) $ is $D_5$-Lipschitz,
 where
\begin{align*}
    D_5 &:=  \left( 1 + \frac{4L_g}{\mu} \right)^2 \left( 3\rho_f  + \frac{2 L_f \rho_g}{\mu} \right) + \left(1 + \frac{L_g}{\mu} \right)^2 \frac{C_f \nu_g}{\mu} \\
    &\quad + \left(2+ \frac{5L_g}{\mu} \right) D_2 \rho_g +  \left( 1+ \frac{L_g}{\mu} \right)^2 \frac{C_f \rho_g^2}{\mu^2} + L_g D_4   = \fO(\ell \kappa^5).
\end{align*}
\end{lem}

\begin{proof}
Similar to the proof of Lemma \ref{lem:nabla2-bound},  we split $\nabla^2 \fL_{\lambda}^*(x)$ into two terms:
\begin{align*}
    \nabla^2 \fL_{\lambda}^*(x) = \nabla A(x) + \nabla B(x),
\end{align*}
where both the mappings $A(x)$ and $B(x)$ both follow the definitions in \Eqref{eq:nabla-AB}. 
rTaking total derivative on $A(x)$, we obtain
\begin{align*}
    \nabla A(x) = \nabla_{xx}^2 f(x,y_{\lambda}^*(x)) + \nabla y_{\lambda}^*(x) \nabla_{yx}^2 f(x,y_{\lambda}^*(x)).
\end{align*}
And recall \Eqref{eq:nabla-Bx} that 
\begin{align*}
    \nabla B(x) &= \lambda (\nabla_{xx}^2 g(x,y_{\lambda}^*(x))-\nabla_{xx}^2 g(x,y^*(x))) \\
    &\quad + \lambda \left(\nabla y_{\lambda}^*(x)  \nabla_{yx}^2 g(x,y_{\lambda}^*(x)) - \nabla y^*(x) \nabla_{yx}^2 g(x,y^*(x)) \right).
\end{align*}
Then we bound the Lipschitz coefficient of $\nabla A(x)$ and $\nabla B(x)$, respectively. 

From \Eqref{eq:nabla2-yx} we can calculate that 
\begin{align} \label{bound:nabla2-yx}
    \Vert \nabla^2 y^*(x) \Vert \le \left( 1 + \frac{L_g}{\mu} \right)^2 \frac{\rho_g}{\mu}  .
\end{align}
Similarly, from \Eqref{eq:nabla2-yx-lambda} we can calculate that
\begin{align} \label{bound:nabla2-y-lambda}
    \Vert \nabla^2 y_{\lambda}^*(x)  \Vert &\le \left( 1 + \frac{4L_g}{\mu} \right)^2 \left( \frac{\rho_f}{\lambda}+ \rho_g \right) \frac{2}{\mu} \le \left( 1 + \frac{4L_g}{\mu} \right)^2 \left( \frac{2\rho_f}{L_f} + \frac{2 \rho_g}{\mu} \right). 
\end{align}
From Lemma \ref{lem:nabla-y-lambda-bound} we know that $y_{\lambda}^*(x)$ is $(4L_g/\mu)$-Lipschitz. This further implies that both $ \nabla_{xx}^2 f(x,y_{\lambda}^*(x))$ and $\nabla_{yx}^2 f(x,y_{\lambda}^*(x)) $ are $(1+ 4L_g/ \mu)\rho_f$-Lipschitz. Then, for any $x_1,x_2 \in \BR^{d_x}$, we have
\begin{align*}
    &\quad  \Vert \nabla A(x_1) - \nabla A(x_2) \Vert \\
   &\le \Vert  \nabla_{xx}^2 f(x_1,y_{\lambda}^*(x_1)) - \nabla_{xx}^2 f(x_2, y_{\lambda}^*(x_2)) \Vert  \\
   &\quad + \Vert \nabla y_{\lambda}^*(x_1) \nabla_{yx}^2 f(x_1,y_{\lambda}^*(x_1)) - \nabla y_{\lambda}^*(x_2) \nabla_{yx}^2 f(x_1,y_{\lambda}^*(x_1)) \Vert \\
   &\quad + \Vert \nabla y_{\lambda}^*(x_2) \nabla_{yx}^2 f(x_1,y_{\lambda}^*(x_1)) - \nabla y_{\lambda}^*(x_2) \nabla_{yx}^2 f(x_2,y_{\lambda}^*(x_2)) \Vert \\
   &\le \Vert \nabla_{xx}^2 f(x_1,y_{\lambda}^*(x_1)) - \nabla_{xx}^2 f(x_2, y_{\lambda}^*(x_2)) \Vert  \\
   &\quad + \Vert \nabla y_{\lambda}^*(x_1) - \nabla y_{\lambda}^*(x_2) \Vert \Vert \nabla_{yx}^2 f(x_1,y_{\lambda}^*(x_1)) \Vert \\
   &\quad + \Vert \nabla y_{\lambda}^*(x_2) \Vert \Vert \nabla_{yx}^2 f(x_1,y_{\lambda}^*(x_1)) - \nabla_{yx}^2 f(x_2,y_{\lambda}^*(x_2)) \Vert \\
   &\le \underbrace{\left(\left( 1+ \frac{4L_g}{\mu} \right)^2 \rho_f + \left( 1 + \frac{4L_g}{\mu} \right)^2 \left( \frac{2\rho_f}{L_f} + \frac{2 \rho_g}{\mu} \right) L_f \right)}_{C_1}\Vert x_1-x_2 \Vert,
\end{align*}
where $C_1$ gives the upper bound of the Lipschitz coefficient of the mapping $\nabla A(x)$.

To bound the Lipschitz coefficient of $\nabla B(x)$,
we first derive the explicit form of the mapping $\nabla^2 B(x): \BR^{d_x} \rightarrow \BR^{d_x} \times \BR^{d_x} \times \BR^{d_x} $ by:
\begin{align*}
    \nabla^2 B(x) 
    &= 
    \nabla^2 y_{\lambda}^*(x) \times_3 [\nabla_{yx}^2 f(x,y_{\lambda}^*(x))]^\top \\
    &\quad + \lambda ( \nabla_{xxx}^3 g(x,y_{\lambda}^*(x)) - \nabla_{xxx}^3 g(x,y^*(x)) ) \\
    &\quad + \lambda ( 
    \nabla_{yxx}^3 g(x,y_{\lambda}^*(x)) \times_1 \nabla y_{\lambda}^*(x) - \nabla_{yxx}^3 g(x,y^*(x)) \times_1 \nabla y^*(x) 
    ) \\
    &\quad + \lambda \left(  \nabla_{xyx}^3 g(x,y_{\lambda}^*(x)) \times_2 \nabla y_{\lambda}^*(x) - \nabla_{xyx}^3 g(x,y^*(x))   \times_2 \nabla y^*(x) \right)  \\
    &\quad + \lambda \left( 
     \nabla_{yyx}^3 g(x,y_{\lambda}^*(x)) \times_1 \nabla y_{\lambda}^*(x) \times_2  \nabla y_{\lambda}^*(x)   - 
    \nabla_{yyx}^3 g(x,y^*(x)) \times_1 \nabla y^*(x)  \times_2 \nabla y^*(x)
    \right) \\
    &\quad + \lambda \left(  \nabla^2 y_{\lambda}^*(x) \times_3 [\nabla_{yx}^2 g(x,y_{\lambda}^*(x))]^\top -  \nabla^2 y^*(x) \times_3 [\nabla_{yx}^2 g(x,y^*(x))]^\top\right).
\end{align*}
Then we can bound the Lipschitz coefficient of $\nabla B(x)$
via the operator norm of $\nabla^2 B(x)$:  
\begin{align*}
   C_2&:= \Vert \nabla^2 B(x) \Vert \\
    &\le 
    \Vert \nabla_{xxx}^3 g(x,y^*(x)) - 
    \nabla_{xxx}^3 g(x,y_{\lambda}^*(x))\Vert  \\
    &\quad + \lambda \Vert \nabla y^*(x) \Vert \Vert \nabla_{yxx}^3 g(x,y^*(x)) - \nabla_{yxx}^3 g(x,y_{\lambda}^*(x)) \Vert \\
    &\quad+ \lambda \Vert \nabla y_{\lambda}^*(x) - \nabla y^*(x) \Vert \Vert \nabla_{yxx}^3 g(x,y_{\lambda}^*(x)) \Vert \\
    &\quad + \lambda \Vert \nabla y^*(x) \Vert \Vert \nabla_{xyx}^3 g(x,y^*(x)) - \nabla_{xyx}^3 g(x,y_{\lambda}^*(x)) \Vert  \\
    &\quad + \lambda \Vert \nabla y^*(x) - \nabla y_{\lambda}^*(x) \Vert \Vert \nabla_{xyx}^3 g(x,y_{\lambda}^*(x)) \Vert \\
    &\quad + \lambda \Vert \nabla y^*(x) \Vert \Vert \nabla_{yyx}^3 g(x,y^*(x)) \Vert  \Vert \nabla y_{\lambda}^*(x) - \nabla y^*(x) \Vert  \\
    &\quad + \lambda \Vert \nabla y_{\lambda}^*(x) \Vert \Vert \nabla_{yyx}^3 g(x,y^*(x)) \Vert  \Vert \nabla y_{\lambda}^*(x) - \nabla y^*(x) \Vert \\
    &\quad + \lambda \Vert \nabla y^*(x) \Vert^2 \Vert \nabla_{yyx}^3 g(x,y^*(x)) - \nabla_{yyx}^3 g(x,y_{\lambda}^*(x)) \Vert \\
    &\quad + \lambda \Vert \nabla^2 y^*(x) \Vert \Vert \nabla_{yx}^2 g(x,y^*(x)) - \nabla_{yx}^2 g(x,y_{\lambda}^*(x)) \Vert \\
    &\quad + \lambda \Vert \nabla^2 y^*(x) - \nabla^2 y_{\lambda}^*(x) \Vert  \Vert \nabla_{yx}^2 g(x,y_{\lambda}^*(x)) \Vert,
\end{align*}
{which only requires using triangle inequality multiple times. Then we plug
\begin{align} \label{eq:lambda}
\begin{split} 
            \Vert y_{\lambda}^*(x) - y^*(x) \Vert &= \fO(1/\lambda), \text{ By Lemma \ref{lem:yy-lambda} } \\
        \Vert \nabla y_{\lambda}^*(x) - \nabla y^*(x) \Vert &= \fO(1/\lambda), \text{ By Lemma \ref{lem:yxyx-lambda}. } \\
        \Vert \nabla^2 y_{\lambda}^*(x) - \nabla^2 y^*(x) \Vert &= \fO(1/\lambda), \text{ By Lemma \ref{lem:nabla2-yx}.}
\end{split}
\end{align}
and
\begin{align} \label{eq:1}
\begin{split}
        \Vert \nabla y^*(x) \Vert &= \fO(1) , \text{ By \Eqref{eq:y-x}.} \\
    \Vert \nabla y_{\lambda}^*(x) \Vert &= \fO(1), \text{ By Lemma \ref{lem:nabla-y-lambda-bound}}. \\
    \Vert \nabla^2 y^*(x) \Vert &= \fO(1), \text{ By \Eqref{bound:nabla2-yx}}.
\end{split}
\end{align}
into the bound for $C_2$ to obtain that
} 
\begin{align*}
    C_1+C_2 & \le \left( 1 + \frac{4L_g}{\mu} \right)^2 \left( 3\rho_f  + \frac{2 L_f \rho_g}{\mu} \right) + \left(1 + \frac{L_g}{\mu} \right)^2 \frac{C_f \nu_g}{\mu} \\
    &\quad + \left(2+ \frac{5L_g}{\mu} \right) D_2 \rho_g +  \left( 1+ \frac{L_g}{\mu} \right)^2 \frac{C_f \rho_g^2}{\mu^2} + L_g D_4.
\end{align*}
\end{proof}
 
\begin{lem} \label{lem:nabla2-phi}
Under Assumption \ref{asm:sc} and \ref{asm:third}, for $\lambda \ge 2 L_f /\mu$, we have $ \Vert \nabla^2 \fL_{\lambda}^*(x) - \nabla^2 \varphi(x) \Vert \le  D_6 /\lambda$,
where
\begin{align*}
    D_6 := 2L_g D_2^2 + \left( 1 + \frac{L_g}{\mu}\right)^2\left( \frac{C_f \rho_f}{\mu} + \frac{C_f L_f \rho_g}{\mu^2} + \frac{C_f^2 \nu_g}{2 \mu^2} + \frac{C_f^2 \rho_g^2}{2 \mu^3}  \right)  = \fO(\ell \kappa^6).
\end{align*}
\end{lem}

\begin{proof}
    Taking total derivative on $ \nabla \varphi(x) = \nabla_x f(x,y^*(x)) + \nabla y^*(x) \nabla_y f(x,y^*(x))$ yields
    \begin{align} \label{eq:phi-hess}
\begin{split}
    \nabla^2 \varphi(x) &= \nabla_{xx}^2 f(x,y^*(x)) + \nabla y^*(x) \nabla_{yx}^2 f(x,y^*(x))  + \nabla^2 y^*(x) \bar \times_3 \nabla_y f(x,y^*(x)) \\
    &\quad + \nabla_{xy}^2f(x,y^*(x)) [\nabla y^*(x)]^\top + \nabla y^*(x) \nabla_{yy}^2 f(x,y^*(x)) [\nabla y^*(x)]^\top.
\end{split}
\end{align}
Plug into the close form of $\nabla^2 y^*(x)$ given by \Eqref{eq:nabla2-yx}, we arrive at
    \begin{align*}
    &\quad \nabla^2 \varphi(x) \\
    &= \underbrace{\nabla_{xx }^2  f(x,y^*(x)) - \nabla_{xxy}^3 g(x,y^*(x)) \times_3 [  \nabla_{yy}^2 g(x,y^*(x))]^{-1} \bar \times_3 \nabla_y f(x,y^*(x)) }_{\rm (I)} \\
    &\quad +  
    \underbrace{  \left(\nabla_{yx}^2 f(x,y^*(x)) - \nabla_{yxy}^3 g(x,y^*(x)) \times_3 [\nabla_{yy}^2 g(x,y^*(x)]^{-1}  \bar \times_3 \nabla_y f(x,y^*(x)) \right) \times_1 \nabla y^*(x)  }_{\rm (II)}  \\
    &\quad +  
    \underbrace{  \left(\nabla_{xy}^2 f(x,y^*(x)) - \nabla_{xyy}^3 g(x,y^*(x)) \times_3 [\nabla_{yy}^2 g(x,y^*(x)]^{-1}  \bar \times_3 \nabla_y f(x,y^*(x)) \right) \times_2 \nabla y^*(x)  }_{\rm (III)}  \\
    &\quad + \underbrace{  \left( \nabla_{yy}^2 f(x,y^*(x)) -  \nabla_{yyy}^3 g(x,y^*(x)) \times_3 [\nabla_{yy}^2 g(x,y^*(x)]^{-1} \bar \times_3 \nabla_y f(x,y^*(x))   \right) \times_1 \nabla y^*(x) \times_2 \nabla y^*(x) }_{\rm (IV)}. 
\end{align*}
Recall that
\begin{align*}
    \nabla^2 \fL_{\lambda}^*(x) &= \nabla_{xx}^2  f(x,y_{\lambda}^*(x)) +  \lambda ( \nabla_{xx}^2 g(x,y_{\lambda}^*(x)) - \nabla_{xx}^2 g(x,y^*(x))) \\
    &\quad + \nabla y_{\lambda}^*(x) \nabla_{yx}^2 \fL_{\lambda}(x,y_{\lambda}^*(x))
    - \lambda \nabla y^*(x) \nabla_{yx}^2 g(x,y^*(x)) .
\end{align*}
{Our goal is show that $\nabla \fL_{\lambda}^*(x) \approx \nabla^2 \varphi(x)$. At first glance, these two quantities are very different and we can not directly bound their difference: $\nabla^2 \varphi(x)$ takes the form of $A + B C + C^\top B^\top + B D B^\top $, while $\nabla^2 \fL_{\lambda}^*(x)$ looks different.  Below, we introduce an
intermediary quantity $\tilde \nabla^2 \fL_{\lambda}^*(x)$ which takes a similar form as $\nabla^2 \varphi(x)$ to serves as a bridge:
}
\begin{align} \label{eq:inter-term}
\begin{split}
        \tilde \nabla^2 \fL_{\lambda}^*(x) &= \underbrace{\nabla_{xx}^2  f(x,y_{\lambda}^*(x)) +  \lambda ( \nabla_{xx}^2 g(x,y_{\lambda}^*(x)) - \nabla_{xx}^2 g(x,y^*(x)))}_{\rm (I')}\\
    &+  \underbrace{\nabla y^*(x) \left(\nabla_{yx}^2 \fL_{\lambda}(x,y_{\lambda}^*(x))- \lambda \nabla_{yx}^2 g(x,y^*(x))\right) }_{\rm (II')}\\ 
    & + \underbrace{ \left(\nabla_{xy}^2 \fL_{\lambda}(x,y_{\lambda}^*(x))- \lambda \nabla_{xy}^2 g(x,y^*(x))\right)[\nabla y^*(x)]^\top }_{\rm (III')} \\
    &+ \underbrace{\nabla y^*(x) \left(\nabla_{yy}^2 \fL_{\lambda} (x,y_{\lambda}^*(x))   - \lambda \nabla_{yy}^2 g(x,y^*(x)) \right) [\nabla y^*(x)]^\top }_{\rm (IV')}. 
\end{split}
\end{align}
{Now we show $\nabla^2 \fL_{\lambda}^*(x) \approx \tilde \nabla^2 \fL_{\lambda}^*(x) \approx \nabla^2 \varphi(x)$.}
\begin{align*}
    &\quad \tilde \nabla^2 \fL_{\lambda}^*(x) - \nabla^2 \fL_{\lambda}^*(x) \\
    &= ( \nabla y^*(x) - \nabla y_{\lambda}^*(x) ) \nabla_{yx}^2 \fL_{\lambda} (x,y_{\lambda}^*(x)) + \\
    &\quad + \nabla_{xy}^2 \fL_{\lambda}(x,y_{\lambda}^*(x)) [ \nabla y^*(x)]^\top + \nabla y^*(x) \nabla_{yy}^2 \fL_{\lambda} (x,y_{\lambda}^*(x)) [\nabla y^*(x)]^\top \\
    &= ( \nabla y^*(x) - \nabla y_{\lambda}^*(x) ) \nabla_{yx}^2 \fL_{\lambda} (x,y_{\lambda}^*(x)) + \nabla_{xy}^2 \fL_{\lambda}(x,y_{\lambda}^*(x)) [ \nabla y^*(x) - \nabla y_{\lambda}^*(x)]^\top \\
    &\quad + \nabla y^*(x) \nabla_{yy}^2 \fL_{\lambda} (x,y_{\lambda}^*(x)) [\nabla y^*(x)]^\top - \nabla y_{\lambda}^*(x) \nabla_{yy}^2 \fL_{\lambda} (x,y_{\lambda}^*(x)) [\nabla y_{\lambda}^*(x)]^\top \\
    &= ( \nabla y^*(x) - \nabla y_{\lambda}^*(x) ) \nabla_{yx}^2 \fL_{\lambda} (x,y_{\lambda}^*(x)) + \nabla_{xy}^2 \fL_{\lambda}(x,y_{\lambda}^*(x)) [ \nabla y^*(x) - \nabla y_{\lambda}^*(x)]^\top \\ 
    &\quad + (\nabla y^*(x) - \nabla y_{\lambda}^*(x)) \nabla_{yy}^2 \fL_{\lambda} (x,y_{\lambda}^*(x)) [ \nabla y_{\lambda}^*(x)]^\top \\
    &\quad + \nabla y_{\lambda}^*(x) \nabla_{yy}^2 \fL_{\lambda} (x,y_{\lambda}^*(x)) (\nabla y^*(x) - \nabla y_{\lambda}^*(x))^\top \\
    &\quad + (\nabla y^*(x) - \nabla y_{\lambda}^*(x)) \nabla_{yy}^2 \fL_{\lambda} (x,y_{\lambda}^*(x)) [ \nabla y^*(x) - \nabla y_{\lambda}^*(x)]^\top \\
    &= (\nabla y^*(x) - \nabla y_{\lambda}^*(x)) \nabla_{yy}^2 \fL_{\lambda} (x,y_{\lambda}^*(x)) [ \nabla y^*(x) - \nabla y_{\lambda}^*(x)]^\top,
\end{align*}
{where we use the identity
\begin{align*}
    U A U^\top - V A V^\top = (U-V) A V^\top + VA (U-V)^\top + (U-V)A (U-V)^\top
\end{align*}
in the second last step, and we cancel the first four terms by
\Eqref{eq:y-lambda-x} in the final step.} Noting that $ \nabla_{yy}^2 \fL_{\lambda}(\,\cdot\, , \,\cdot\,) \preceq 2 \lambda L_g$, we have
\begin{align*}
    \Vert \tilde \nabla^2 \fL_{\lambda}^*(x) - \nabla^2 \fL_{\lambda}^*(x) \Vert \le 2 \lambda L_g \Vert \nabla y^*(x) - \nabla y_{\lambda}^*(x) \Vert^2 \le \frac{2L_g D_2^2}{\lambda }. 
\end{align*}
{Now, we have successfully simplified our goal to showing $\tilde \nabla^2 \fL_{\lambda}^*(x) \approx \nabla^2 \varphi(x)$, which can be done by showing 
${\rm (I)} \approx {\rm (I')}$, ${\rm (II)} \approx {\rm (II')}$, ${\rm (III)} \approx {\rm (III')}$, and ${\rm (IV)} \approx {\rm (IV')}$, separately. }Firstly,
{\begin{align*}
    &\quad \rm (I) - {\rm (I')} \\
    &= \nabla_{xx}^2 f(x,y^*(x)) - \nabla_{xx}^2 f(x,y_{\lambda}^*(x))  \\
    & \quad + \nabla_{xxy}^3 g(x,y^*(x)) \times_3 [\nabla_{yy}^2 g(x,y^*(x))]^{-1} \bar \times_3 ( \nabla_y f(x,y_{\lambda}^*(x)) - \nabla_y f(x,y^*(x))) \\
    &\quad + \lambda \left( \nabla_{xx}^2 g(x,y^*(x)) -  \nabla_{xx}^2 g(x,y_{\lambda}^*(x))  + \nabla_{xxy}^3  g(x,y^*(x)) \bar \times_3
    (y_{\lambda}^*(x) - y^*(x)) \right) \\
    &\quad + \lambda \nabla_{xxy}^3 g(x,y^*(x)) \times_3 [\nabla_{yy}^2 g(x,y^*(x))]^{-1} 
 \\
 &~~~~~~~~~~~~~~~~~~~~~~~~~~~\bar \times_3 \left(\nabla_y g(x,y_{\lambda}^*(x)) 
    - \nabla_y g(x,y^*(x)) - \nabla_{yy}^2 g(x,y^*(x)) (y_{\lambda}^*(x) -y^*(x))  \right).
\end{align*}}
Therefore, 
\begin{align*}
    \Vert {\rm (I) } - {\rm (I')} \Vert &\le \rho_f \Vert y^*(x) - y_{\lambda}^*(x) \Vert  + \frac{L_f \rho_g}{\mu} \Vert y^*(x) - y_{\lambda}^*(x) \Vert \\
    &\quad + \frac{\lambda \nu_g}{2} \Vert y^*(x) - y_{\lambda}^*(x) \Vert^2 + \frac{\lambda \rho_g^2}{2\mu} \Vert y^*(x) - y_{\lambda}^*(x) \Vert^2  \\
    &\le \frac{C_f \rho_f}{\lambda \mu} + \frac{C_f L_f \rho_g}{\lambda \mu^2} +  \frac{C_f^2 \nu_g}{2\lambda \mu^2}+ \frac{C_f^2 \rho_g^2}{2 \lambda \mu^3} .
\end{align*}
Using $\Vert \nabla y^*(x) \Vert \le  L_g / \mu$, we can similarly bound the difference between ${\rm (II)}$ and ${\rm (II')}$ by
{ \begin{align*}
    &\quad \lVert{\rm (II)} - {\rm (II')}\rVert \\
    &\le 
      \Vert \nabla y^*(x) \Vert \Big \Vert  \nabla_{yx}^2 f(x,y^*(x)) - \nabla_{yxy}^3 g(x,y^*(x)) \times_3 [\nabla_{yy}^2 g(x,y^*(x)]^{-1} \bar \times_3 \nabla_y f(x,y^*(x))\\
     &\quad - \nabla_{yx}^2 f(x,y_{\lambda}^*(x))-\lambda  \left(\nabla_{yx}^2 g(x,y_{\lambda}^*(x))-\nabla_{yx}^2 g(x,y^*(x)) \right) \Big \Vert \\
     & \le \frac{L_g }{\mu} \left( \frac{C_f \rho_f}{\lambda \mu} + \frac{C_f L_f \rho_g}{\lambda \mu^2}  + \frac{C_f^2 \nu_g}{2\lambda \mu^2} +\frac{C_f^2 \rho_g^2}{2 \lambda \mu^3} \right),
\end{align*}}
And the bound for ${\rm(III)}$ and $ {\rm (III')}$ is the same. Finally, we know that
\begin{align*}
    &\quad \left \Vert {\rm (IV)} - {\rm (IV')} \right \Vert \\
    &\le \Vert \nabla y^*(x) \Vert^2 \Big \Vert \nabla_{yy}^2 f(x,y^*(x)) - \nabla_{yyy}^3 g(x,y^*(x)) \times_3 [ \nabla_{yy}^2 g(x,y^*(x))]^{-1} \bar \times_3 \nabla_y f(x,y^*(x))  \\
    &\quad - \nabla_{yy}^2 f(x,y_{\lambda}^*(x)) - \lambda \left( \nabla_{yy}^2 g(x,y_{\lambda}^*(x)) - \nabla_{yy}^2 g(x,y^*(x)) \right) \Big \Vert \\
    &\le \frac{L_g^2}{\mu^2} \left( \frac{C_f \rho_f}{\lambda \mu} + \frac{C_f L_f \rho_g}{\lambda \mu^2}  + \frac{C_f^2 \nu_g}{2\lambda \mu^2} +  \frac{C_f^2 \rho_g^2}{2 \lambda \mu^3}\right).
\end{align*}
Combing the above bounds completes our proof.
\end{proof}

\section{Proofs of Finding First-Order Stationarity} \label{apx:1st}

We recall the standard result for gradient descent for strongly convex functions, which is used for proving the linear convergence of $y_t^k \rightarrow y_{\lambda}^*(x_t) $ and $z_{t}^k \rightarrow y^*(x_t)$.
\begin{thm}[{\citet[Theorem 3.10]{bubeck2015convex}}] \label{thm:GD-SC}
Suppose $h(x): \BR^d \rightarrow \BR$ is  $\beta$-gradient Lipschitz and $\alpha$-strongly convex. Consider the following update of gradient descent:
\begin{align*}
    x_{t+1} = x_t - \frac{1}{\beta} \nabla h(x_t).
\end{align*}
Let $x^* = \arg \min_{x \in \BR^d} h(x)$. Then it holds that
\begin{align*}
    \Vert x_{t+1} - x^* \Vert^2 \le \left(1 - \frac{\alpha}{\beta} \right) \Vert x_t - x^* \Vert^2.
\end{align*}
\end{thm}

Below, we prove that F${}^2$BA achieves the near-optimal first-order oracle complexity.

\thmFBA* 
\begin{proof}
Let $L$ be the gradient Lipschitz coefficient of $\fL_{\lambda}^*(x)$.
According to Lemma \ref{lem:1st} and the setting of $\lambda$, we have
\begin{enumerate}[label=\alph*.]
    \item $\sup_{x \in \BR^{d_x}} \Vert \nabla \fL_{\lambda}^*(x) - \nabla \varphi(x) \Vert = \fO(\epsilon)$. 
    \item $\fL_{\lambda}^*(x_0) - \inf_{x \in \BR^{d_x}}  \fL_{\lambda}^*(x) = \fO(\Delta)$.
    \item $ L:=\sup_{x \in \BR^{d_x}}\Vert \nabla^2 \fL_{\lambda}^*(x) \Vert = \fO(\ell \kappa^3)$.
\end{enumerate}
Due to Lemma \ref{lem:yy-lambda}, we also have $ \Vert y_0 - y_{\lambda}^*(x_0) \Vert^2 + \Vert y_0 - y^*(x_0) \Vert^2 =\fO(R)$.
Now it suffices to show that the algorithm can find an $\epsilon$-first-order stationary of $\fL_{\lambda}^*(x)$. 
We begin from the following descent lemma for gradient descent. Let $\eta_x \le 1/(2L)$, then
\begin{align} \label{eq:des}
\begin{split}
     \fL_{\lambda}^*(x_{t+1})
    &\le \fL_{\lambda}^*(x_t) + \langle \nabla \fL_{\lambda}^*(x_t) , x_{t+1} - x_t
    \rangle + \frac{L}{2} \Vert x_{t+1 } - x_t \Vert^2 \\
    &= \fL_{\lambda}^*(x_t) -
    \frac{\eta_x}{2} \Vert \nabla \fL_{\lambda}^*(x_t) \Vert^2 -
    \left( \frac{\eta_x}{2} - \frac{\eta_x^2 L}{2} \right) \Vert \hat \nabla \fL_{\lambda}^*(x_t) \Vert^2 + \frac{\eta_x}{2} \Vert \hat  \nabla \fL_{\lambda}^*(x_t) - \nabla \fL_{\lambda}^*(x_t) \Vert^2 \\
    &\le \fL_{\lambda}^*(x_t) 
    - \frac{\eta_x}{2} \Vert \nabla \fL_{\lambda}^*(x_t) \Vert^2 - 
    \frac{1}{4\eta_x} \Vert x_{t+1} - x_{t} \Vert^2
    + \frac{\eta_x}{2} \Vert  \hat \nabla \fL_{\lambda}^*(x_t) - \nabla \fL_{\lambda}^*(x_t) \Vert^2.
\end{split}
\end{align}
Note that
\begin{align*}
    \Vert \hat \nabla \fL_{\lambda}^*(x_t) - \nabla \fL_{\lambda}^*(x_t) \Vert \le 2 \lambda L_g \Vert y_{t}^K- y_{\lambda}^*(x_t) \Vert + \lambda L_g \Vert z_{t}^K - y^*(x_t) \Vert.
\end{align*}
By Theorem \ref{thm:GD-SC}, we have
\begin{align*}
    \Vert y_t^K - y_{\lambda}^*(x_t) \Vert^2 &\le \exp\left( - \frac{\mu K}{4 L_g} \right) \Vert y_t^0 - y_{\lambda}^*(x_t) \Vert^2 \\
     \Vert z_t^K - y^*(x_t) \Vert^2 &\le \exp\left( - \frac{\mu K}{L_g} \right) \Vert z_t^0 - y^*(x_t) \Vert^2
\end{align*}
Therefore, we have
\begin{align} \label{eq:err}
    \Vert \hat \nabla \fL_{\lambda}^*(x_t)- \nabla \fL_{\lambda}^*(x_t) \Vert^2 \le 4 \lambda^2 L_g^2 \exp\left( - \frac{\mu K}{4 L_g} \right) \left( \Vert y_t^0 - y_{\lambda}^*(x_t) \Vert^2 + \Vert z_t^0 - y^*(x_t) \Vert^2 \right)
\end{align}
By Young's inequality, 
\begin{align*}
    \Vert y_{t+1}^0 - y_{\lambda}^*(x_{t+1}) \Vert^2 & \le 2 \Vert y_t^K - y_{\lambda}^*(x_{t}) \Vert^2 + 2 \Vert y_{\lambda}^*(x_{t+1}) - y_{\lambda}^*(x_t) \Vert^2  \\
    &\le 2 \exp \left( - \frac{\mu K}{4 L_g} \right) \Vert y_t^0 - y_{\lambda}^*(x_t) \Vert^2 + \frac{32 L_g^2}{\mu^2} \Vert x_{t+1} - x_t \Vert^2,
\end{align*}
where the second inequality follows Lemma \ref{lem:nabla-y-lambda-bound} that $y_{\lambda}^*(x)$ is $(4 L_g/\mu)$-Lipschitz.  Similarly, we can derive the recursion about $\Vert z_t^0 - y^*(x_t) \Vert^2$. 

Put them together and let $K \ge 8 L_g/ \mu$, we have
\begin{align} \label{eq:delta_t}
    &\quad \Vert y_{t+1}^0 - y_{\lambda}^*(x_{t+1}) \Vert^2  + \Vert z_{t+1}^0 - y^*(x_{t+1}) \Vert^2 \\
    &\le 2 \exp \left( - \frac{\mu K}{4 L_g} \right) \left(\Vert y_t^0 - y_{\lambda}^*(x_t) \Vert^2 + \Vert z_t^0 - y^*(x_t) \Vert^2 \right) +\frac{34 L_g^2}{\mu^2} \Vert x_{t+1} - x_t \Vert^2 \\
    &\le \frac{1}{2} \left(\Vert y_t^0 - y_{\lambda}^*(x_t) \Vert^2 + \Vert z_t^0 - y^*(x_t) \Vert^2 \right) +\frac{34 L_g^2}{\mu^2} \Vert x_{t+1} - x_t \Vert^2.
\end{align}
Telescoping over $t$ yields
\begin{align*}
    &\quad \Vert y_t^0 - y_{\lambda}^*(x_t) \Vert^2 + \Vert z_t^0 - y^*(x_t) \Vert^2 \\
    &\le \underbrace{\left( \frac{1}{2} \right)^t \left(\Vert y_0 - y_{\lambda}^*(x_0) \Vert^2 + \Vert y_0 - y^*(x_0) \Vert^2 \right) + \frac{34 L_g^2}{\mu^2} \sum_{j=0}^{t-1}  \left( \frac{1}{2}\right)^{t-1-j} \Vert x_{j+1} - x_j \Vert^2}_{:=(*)}.
\end{align*}
Plug into \Eqref{eq:err}, which, in conjunction with \Eqref{eq:des}, yields that
\begin{align*}
     \fL_{\lambda}^*(x_{t+1}) &\le \fL_{\lambda}^*(x_t) 
    - \frac{\eta_x}{2} \Vert \nabla \fL_{\lambda}^*(x_t) \Vert^2 - 
    \frac{1}{4\eta_x} \Vert x_{t+1} - x_{t} \Vert^2 + 2 \eta_x \times \underbrace{ \lambda^2 L_g^2 \exp \left( - \frac{\mu K}{4L_g}\right)}_{:=\gamma} \times (*).
\end{align*}
Telescoping over $t$ further yields
\begin{align} \label{eq:final}
\begin{split}
    \frac{\eta_x}{2}\sum_{t=0}^{T-1} \Vert \nabla \fL_{\lambda}^*(x_t) \Vert^2 &\le \fL_{\lambda}^*(x_0) - \inf_{x \in \BR^{d_x}} \fL_{\lambda}^*(x) + 4 \eta_x \gamma \left(\Vert y_0 - y_{\lambda}^*(x_0) \Vert^2 + \Vert y_0 - y^*(x_0) \Vert^2 \right) \\
    &\quad - \left( \frac{1}{4 \eta_x} - \frac{136 \eta_x \gamma L_g^2}{\mu^2} \right) \sum_{t=0}^{T-1} \Vert x_{t+1} - x_t \Vert^2
\end{split}
\end{align}
Let $K = \fO(\kappa \log(\lambda \kappa)) = \fO(\kappa \log(\nicefrac{\ell \kappa}{\epsilon}))$ such that $\gamma$ is sufficiently small with
\begin{align*}
    \gamma \le \min \left\{ \frac{\mu^2}{1088 \eta_x^2 L_g^2}, \frac{1}{4 \eta_x}\right\}.
\end{align*}
Then we have,
\begin{align*}
    \frac{1}{T} \sum_{t=0}^{T-1} \Vert \nabla \fL_{\lambda}^*(x_t) \Vert^2 \le  \frac{2}{\eta_x T} \left( \fL_{\lambda}^*(x_0) - \inf_{x \in \BR^{d_x}} \fL_{\lambda}^*(x) + \Vert y_0 - y_{\lambda}^*(x_0) \Vert^2 + \Vert y_0 - y^*(x_0) \Vert^2 \right).
\end{align*}
\end{proof}

Below, we focus on the stochastic setting. We need to use the following theorem for large batch SGD to bound the error from inner loop.

\begin{thm}
\label{thm:SGD-SC}
Suppose $h(x): \BR^d \rightarrow \BR$ is  $\beta$-gradient Lipschitz and $\alpha$-strongly convex. Consider the following update of stochastic gradient descent (SGD):
\begin{align*}
    x_{t+1} = x_t - \frac{1}{\beta} \nabla h(x_t;\fB_t),
\end{align*}
where the mini-batch gradient satisfies 
\begin{align*}
    \BE_{\fB_t} \left[ \nabla h(x_t;\fB_t) \right] = \nabla h(x_t) , \quad \BE_{\fB_t} \Vert \nabla h(x_t;\fB_t) - \nabla h(x_t) \Vert^2 \le \frac{\sigma^2}{B}.
\end{align*}
Then it holds that
\begin{align*}
    \BE \left[ \Vert x_T - x^* \Vert^2 \right] \le \left(1 - \frac{\alpha}{\beta} \right)^T \frac{\beta}{\alpha} \Vert x_0 - x^* \Vert^2 + \frac{\sigma^2}{\alpha^2 B},
\end{align*}
where $x^* = \arg \min_{x \in \BR^d} h(x)$.
\end{thm} 

\begin{proof}
We have that
    \begin{align*}
         \BE \left[ h(x_{t+1}) - h^*\right] 
        &\le \BE \left[ h(x_t) - h^* + \nabla h(x_t)^\top (x_{t+1} - x_t) + \frac{\beta}{2} \Vert x_{t+1} - x_t \Vert^2 \right]\\
        &= \BE \left[ h(x_t) - h^* - \frac{1}{2\beta} \Vert \nabla h(x_t) \Vert^2 + \frac{1}{2 \beta} \Vert \nabla h(x_t) -  \nabla h(x_t; \fB_t) \Vert^2 \right] \\
        &\le \left( 1 - \frac{\alpha}{\beta}\right) \left( h(x_t) - h^* \right) + \frac{\sigma^2}{2 \beta B}.
    \end{align*}
Telescoping gives 
\begin{align*}
    \BE \left[ h(x_T) - h^*\right] \le \left(1 - \frac{\alpha}{\beta} \right)^T  \left( h(x_0) - h^* \right)  + \frac{\sigma^2}{2 \alpha B}.
\end{align*}
We conclude the proof by using the fact that
\begin{align*}
    \frac{\alpha}{2} \Vert x - x^* \Vert^2 \le h(x) - h^* \le \frac{\beta}{2} \Vert x - x^* \Vert^2. 
\end{align*}
\end{proof}

\thmFBSA*

\begin{proof} By taking expectation in \Eqref{eq:des}, we get
\begin{align} \label{eq:des-stoc}
\begin{split}
    &\quad \BE[ \fL_{\lambda}^*(x_{t+1})]  \\
    &\le \BE \left[\fL_{\lambda}^*(x_t) 
    - \frac{\eta_x}{2} \Vert \nabla \fL_{\lambda}^*(x_t) \Vert^2 - 
    \frac{1}{4\eta_x} \Vert x_{t+1} - x_{t} \Vert^2
    + \frac{\eta_x}{2} \Vert G_t - \nabla \fL_{\lambda}^*(x_t) \Vert^2 \right].
\end{split}
\end{align}

Then to prove that this algorithm can output a point such that $\BE \Vert \nabla \varphi(x) \Vert \le \epsilon$, it
suffices to show $\BE \Vert G_t - \nabla \fL_{\lambda}^*(x_t) \Vert^2 = \fO(\epsilon^2)$ holds for all $t$. This requires
\begin{enumerate}[label=\alph*.]
\item $\BE \Vert \nabla_x f(x_t,y_t^K;B_{\rm out}) - \nabla_x f(x_t,y_{\lambda}^*(x_t)) \Vert^2 = \fO(\epsilon^2)$.
\item $\BE \Vert \nabla_x g(x_t,y_t^K;B_{\rm out}) - \nabla_x g(x_t,y_{\lambda}^*(x_t)) \Vert^2 = \fO\left( \epsilon^2/ \lambda^2 \right)$.
\item $  \BE \Vert \nabla_x g(x_t,z_t^K;B_{\rm out}) - \nabla_x g(x_t,y^*(x_t)) \Vert^2 = \fO\left( {\epsilon^2}/{\lambda^2} \right)$.
\end{enumerate}
Our setting of $B_{\rm out}$ shows that it suffices to have
\begin{align*}
\max \left\{ \BE \Vert y_t^K - y_{\lambda}^*(x_t) \Vert^2 ,\BE \Vert z_t^K - y^*(x_t) \Vert^2 \right\}  &= \fO \left( \frac{\epsilon^2}{L_f^2+ \lambda^2 L_g^2} \right), 
\end{align*}
And our setting of $B_{\rm in}$ and $K_t$ fulfills these conditions if $\Vert y_t^0 - y_{\lambda}^*(x_t) \Vert^2 + \Vert z_t^0 - y^*(x_t) \Vert^2 \le \delta_t$.
Now we use Theorem \ref{thm:SGD-SC} to establish the recursion of $\delta_t$. Note that
\begin{align} \label{eq:recursion-delta-t}
\begin{split}
    &\quad \BE \Vert y_{t+1}^0- y_{\lambda}^*(x_{t+1}) \Vert^2  \\
    & \le 2 \BE \Vert y_t^K- y_{\lambda}^*(x_{t}) \Vert^2 + 2 \BE \Vert y_{\lambda}^*(x_{t}) - y_{\lambda}^*(x_{t+1}) \Vert^2 \\
    &\le \frac{4 L_g}{\mu} \exp \left( - \frac{\mu K_t}{2L_g} \right) \BE \Vert y_t^0 - y_{\lambda}^*(x_t) \Vert^2  + \frac{32 L_g^2}{\mu^2}  \Vert x_{t+1} - x_t \Vert^2 + \frac{ \sigma_g^2}{4 L_g B_{\rm in}} \\
    &\le \frac{1}{2} \BE \Vert y_t^0 - y_{\lambda}^*(x_t) \Vert^2  + \frac{32 L_g^2}{\mu^2}  \Vert x_{t+1} - x_t \Vert^2 + \frac{ \sigma_g^2}{4 L_g B_{\rm in}} 
\end{split}
\end{align}
where the second last inequality uses Theorem \ref{thm:SGD-SC} and the last inequality holds once we get $K_t = \Omega( \kappa \log (\kappa))$.  A similar bound also holds for $\BE \Vert z_{t+1}^0 - y^*(x_{t+1}) \Vert^2$. Putting them together, we get that
\begin{align*}
    \delta_{t+1} \le \frac{1}{2} \delta_t +  \frac{34 L_g^2}{\mu^2}  \Vert x_{t+1} - x_t \Vert^2 + \frac{ \sigma_g^2}{2 L_g B_{\rm in}}
\end{align*}
Then we telescope the recursion of $\delta_t$ and take expectation to obtain that
\begin{align} \label{eq:ub-deltat}
\begin{split}
        \sum_{t=0}^{T-1} \BE[\delta_t] &\le 2 \delta_0 + \frac{68  L_g^2}{\mu^2} \sum_{t=0}^{T-1} \BE [\Vert x_{t+1} - x_t \Vert^2] + \frac{ \sigma_g^2}{ L_g B_{\rm in}} \\
    &=  2 \delta_0 + \fO(\eta_x \Delta + T \eta_x^2 \epsilon^2) + \frac{ \sigma_g^2}{ L_g B_{\rm in}},
\end{split}
\end{align}
where we use \Eqref{eq:des-stoc} in the last step. Let $C$ be a constant that  contains logarithmic factors. We can the bound the expected total iteration number via
\begin{align*}
    \sum_{t=0}^{T-1} \BE[K_t] = \frac{C L_g}{\mu} \sum_{t=0}^{T-1} \BE [\log (\delta_t)]  \le \frac{C L_g T }{\mu} \log \left( \frac{\sum_{t=0}^{T-1} \BE[\delta_t]}{T} \right),
\end{align*}
where we use Jensen's inequality and the concavity of $\log (\,\cdot\,)$. Note that $\sum_{t=0}^{T-1} \BE[\delta_t]$ can be upper bounded via \Eqref{eq:ub-deltat}.
Finally, the total sample complexity of the algorithm is
\begin{align*}
    B_{\rm in} \sum_{t=0}^{T-1} K_t + T B_{\rm out},
\end{align*}
which is the claimed complexity by plugging the choice of $T,B_{\rm in}, B_{\rm out}$.
\end{proof}



{
\section{Removing the Large Batches in F${}^2$BSA} \label{apx:single-batch}

Theorem \ref{thm:F2BSA} requires large batch size $B_{\rm out} \asymp (\sigma_f^2 + \lambda^2 \sigma_g^2) / \epsilon^2$ to track the deterministic algorithm. Below, we get rid of the large batches via a more intricate algorithm design. In this section, we propose new techniques to reduce the batch size to $M=\tilde \fO(\kappa)$ in Algorithm~\ref{alg:F2HE}. Following the ideas in Section~\ref{sec:second}, we define the following estimators:
\begin{align}
\begin{split}
    &\quad \nabla \fL_{\lambda}^*(x; B) =  \nabla_x \fL_{\lambda}(x, y_{\lambda}^*(x), y^*(x); B), \\
    &{\rm where}~~ \nabla_x \fL_{\lambda}(x, y,z; B) = \nabla_x f(x,y;B) + \lambda (\nabla_x g(x,y; B) - \nabla_x g(x,z;B)).
\end{split}
\end{align}
It is clear that $\BE \nabla \fL_{\lambda}(x,y,z; B) = \nabla \fL_{\lambda}(x,y,z)$. 
Our improved convergence rate in Theorem~\ref{thm:F2BSA} requires the following condition on the hyper-gradient estimator $G$.

\begin{algorithm*}[t]  
\caption{F${}^3$BSA$(x, y_0)$} \label{alg:F3BSA}
\begin{algorithmic}[1]
\STATE $z_0 = y_0$ \\[1mm]
\STATE \textbf{for} $ t=0,1,\cdots,T-1$ \\[1mm]
\STATE \quad $(G_t, y_{t+1} , z_{t+1}) = \text{MLMC-HyperGradEst}(x_t, y_t, z_t, M_t, N_t, S_t )$ \\[1mm]
\STATE \quad $x_{t+1} = x_t - \eta_x G_t$ \\[1mm]
\STATE \textbf{end for} 
\end{algorithmic}
\end{algorithm*}
\begin{algorithm*}[t]  

\caption{MLMC-HyperGradEst$(x, y_0, z_0, M,  N,  S)$} \label{alg:F2HE}
\begin{algorithmic}[1] 
\STATE \textbf{for} $m=1,\cdots,M$ \\ [1mm]
\STATE \quad Draw $J \sim {\rm Geom}(1/2)$ \\[1mm]  
\STATE \quad $(G_m^{(j)}, y_m^{(j)} , z_m^{(j)}) = \text{HyperGradEst}(x,y_0,z_0, N, j, 2^j )$ \\[1mm]
\STATE \quad $G_m = G_m^{(0)} + 2^J (G_m^{(J)} - G_m^{(J-1)}) \BI[ J \le S] $ \\[1mm] 
\STATE \textbf{end for} \\[1mm]
\STATE $ \hat G = \frac{1}{M} \sum_{m=1}^M G_m$, $ \hat y = y_1^{(0)}$, $ \hat z = y_1^{(0)}$ \\[1mm]
\STATE \textbf{return} $(\hat G, \hat y , \hat z)$ 
\end{algorithmic}
\end{algorithm*}

\begin{algorithm*}[t]  
\caption{HyperGradEst$(x, y_0, z_0, N, K, B)$} \label{alg:single-F2HE}
\begin{algorithmic}[1] 
\STATE $\hat y = \text{SGD}{}^{\rm SC}( f+ \lambda g, y_0, N, K) $ \\[1mm]
\STATE $\hat z = \text{SGD}{}^{\rm SC}( \lambda g, z_0, N, K) $ \\[1mm]
\STATE $\hat G = \nabla_x f(x,\hat y; B) + \lambda (\nabla_x g(x, \hat y; B) - \nabla_x g(x, \hat z; B)) $ \\[1mm]
\STATE \textbf{return} $(\hat G, \hat y, \hat z)$
\end{algorithmic}
\end{algorithm*}

\begin{asm} \label{asm:hyper-bias-variance}
Assume that the hyper-gradient estimator $G$ satisfies 
\begin{align} \label{eq:cond-HE}
\begin{split}
 \Vert \BE G -  \nabla \fL_{\lambda}^*(x)  \Vert^2 \le \zeta_{\rm bias}^2 &= \fO( \min\{ \epsilon^2, \sigma_f^2 + \lambda^2 \sigma_g^2 \}), \\
\BE \Vert G - \BE G \Vert^2 \le \zeta_{\rm varinace}^2 &= \fO(\sigma_f^2 + \lambda^2 \sigma_g^2),
\end{split}
\end{align}
\end{asm}
\begin{restatable}{thm}{thmouter} \label{thm:outer}
    Suppose Assumption \ref{asm:sc} and \ref{asm:stochastic} hold. 
Let $\ell$,  $\kappa$ defined as Definition \ref{dfn:kappa-sc}. Let
$\lambda \asymp \max\{ \ell \kappa^2 / \Delta, ~\ell \kappa^3 / \epsilon \}$, where $\Delta:=\varphi(x_0) - \inf_{x \in \BR^{d_x}} \varphi(x)$.
If for any $t$ the hyper-gradient estimator $G_t$ satisfies Assumption \ref{asm:hyper-bias-variance},
then iterating $ x_{t+1}  = x_t - \eta_x G_t$ with
\begin{align} \label{eq:eta-outer-F2BSA}
    \eta_x \asymp \frac{\epsilon^2}{(\sigma_f^2 + \lambda^2 \sigma_g^2)  \ell 
 \kappa^3}, \quad T \asymp \frac{\Delta}{\eta_x \epsilon^2},
\end{align}
then sampling uniformly from $\{ x_0,\cdots, x_{T-1} \}$ gives a point $\hat x_T$ such that
$\BE \Vert \nabla \varphi(\hat x_T) \Vert \le \epsilon$. 
\end{restatable}

In Assumption \ref{asm:hyper-bias-variance}, a nearly unbiased hyper-gradient estimator $G_t$ is required to recover the same convergence rate as SGD on $\varphi(x)$. In F${}^2$BSA (Algorithm \ref{alg:F2BSA}), this condition of $\zeta_{\rm bias}^2 \lesssim \epsilon^2$ is achieved by running inner-loop SGD with large batch to accuracy $ \epsilon^2$. In the following, we show it is possible to use existing technique~\citep{asi2021stochastic,hu2021bias} for improving biased SGD, which helps relax the accuracy of inner-loop SGD to $ \sigma_f^2 + \lambda^2 \sigma_g^2$. In this case, using single-batch SGD would not harm the convergence rate. We call this improved algorithm F${}^3$BSA (see Algorithm \ref{alg:F3BSA}). The ultimate algorithms look a little complicated at first glance, but they all come from standard techniques in the literature. We summarize the key steps below.
\begin{enumerate}
    \item F${}^3$BSA (Algorithm \ref{alg:F3BSA}) conducts a SGD update on variable $x$ with the MLMC hyper-gradient estimator $G_t$ given by Algorithm \ref{alg:F2HE}.
    \item The design of MCMC-HyperGradEst (Algorithm \ref{alg:F2HE}) has the same structure as the standard MLMC estimator \citep{giles2008multilevel,blanchet2015unbiased}. Following \citep{asi2021stochastic,hu2021bias}, we wrap a high-bias hyper-gradient estimator (Algorithm~\ref{alg:single-F2HE}) into the MLMC structure to return a low-bias estimator. 
    \item HyperGradEst (Algorithm \ref{alg:single-F2HE}) applies SGD updates to solve the lower-level problems with respect to variables $y$ and $z$ in \Eqref{eq:nabla-Lag}. 
    Since the naive SGD update in F${}^2$BSA can not meet the 
    requirements in MLMC, we use the SGD${}^{
    \rm SC}$ algorithm \citep[Algorithm 2]{allen2018make}, which is presented in Section \ref{apx:SGD-SC} for completeness. It runs multiple epochs of SGD with different step sizes like \citep{hazan2014beyond} to achieve the optimal rate in smooth strongly-convex stochastic first-order optimization.
\end{enumerate}

The following theorem shows Algorithm \ref{alg:F2HE} can output a nearly unbiased hyper-gradient estimator with batch size $M = \tilde \fO(\kappa)$. The proof follows from \citep{asi2021stochastic}.

\begin{restatable}{thm}{thminner} \label{thm:inner}
Suppose Assumption \ref{asm:sc} and \ref{asm:stochastic} hold. 
Let $\ell$,  $\kappa$ defined as Definition \ref{dfn:kappa-sc}. For any given $x \in \BR^{d_x}$, then Algorithm \ref{alg:F2HE} with 
\begin{align} \label{eq:hyper-F2HE}
    N = S \asymp \log \left( \frac{\max\{ \lambda^2 \ell^2 \delta,   \kappa (\sigma_f^2 + \lambda^2 \sigma_g^2) \}}{ \zeta_{\rm bias}^2} \right), \quad M \asymp \frac{\kappa (\sigma_f^2 + \lambda^2 \sigma_g^2) N}{\zeta_{\rm variance}^2}
\end{align}
outputs an hyper-gradient estimator $G$ satisfying Condition \ref{eq:cond-HE} in $\fO( \kappa N M)$ SFO complexity, where $\delta$ is the upper bound of $\Vert y_0 - y_{\lambda}^*(x) \Vert^2 + \Vert z_0 - y^*(x) \Vert^2$.
\end{restatable}

Combining Theorem \ref{thm:outer} and \ref{thm:inner}, we obtain the complexity of Algorithm \ref{alg:F3BSA} as follows. 

\begin{restatable}{thm}{thmFSBAsingle} \label{thm:F3BSA}
Suppose Assumption \ref{asm:sc} and \ref{asm:stochastic} hold. 
Let $\ell$,  $\kappa$ defined as Definition \ref{dfn:kappa-sc}.
Define
$\Delta:=\varphi(x_0) - \inf_{x \in \BR^{d_x}} \varphi(x)$ and $R := \Vert y_0 - y^*(x_0) \Vert^2 $.
Let $\lambda \asymp \max\left\{\kappa/R,~ \ell \kappa^2/ \Delta, ~ \ell \kappa^3 / \epsilon \right\} $ 
and set other parameters in Algorithm \ref{alg:F3BSA} according to \Eqref{eq:eta-outer-F2BSA} and \Eqref{eq:hyper-F2HE}, 
where $\delta_t$ is defined via the recursion
\begin{align} \label{eq:recursion-delta-single}
    \delta_{t+1} = \frac{1}{2}\delta_t + \frac{34 L_g^2}{\mu^2}  \Vert x_{t+1} - x_t \Vert^2 + \frac{24 (\sigma_f^2+ \lambda^2 \sigma_g^2)}{\lambda \mu (L_f + \lambda L_g)}, \quad \delta_0 = \fO(R),
\end{align}
then it
can output a point such that $\BE \Vert \nabla \varphi(x) \Vert \le \epsilon$ in the total SFO complexity of
\begin{align*}
\begin{cases}
    \fO(\ell \kappa^5 \epsilon^{-4} \log (\nicefrac{ \ell \kappa}{\epsilon})), & \sigma_f >0 , \sigma_g = 0;\\
    {\fO(\ell^3 \kappa^{11} \epsilon^{-6} \log (\nicefrac{ \ell \kappa}{\epsilon})) }, & \sigma_f>0, \sigma_g >0.
\end{cases}
\end{align*}
\end{restatable}

The above theorem shows that F${}^3$BSA achieves the $\tilde \fO(\epsilon^{-4})$ and $\tilde \fO(\epsilon^{-6})$ SFO complexity for the partially and fully stochastic cases as F${}^2$BSA while using a single batch at each iteration. 
Another side advantage of F${}^3$BSA is that it also improves the dependency in $\kappa$ because its inner loop uses the optimal stochastic algorithm SGD${}^{\rm SC}$ \citep[Algorithm 2]{allen2018make} to achieve better bias-variance tradeoff, instead of using naive SGD in F${}^2$BSA.

\subsection{Missing Proofs for F${}^3$BSA} \label{apx:proof-single-batch}

We present the missing proofs in this section. 

\thmouter* 
\begin{proof} Let $L$ be the constant of gradient Lipschitz continuity of $\fL_{\lambda}^*(x)$. From 
Lemma \ref{lem:1st} we know that $L = \fO(\ell \kappa^3)$. We have that
    \begin{align*}
    &\quad \BE[ \fL_{\lambda}^*(x_{t+1})] \\
    &\le \BE \left[\fL_{\lambda}^*(x_t) + \langle \nabla \fL_{\lambda}^*(x_t) , x_{t+1} - x_t
    \rangle + \frac{L}{2} \Vert x_{t+1 } - x_t \Vert^2 \right] \\
    &= \BE \left[ \fL_{\lambda}^*(x_t) -
    \eta_x \nabla \fL_{\lambda}^*(x_t)^\top G_t + \frac{\eta_x^2 L}{2} \Vert G_t \Vert^2 \right] \\
    &\le \fL_{\lambda}^*(x_t) - \left( \frac{\eta_x}{2} - \frac{\eta_x^2 L}{2} \right) \Vert \nabla \fL_{\lambda}^*(x_t) \Vert^2 + \frac{\eta_x}{2} \Vert \BE G_t -  \nabla \fL_{\lambda}^*(x_t)  \Vert^2 + \frac{\eta_x^2 L}{2} \BE \Vert G_t -  \nabla \fL_{\lambda}^*(x_t)  \Vert^2 \\
    &\le \fL_{\lambda}^*(x_t) - \left( \frac{\eta_x}{2} - \frac{\eta_x^2 L}{2} \right) \Vert \nabla \fL_{\lambda}^*(x_t) \Vert^2 + \eta_x \Vert \BE G_t -  \nabla \fL_{\lambda}^*(x_t)  \Vert^2 + \frac{\eta_x^2 L}{2} \BE \Vert G_t -  \BE G_t \Vert^2.
    \end{align*}
    Plugging in \Eqref{eq:cond-HE} and \Eqref{eq:eta-outer-F2BSA}, we have that
\begin{align*}
     \BE[ \fL_{\lambda}^*(x_{t+1})] \le \fL_{\lambda}^*(x_t) - \frac{\eta_x}{4} \Vert \nabla \fL_{\lambda}^*(x_t) \Vert^2  + \fO(\eta_x \epsilon^2).
\end{align*}
Telescope over $t =0,1,\cdots,T-1$, we get
\begin{align} \label{eq:final-fosp}
    \frac{1}{T} \sum_{t=0}^{T-1} \BE \Vert \nabla \fL_{\lambda}^*(x_t) \Vert^2 \le \frac{4 (\fL_{\lambda}^*(x_t) - \inf_{x \in \BR^{d_x}} \fL_{\lambda}^*(x))}{ \eta_x T} + \fO(\epsilon^2),
\end{align}
which outputs an $\epsilon$-stationary point in $T = \fO( \Delta / (\eta_x \epsilon^2) )$ iterations.
\end{proof} 

\thminner* 
\begin{proof} 
By applying Theorem \ref{thm:optimal-SGD-SC}, we know that the SGD${}^{\rm SC}$ subroutine in Algorithm \ref{alg:single-F2HE} ensures that
\begin{align*} 
\Vert y_m^{(j)} - y_{\lambda}^*(x) \Vert^2 + \Vert z_m^{(j)} - y^*(x) \Vert^2 \le \frac{\delta}{2^{N+2j}} + \frac{\sigma_f^2 + \lambda^2 \sigma_g^2}{2^{j-3} \lambda \mu (L_f + \lambda L_g)}.
\end{align*}
Let $\nabla_x L_{\lambda}(x,y,z) =\nabla_x f(x,y) + \lambda (\nabla_x g(x,y) - \nabla_x g(x,z))$. Then, the estimator in Algorithm~\ref{alg:F2HE} for evert $m \in [M]$ satisfies that
\begin{align} \label{eq:single-F2HE}
\begin{split}
    &\quad \BE \Vert G_m^{(j)} - \nabla \fL_{\lambda}^*(x) \Vert^2 \\
    &\le 2 \BE \Vert G_m^{(j)} - L_\lambda(x, y_m^{(j)}, z_m^{(j)})\Vert^2 + 2 \BE \Vert \nabla \fL_{\lambda}^*(x) - \nabla_x L_\lambda(x, y_m^{(j)}, z_m^{(j)}) \Vert^2  \\
    &\le \frac{2(\sigma_f^2 + \lambda^2 \sigma_g^2)}{2^j} + 2(L_f^2 + \lambda^2 L_g^2)
    \cdot \left( \frac{\delta_t}{2^{N+2j}} + \frac{\sigma_f^2 + \lambda^2 \sigma_g^2}{2^{j-3} \lambda \mu (L_f + \lambda L_g)} \right) \\
    &\le 4 \lambda^2(L_f^2 + L_g^2)
    \cdot \left( \frac{\delta}{2^{N+2j}} + \frac{\sigma_f^2 + \lambda^2 \sigma_g^2}{2^{j-3} \lambda \mu (L_f + \lambda L_g)} \right).
\end{split}
\end{align}
The fact $\BP(J=j)=2^{-j}$ implies that 
\begin{align*} 
    \BE \hat G = \BE G_1 = \BE G_1^{(0)} + \sum_{j=1}^{S} \BP(J=j) 2^j ( G_1^{(j)} - G_1^{(j-1)}  ) = \BE \left[ G_1^{(S)} \right].
\end{align*}
Therefore, using \Eqref{eq:single-F2HE}, the bias of our estimator can be upper bounded by 
\begin{align*}
    &\quad \Vert \BE \hat G - \nabla \fL_{\lambda}^*(x) \Vert^2 = \Vert \BE G_1^{(S)} - \nabla \fL_{\lambda}^*(x) \Vert^2 \le \BE \Vert G_1^{(S)} - \nabla \fL_{\lambda}^*(x) \Vert^2 \\
    &\le 4\lambda^2(L_f^2 +  L_g^2)\cdot \left( \frac{\delta}{2^{N +2 S}} + \frac{\sigma_f^2 + \lambda^2 \sigma_g^2}{2^{S-3} \lambda \mu (L_f + \lambda L_g)  } \right).
\end{align*}
It is easy to check that $\Vert \BE \hat G - \nabla \fL_{\lambda}^*(x) \Vert^2 \le \zeta_{\rm bias}^2$ holds by our hyper-parameter setting. 
Below, we analyze the variance of our estimator. Note that
\begin{align*}
    &\quad \BE \Vert G_m - G_m^{(0)} \Vert^2 \\
    &= \sum_{j=1}^{S} \BP( J = j) 2^{2j} \BE \Vert G_m^{(j)} - G_m^{(j-1)} \Vert^2 \\
    &= \sum_{j=1}^{S} 2^j \BE \Vert G_m^{(j)} - G_m^{(j-1)} \Vert^2 \\
    & \le \sum_{j=1}^{S} 2^j \cdot \left( 2 \BE \Vert G_m^{(j)} - \nabla \fL_{\lambda}^*(x) \Vert^2 + 2 \BE \Vert G_m^{(j-1)} - \nabla \fL_{\lambda}^*(x) \Vert^2 \right) \\
    &\le 4\lambda^2(L_f^2 +  L_g^2) \sum_{j=1}^{S} \frac{10 \delta}{2^{N+j}} + \frac{48 (\sigma_f^2 + \lambda^2 \sigma_g^2)}{ \lambda \mu (L_f + \lambda L_g)} \\
    &\le 4\lambda^2(L_f^2 +  L_g^2)  \cdot \left(\frac{10 \delta}{2^N} +  \frac{48 (\sigma_f^2 + \lambda^2 \sigma_g^2)}{ \lambda \mu (L_f + \lambda L_g)} \cdot S \right),
\end{align*}
where the second last line uses \Eqref{eq:single-F2HE}.
Therefore, we have that
\begin{align*}
    &\quad \BE \Vert \hat G - \BE \hat G \Vert^2 = \frac{1}{M} \BE \Vert G_1 - \BE G_1 \Vert^2 \le \frac{1}{M} \BE \Vert G_1 - \nabla \fL_{\lambda}^*(x) \Vert^2 \\
    &\le \frac{2}{M} \BE \Vert G_1 - G_1^{(0)} \Vert^2 + \frac{2}{M} \BE \Vert G_1^{(0)} - \nabla \fL_{\lambda}^*(x) \Vert^2  \\
    &\le \frac{4\lambda^2(L_f^2 + L_g^2)}{M}  \cdot \left(\frac{22 \delta}{2^{N}} +  \frac{100 (\sigma_f^2 + \lambda^2 \sigma_g^2)}{ \lambda \mu (L_f + \lambda L_g)} \cdot S \right).
\end{align*}
It is also easy to check that $\BE \Vert \hat G - \BE \hat G \Vert^2 \le \zeta_{\rm variance}^2$ holds by our hyper-parameter setting.
Finally, the expected number of SFO calls is 
\begin{align*}
    \frac{8(L_f + \lambda L_g)}{\lambda \mu} \left( N + \sum_{j=1}^{S} \BP(J=j)  2^j  \right) M  = \frac{8(L_f + \lambda L_g)}{\lambda \mu} \left( N + S \right) M = \fO(\kappa N M).
\end{align*}
\end{proof}

\thmFSBAsingle* 

\begin{proof}
By Theorem \ref{thm:outer} the outer-loop complexity is $ T =  \fO \left( (\sigma_f^2 + \lambda^2 \sigma_g^2) \ell \kappa^3 \Delta \epsilon^{-4} \right) $. By Theorem \ref{thm:inner} the inner-loop complexity in the $t$-th iteration is $\tilde \fO \left( \kappa^2 \log^2 \delta_t \right)$ for $\delta_t$ satisfying $\Vert y_t^0 - y_{\lambda}^*(x_t) \Vert^2 + \Vert z_t^0 - y^*(x_t) \Vert^2 \le \delta_t$.
Similar to Theorem \ref{thm:F2BSA}, we also require to specify the recursion for $\delta_t$ with known problem parameters. Similar to \Eqref{eq:recursion-delta-t},we have that
\begin{align*}
   \BE \Vert y_{t+1}-  y_{\lambda}^*(x_{t+1}) \Vert^2  &\le 2 \BE \Vert y_{t+1} - y_{\lambda}^*(x_{t}) \Vert^2 + 2 \BE \Vert y_{\lambda}^*(x_{t}) - y_{\lambda}^*(x_{t+1}) \Vert^2 \\
    &\le \frac{1}{2} \Vert y_t - y_{\lambda}^*(x_t) \Vert^2 + \frac{16 (\sigma_f^2 + \lambda^2 \sigma_g^2)}{ \lambda \mu (L_f + \lambda L_g)} + \frac{32 L_g^2}{\mu^2} \Vert x_{t+1} - x_t \Vert^2, 
\end{align*}
where the last step follows Theorem \ref{thm:optimal-SGD-SC} with $N_t \ge 8$ and Lemma \ref{lem:nabla-y-lambda-bound} that $y_{\lambda}^*(x)$ is $(4L_g/\mu)$-Lipschitz continuous. Similarly, we also have that
\begin{align*}
\BE \Vert z_{t+1}-  y^*(x_{t+1}) \Vert^2   &\le 2 \BE \Vert z_{t+1} - y^*(x_{t}) \Vert^2 + 2 \BE \Vert y^*(x_{t}) - y^*(x_{t+1}) \Vert^2 \\
    &\le \frac{1}{2} \Vert z_t - y^*(x_t) \Vert^2 + \frac{8 \sigma_g^2}{ \mu \lambda L_g} + \frac{2 L_g^2}{\mu^2} \Vert x_{t+1} - x_t \Vert^2, 
\end{align*}
where the last step follows Theorem \ref{thm:optimal-SGD-SC} with $N_t \ge 8$
and that $y^*(x)$ is $(L_g/\mu)$-Lipschitz continuous. Combining the above two inequalities yields \Eqref{eq:recursion-delta-single}. Telescoping \Eqref{eq:recursion-delta-single} yields that
\begin{align*}
    \sum_{t=0}^{T-1} \BE[\delta_{t}] &\le 2\delta_0  + \frac{48 (\sigma_f^2+ \lambda^2 \sigma_g^2)}{\lambda \mu (L_f + \lambda L_g)} + \frac{68 L_g^2}{\mu^2} \sum_{t=0}^{T-1} \BE  \Vert x_{t+1} - x_t \Vert^2 \\
    &\le 2\delta_0 + \frac{48 (\sigma_f^2+ \lambda^2 \sigma_g^2)}{\lambda \mu (L_f + \lambda L_g)} + \frac{136 L_g^2 \eta_x^2}{\mu^2} \sum_{t=0}^{T-1} \BE [\Vert 
    \nabla \fL_{\lambda}^*(x_t)
    \Vert^2+ \Vert G_t - \nabla \fL_{\lambda}^*(x_t) \Vert^2] \\
    &\le 2\delta_0 + \frac{48 (\sigma_f^2+ \lambda^2 \sigma_g^2)}{\lambda \mu (L_f + \lambda L_g)} + \frac{136 L_g^2 \eta_x^2 T}{\mu^2} \cdot \left( \epsilon^2 + \sigma_f^2 + \lambda^2 \sigma_g^2 \right).
\end{align*}
Therefore, the total SFO complexity (in expectation) is
\begin{align*}
    \tilde \fO \left( \kappa^2 \sum_{t=0}^{T-1} \log^2 \delta_t \right) \le \tilde \fO \left( \kappa^2 T \max_{ 0 \le t \le T-1} \log^2 \delta_t \right) = \tilde \fO(\kappa^2 T),
\end{align*}
where $T$ is given in \Eqref{eq:eta-outer-F2BSA}.
\end{proof}

\subsection{The SGD Subroutine} \label{apx:SGD-SC}

This section recalls the optimal stochastic algorithm for smooth strongly-convex optimization \citep{allen2018make}, which be used as the subroutines in F${}^3$BSA (Algorithm \ref{alg:F3BSA}). 

\begin{asm} \label{asm:sgd}
Assume that the stochastic gradient satisfies 
\begin{align*}
    \BE_{\phi_t} \left[ \nabla h(x_t;\phi_t) \right] = \nabla h(x_t) , \quad \BE_{\phi_t} \Vert \nabla h(x_t;\phi_t) - \nabla h(x_t) \Vert^2 \le \sigma^2.
\end{align*}
\end{asm}

\begin{algorithm*}[htbp]  
\caption{SGD $(h, x_0, \eta, K)$} \label{alg:SGD-C}
\begin{algorithmic}[1] 
\STATE \textbf{for} $k = 0,1,\cdots,K-1$ \\
\STATE \quad Sample a random index $\phi_t$ \\[1mm]
\STATE \quad $x_{k+1} = x_t - \eta \nabla h(x_t ; \phi_t)$ \\[1mm]
\STATE \textbf{end for} \\[1mm]
\STATE \textbf{return} $\bar x_K = \frac{1}{K} \sum_{k=1}^K x_k$  
\end{algorithmic}
\end{algorithm*}

\begin{thm}{\citet[Theorem 4.1 (a)]{allen2018make}}
\label{thm:SGD-C}
Suppose $h(x): \BR^d \rightarrow \BR$ is  $\beta$-gradient Lipschitz and Assumption \ref{asm:sgd} holds for the stochastic gradient. Algorithm \ref{alg:SGD-C} outputs $\bar x_K$ such that
\begin{align*}
    \BE h(\bar x_K)  - h(x^*) \le \frac{\Vert x_0 -x^* \Vert^2}{2\eta K} +  \frac{\eta \sigma^2}{2 (1- \eta \beta)},
\end{align*}
where $x^* = \arg \min_{x \in \BR^{d_x}} h(x)$.

\end{thm}

Based on the above result, we can analyze the SGD${}^{\rm SC}$ algorithm \citep[Algorithm 2]{allen2018make}, which achieves the optimal rate for smooth strongly-convex stochastic optimization. We present this algorithm in Algorithm \ref{alg:SGD-SC}. The convergence of Algorithm \ref{alg:SGD-SC} can be found in {\citep[Theorem 4.1 (b)]{allen2018make}}. However, they did not explicitly calculate the constants in the convergence rates. Therefore, we provide a self-contained proof below.

\begin{algorithm*}[htbp]  
\caption{SGD${}^{\rm SC}$ $(h, x_0, N, K)$} \label{alg:SGD-SC}
\begin{algorithmic}[1] 
\STATE \textbf{for} $k = 1,\cdots,N$ \quad \textbf{do} $x_{k} = \text{SGD}(h, x_{k-1}, 1/(2\beta), 4 \beta /\alpha) $ \\ [1mm]
\STATE \textbf{for} $k= 1,\cdots,K$ \quad \textbf{do} $x_{N+k} = \text{SGD}(h, x_{N+k-1},  1/ (2^k \beta), 2^{k+2} \beta /\alpha)$ \\ [1mm] 
\STATE \textbf{return} $x_{N+K}$   
\end{algorithmic}
\end{algorithm*}

\begin{thm} \label{thm:optimal-SGD-SC}
    Suppose $h(x): \BR^d \rightarrow \BR$ is $\beta$-gradient Lipschitz and $\alpha$-strongly convex. Algorithm \ref{alg:SGD-SC} outputs $x_{N+K}$ satisfying
    \begin{align}
        \BE \Vert x_{N+K} - x^* \Vert^2 &\le \frac{\Vert x_0 - x^* \Vert^2}{2^{N+2K}} + \frac{\sigma^2}{2^{K-2} \alpha \beta}.
        \label{eq:EpochSGD-xK}
    \end{align}
The total stochastic first-order oracle (SFO) complexity is bounded by $ \lceil 4 \kappa ( N + 2^K) \rceil $, where $\kappa = \beta /\alpha$ is the condition number.
\end{thm}

\begin{proof}
Under the strong convexity assumption, Theorem \ref{thm:SGD-C} yields that
\begin{align} \label{eq:each-epoch}
    \BE \Vert \bar x_K - x^* \Vert^2 \le \frac{2}{\alpha} [\BE h(\bar x_K)  - h(x^*)] \le  \frac{\Vert x_0 - x^* \Vert^2}{\alpha \eta K} + \frac{\eta \sigma^2}{\alpha (1 - \eta \beta)}.
\end{align}
For the first $N$ epochs,  by the parameter setting in Algorithm \ref{alg:SGD-SC} and \Eqref{eq:each-epoch}, we have
\begin{align*}
    \BE \Vert x_{k} - x^* \Vert^2 \le \frac{\Vert x_{k-1} - x^* \Vert^2}{2} + \frac{\sigma^2}{\alpha \beta},
\end{align*}
Telescoping for all the $N$ epochs yields
\begin{align}
     \BE \Vert x_N - x^* \Vert^2 &\le \frac{\Vert x_0 -x^* \Vert^2}{2^N} + \frac{2 \sigma^2}{\alpha \beta}. \label{eq:EpochSGD-xN}
\end{align}
For the subsequent $K$ epochs, by the parameter setting in Algorithm \ref{alg:SGD-SC} and \Eqref{eq:each-epoch}, we have that 
\begin{align*}
   \BE \Vert x_{N+k} - x^* \Vert^2 \le \frac{\Vert x_{N+k-1} - x^* \Vert^2}{4} + \frac{\sigma^2}{2^{k-1} \alpha \beta}.
\end{align*}
Then it is easy to prove by induction that
\begin{align} \label{eq:induction-SGD}
    \BE \Vert x_{N+k} - x^* \Vert^2 \le \frac{\Vert x_0 - x^* \Vert^2}{2^{N+2k}} + \frac{\sigma^2}{2^{k-2} \alpha \beta}.
\end{align}
In the above, the induction base follows \Eqref{eq:EpochSGD-xN}. If we suppose that \Eqref{eq:induction} holds for $k=j-1$, then for $k=j$ we have that
\begin{align*}
    \BE \Vert x_{N+j+1} - x^* \Vert^2 &\le \frac{1}{4} \cdot \left( \frac{\Vert x_0 -x^* \Vert^2}{2^{N+2(j-1)}} + \frac{\sigma^2}{2^{j-3} \alpha \beta} \right) + \frac{\sigma^2}{2^{j-1} \alpha \beta} \\
    &\le \frac{\Vert x_0 - x^* \Vert^2}{2^{N+2j}} + \frac{\sigma^2}{2^{j-2} \alpha \beta},
\end{align*}
which completes the induction of \Eqref{eq:induction-SGD}.
\end{proof}
}

\section{Proofs of Finding Second-Order Stationarity} \label{apx:2nd}

The following theorem generalizes the result of Perturbed Gradient Descent~\citep{jin2017escape}
to Inexact Perturbed Gradient Descent by
allowing an inexact gradient $\hat \nabla h(x)$ that is close to the exact gradient $\nabla h(x)$.

\begin{algorithm*}[htbp]  
\caption{Inexact Perturbed Gradient Descent} \label{alg:iPGD}
\begin{algorithmic}[1] 
\STATE \textbf{for} $t = 0,1,\cdots,T-1$ \\
\STATE \quad \textbf{if} $ \Vert \hat \nabla h(x_t) \Vert \le \frac{4}{5} \epsilon$ \textbf{and} no perturbation added in the last $\fT$ steps \\[1mm]
\STATE \quad \quad $x_{t} = x_t - \eta \xi_t $, where $\xi_t \sim \sB(r)$ \\[1mm]
\STATE \quad \textbf{end if} \\[1mm]
\STATE \quad $x_{t+1} = x_t -  \eta  \hat \nabla h(x_t)$ \\[1mm]
\STATE \textbf{end for} \\[1mm]
\end{algorithmic}
\end{algorithm*}

\begin{thm}[{\citet[Lemma 2.9]{huang2022efficiently}}]
    \label{thm:iPGD}
Suppose $h(z): \BR^d \rightarrow \BR$ is
$L$-gradient Lipschitz and $\rho$-Hessian Lipschitz. 
Set the parameters in Algorithm \ref{alg:iPGD} as 
\begin{align} \label{cond:iota}
    \eta = \frac{1}{L}, \quad r = \frac{\epsilon}{400 \iota^3}, \quad \fT = \frac{L}{\sqrt{\rho \epsilon}} \cdot \iota ~~{\rm with }~~ \iota \ge 1 ~~ {\rm and} ~~ \delta \ge  \frac{L \sqrt{d}}{\sqrt{\rho \epsilon}} \iota^2 2^{8-\iota}.
\end{align}
Once the following condition holds in each
iteration:
\begin{align} \label{cond:PF2BA}
    \Vert \hat \nabla h(x_t) - \nabla h(x_t) \Vert \le \underbrace{\min \left\{ 
    \frac{1}{20 \iota^2},  \frac{1}{16 \iota^2 2^{\iota}} \right\} \epsilon}_{:=\zeta},
\end{align}
then we can find a point $x_t$ satisfying
\begin{align*}
    \Vert \nabla h(x_t) \Vert \le \epsilon, \quad  \nabla^2 h(x_t)  \succeq - \sqrt{\rho \epsilon}~ I_{d}
\end{align*}
within $T= \fO \left( {\iota^4  \Delta L}{\epsilon^{-2}}   \right)$ iterations with probability $1-\delta$, where
$\Delta = h(x_0) - \inf_{x \in \BR^{d}} h(x)$.    
\end{thm}

Then we can show the convergence of perturbed F${}^2$BA.
\thmPFBA*
\begin{proof}
Let $L$ be the gradient Lipschitz coefficient and $\rho$ be the Hessian Lipschitz coefficient of $\fL_{\lambda}^*(x)$, respectively.
According to Lemma \ref{lem:1st}, Lemma \ref{lem:2rd}, Lemma \ref{lem:yy-lambda}  and the setting of $\lambda$, we have
\begin{enumerate}[label=\alph*.]
    \item $\sup_{x \in \BR^{d_x}} \Vert \nabla \fL_{\lambda}^*(x) - \nabla \varphi(x) \Vert = \fO(\epsilon)$. 
    \item $\sup_{x \in \BR^{d_x}} \Vert \nabla^2 \fL_{\lambda}^*(x) - \nabla^2 \varphi(x) \Vert = \fO(\sqrt{\rho \epsilon})$.
    \item $\fL_{\lambda}^*(x_0) - \inf_{x \in \BR^{d_x}}  \fL_{\lambda}^*(x) = \fO(\Delta)$.
    \item $\Vert y_0 - y_{\lambda}^*(x_0) \Vert^2 + \Vert y_0 - y^*(x_0) \Vert^2 =\fO(R)$.
    \item $ L:=\sup_{x \in \BR^{d_x}}\Vert \nabla^2 \fL_{\lambda}^*(x) \Vert = \fO(\ell \kappa^3)$.
    \item $\rho: = \sup_{x \in \BR^{d_x}} \Vert \nabla^3 \fL_{\lambda}^*(x) \Vert = \fO(\ell \kappa^5) $.
\end{enumerate}
Then it suffices to show that the algorithm can find an $\epsilon$-second-order stationary point of $\fL_{\lambda}^*(x)$ under our choice of parameters.
We apply Theorem \ref{thm:iPGD} to prove this result. 
Recall $\zeta$ defined in \Eqref{cond:PF2BA}.
It remains to show that 
\begin{align} \label{eq:cond}
    \Vert \nabla \fL_{\lambda}^*(x_t) - \nabla \fL_{\lambda}^*(x_t) \Vert \le \zeta
\end{align}
holds for all $t$. Recall that
\begin{align*}
    \Vert \hat \nabla \fL_{\lambda}^*(x_t) - \nabla \fL_{\lambda}^*(x_t) \Vert \le 2 \lambda L_g \Vert y_{t}^K- y_{\lambda}^*(x_t) \Vert + \lambda L_g \Vert z_{t}^K - y^*(x_t) \Vert.
\end{align*} 
 If we let
\begin{align} \label{eq:Kt}
    K_t = \frac{4 L_g}{\mu} \log \left( 
    \frac{4 \lambda L_g \delta_t}{\zeta} 
    \right), 
\end{align}
where $\delta_t \ge \Vert y_t^0 - y_{\lambda}^*(x_t) \Vert^2 + \Vert z_t^0 - y^*(x_t) \Vert^2  $.
Then by Theorem \ref{thm:GD-SC}, we have
\begin{align} \label{eq:induction}
     \max \left\{ \Vert y_t^K - y_{\lambda}^*(x_t) \Vert , ~\Vert z_t^K-  y^*(x_t) \Vert   \right\} \le \frac{\zeta}{4 \lambda L_g},
\end{align}
which implies \Eqref{eq:cond}. 
By \Eqref{eq:delta_t}, if we let $K_t \ge 8L_g/\mu$ for all $t$, then we can show that
\begin{align}  \label{eq:dekta-recur}
     \delta_{t+1}
    &\le \frac{1}{2} \delta_t +\frac{34 L_g^2}{\mu^2} \Vert x_{t+1} - x_t \Vert^2.
\end{align}
The total number of iterations can be bounded by
\begin{align} \label{plug:2}
\sum_{t=0}^{T-1} K_t \le \frac{4L_g}{\mu} \sum_{t=0}^{T-1} \log \left( 
\frac{4 \lambda L_g \delta_t}{\zeta}
\right)   \le \frac{4L_g T}{\mu} \log \left( \frac{4 \lambda L_g \sum_{t=0}^{T-1}\delta_t}{\zeta T} \right)
\end{align}
 first-order oracle calls. It remains to show that $\sum_{t=0}^{T-1} \delta_t$ is bounded.
Telescoping over $t$ in \Eqref{eq:dekta-recur}, we obtain
\begin{align} \label{plug:1}
    \sum_{t=0}^{T-1} \delta_{t} \le 2 \delta_0 + \frac{68L_g^2}{\mu^2} \sum_{t=0}^{T-1} \Vert x_{t+1} - x_t \Vert^2.
\end{align}
By \Eqref{eq:final}, if we let $K_t \ge \Omega( \kappa \log(\lambda \kappa))$ for all $t$, we can obtain
\begin{align*}
\frac{1}{8 \eta} \sum_{t=0}^{T-1} \Vert x_{t+1} - x_t \Vert^2 \le \fL_{\lambda}^*(x_0) - \inf_{x \in \BR^{d_x}} \fL_{\lambda}^*(x) + \Vert y_0 - y_{\lambda}^*(x_0) \Vert^2 + \Vert y_0 - y^*(x_0) \Vert^2 .
\end{align*}
Plugging into \Eqref{plug:1} and then \Eqref{plug:2} yields the upper complexity bound of $\tilde \fO(\ell \kappa^4 \epsilon^{-2})$ as claimed. 
\end{proof}

The following theorem generalizes the result of Nonconvex AGD~\citep{li2023restarted} by allowing an inexact gradient. To simplify the description, we define one round of the algorithm between
two successive restarts to be one ``epoch'' by following \citet{li2023restarted}.

\begin{algorithm*}[htbp]  
\caption{Inexact Nonconvex Accelerated Gradient Descent} \label{alg:iAGD}
\begin{algorithmic}[1] 
\STATE $x_{-1} = x_0 $ \\
\STATE \textbf{while} $t <T$ \\
\STATE \quad $x_{t+1/2} = x_t + (1- \theta)(x_t - x_{t-1})$ \\
\STATE \quad $x_{t+1} = x_{t+1/2} - \eta \hat \nabla h(x_{t+1/2}) $ \\
\STATE \quad $t = t+1$ \\
\STATE \quad \textbf{if} $t \sum_{j=0}^{t-1} \Vert x_{j+1} - x_j \Vert^2 > B^2$ \\
\STATE \quad \quad $t=0$, $x_{-1} = x_0 = x_t + \xi_t \vone_{ \Vert \hat \nabla h(x_{t+1/2}) \Vert \le \frac{B}{2 \eta} }, ~{\rm where} ~~\xi_t \sim \sB(r)$ \\
\STATE \quad \textbf{end if} \\
\STATE \textbf{end while} \\
\STATE $T_0 = \arg \min_{ \lfloor \frac{T}{2} \rfloor \le t \le T-1} \Vert x_{t+1} - x_t \Vert$ \\
\STATE \textbf{return} $ x_{\rm out} = \frac{1}{T_0+ 1} \sum_{t=0}^{T_0} x_{t+1/2}$ \\
\end{algorithmic}
\end{algorithm*}

\begin{thm}[{\citet[Theorem 4.1]{yang2023accelerating}}]
    \label{thm:iAGD}
Suppose $h(z): \BR^d \rightarrow \BR$ is
$L$-gradient Lipschitz and $\rho$-Hessian Lipschitz. Assume that $\sqrt{\epsilon \rho } \le L$.
Set the parameters in Algorithm \ref{alg:iAGD} as 
\begin{align*}
    &\chi = \fO( \log (\nicefrac{d}{\delta \epsilon})),~~ \eta = \frac{1}{4 L}, ~~ T = \frac{2 \chi}{\theta}, ~ \theta = \frac{1}{2} ( \rho \epsilon \eta^2 )^{1/4}<1,\\
    & B = \frac{1}{288 \chi^2} \sqrt{\frac{\epsilon}{\rho}}, ~~ 
     r =\min \left\{ \frac{B + B^2}{\sqrt{2}}, \frac{\theta B}{20 T}, \frac{\theta B^2}{2T} \right\} = \fO(\epsilon).
\end{align*}
Once the following condition holds in each iteration:
\begin{align} \label{cond:RAGD}
     \Vert \hat \nabla h(x_t) - \nabla h(x_t) \Vert \le \underbrace{\min \left\{ \frac{\rho B \zeta r \theta}{2 \sqrt{d}}, \epsilon^2 \right\}}_{:=\zeta}.
\end{align}
The algorithm ensures that
\begin{align} \label{eq:des-AGD}
    h(x_{\fT}) - h(x_0) \le - \frac{\epsilon^{1.5}}{663552 \sqrt{\rho} \chi^5},
\end{align}
where $\fT$ denotes the iteration number when ``if'' condition in Algorithm \ref{alg:iAGD} (Line 6) triggers:
\begin{align*}
\fT = \min_{t'} \left\{ t' \mid t' \sum_{t=0}^{t'-1} \Vert x_{t+1} - x_t \Vert^2 > B^2 \right\}.
\end{align*}
Therefore, the algorithm terminates in at most $\fO(\Delta \sqrt{\rho} \chi^5 \epsilon^{-3/2})$ epochs and $ \fO( \Delta \sqrt{L} \rho^{1/4} \chi^6 \epsilon^{-1.75})$ iterations. In addition, the output satisfies
\begin{align*}
    \Vert \nabla h(x_t) \Vert \le \epsilon, \quad  \nabla^2 h(x_t)  \succeq - \sqrt{\rho \epsilon}~ I_{d}, 
\end{align*}
with probability $1-\delta$, where
$\Delta = h(x_0) - \inf_{x \in \BR^{d}} h(x)$.    
\end{thm}

\thmAccFBA*

\begin{proof}
Let $T'$ be the number of total iterations. We renumber the iterates in the algorithm as $\{x_{t} \}_{t=0}^{T'-1}$,  $\{x_{t+1/2} \}_{t=0}^{T'-1}$, and let the corresponding inner loop number as $\{K_t \}_{t=0}^{T'-1}$.
Identical to the proof of Theorem \ref{thm:PF2BA} (with a different definition of $\zeta$ that only effects the logarithmic factors),
we can set $K_t = \fO(\kappa \log (\nicefrac{\lambda L_g \delta_t}{\zeta}))$ to
ensure Condition \ref{cond:RAGD} to hold, where 
$\delta_t \ge \Vert y_t^0 - y_{\lambda}^*(x_{t+1/2}) \Vert^2 + \Vert z_t^0 - y^*(x_{t+1/2}) \Vert^2 $ is ensured recursively identical to \Eqref{eq:dekta-recur}.
By \Eqref{plug:2} we have
\begin{align*}
    \sum_{t=0}^{T'-1} K_t \le \frac{4L_g T}{\mu} \log \left( \frac{4 \lambda L_g \sum_{t=0}^{T'-1}\delta_t}{\zeta T'} \right)
\end{align*}
It remains to show that $\sum_{t=0}^{T'-1} \delta_t$ is bounded.  By \Eqref{plug:1}, it suffices to show that $\sum_{t=0}^{T'-1} \Vert x_{t+3/2} - x_{t+1/2} \Vert^2  $ is bounded. Since $x_{t+1/2} = x_t + (1-\theta) (x_t - x_{t-1})$, it suffices to show that $\sum_{t=0}^{T'-1} \Vert x_{t+1} - x_t \Vert^2$ is bounded.
Because the number of epochs is finite by Theorem \ref{thm:iAGD}, it suffices to bound the term in each epoch. 
We know that
\begin{align*}
    &\quad \sum_{t=0}^{\fT -1} \Vert x_{t+1} - x_t \Vert^2  \\
    &= \sum_{t=0}^{\fT -2} \Vert x_{t+1} - x_t \Vert^2  + \Vert x_{\fT} - x_{\fT-1} \Vert^2 \\
    &\le B^2 + \Vert x_{\fT} - x_{\fT-1} \Vert^2.
\end{align*}
We continue to bound  the last term $\Vert x_{\fT} - x_{\fT-1} \Vert^2$.
Since we have
\begin{align*}
    x_{t+1} - x_{t} &= (1- \theta) (x_t - x_{t-1}) - \eta_x \hat \nabla \fL_{\lambda}^*(x_{t+1/2}).
\end{align*}
It remains to upper bound $ \Vert \hat \nabla \fL_{\lambda}^*(x_{\fT-1/2}) \Vert$. 
As $ \Vert \hat \nabla \fL_{\lambda}^*(x_{\fT-1/2}) - \nabla \fL_{\lambda}^*(x_{\fT-1/2}) \Vert = \fO(\zeta)$, it remains to upper bound $\Vert \nabla \fL_{\lambda}^*(x_{\fT-1/2}) \Vert$. 
By (18) in \citet{yang2023accelerating}, we have that
\begin{align*}
    \fL_{\lambda}^*(x_{\fT}) \le \fL_{\lambda}^*(x_0) + \frac{B^2}{8 \eta_x} - \frac{3 \eta_x}{8} \Vert \nabla \fL_{\lambda}^*(x_{\fT-1/2}) \Vert^2 + B \zeta + \frac{5 \eta_x \zeta^2 \fT}{8}.
\end{align*}
This implies that $\Vert \nabla \fL_{\lambda}^*(x_{\fT-1/2}) \Vert$ is bounded once $\Delta$ is bounded.
\end{proof}

\section{Implement Our Algorithms in the Distributed Setting} \label{apx:dist}

Our algorithms can also be easily applied to the distributed scenario, when both the upper and lower-level functions adopt the following finite-sum structure:
\begin{align}  \label{prob:fn}
    f(x,y) := \frac{1}{m} \sum_{i=1}^m f_i(x,y), \quad g(x,y) := \frac{1}{m} \sum_{i=1}^m g_i(x,y),
\end{align}
where $f_i$ and $g_i$ denote the local function on the $i$-th agent. 
Under this setting, each agent $i$ has its own local variables within the optimization algorithm: $X(i) \in \BR^{d_x}, Y(i) \in \BR^{d_y},Z(i) \in \BR^{d_y}$.
For the convenience of presenting the algorithm, we aggregate all the local variables (denoted by row vectors) in one matrix and denote
\begin{align*}
X=\begin{bmatrix}
X(1) \\ \vdots \\ X(m)
\end{bmatrix}&\in\BR^{m\times d_x},
\quad
Y=\begin{bmatrix}
Y(1) \\ \vdots \\ Y(m)
\end{bmatrix}\in\BR^{m\times d_y}
\quad \text{and} \quad
Z=\begin{bmatrix}
Z(1) \\ \vdots \\ Z(m)
\end{bmatrix}\in\BR^{m\times d_y},
\end{align*}
We denote~$\vone = [1,\cdots,1]^\top \in \BR^m$ and use the lowercase with the bar to represent the mean vector, such as $\bar x = \frac{1}{m} \vone^\top X$.
Let $d = d_x+d_y$ and we similarly define the aggregated gradients 
{\small \begin{align*}
\nabla \vf(X,Y)=\begin{bmatrix}
\nabla f_1(X(1),Y(1)) \\ \vdots \\ \nabla f_m(X(m),Y(m))
\end{bmatrix}
\in\BR^{m\times d}, ~\nabla \vg(X,Y)=\begin{bmatrix}
\nabla g_1(X(1),Y(1)) \\ \vdots \\ \nabla g_m(X(m),Y(m))
\end{bmatrix}
\in\BR^{m\times d}.    
\end{align*}}

\begin{algorithm*}[t]  
\caption{Distributed F${}^2$BA $\left(\bar x_0,\bar y_0, \eta_x,\eta_y,\eta_z, \lambda, T,K\right)$} \label{alg:CF2BA}
\begin{algorithmic}[1] 
\STATE $ X_0 = \vone  \bar x_0, ~ Y_0 = \vone  \bar y_0,~ Z_0 = \vone  \bar y_0$ \\[1mm]
\STATE \textbf{for} $ t =0,1,\cdots,T-1 $ \\[1mm]
\STATE \quad $ Y_t^0 =  Y_{t}, ~ Z_t^0(i) = Z_{t}$ \\[1mm]
\STATE \quad \textbf{for} $ k =0,1,\cdots,K-1$ \\[1mm]
\STATE \quad \quad  $ V_t^{k} = \lambda \nabla_y \vg(X_t, Z_t^k)$, \quad  $U_t^{k} =  \nabla_y \vf(X_t,Y_t^k) +  \lambda  \nabla_y \vg(X_t,Y_t^k) $ \\[1mm]
\STATE \quad \quad Aggregate and broadcast $\bar v_t^{k} = \frac{1}{m} \vone \vone^\top V_t^k, $~ $\bar u_t^{k} = \frac{1}{m} \vone   \vone^\top U_t^k$ \\[1mm]
\STATE \quad \quad  $ Z_t^{k+1} = Z_t^k - \eta_z \vone   \bar v_t^k   $ ,\quad  $ Y_t^{k+1} = Y_t^k - \eta_y \vone \bar u_t^k   $ \\[1mm]
\STATE \quad \textbf{end for} \\[1mm]
\STATE \quad $ H_t= \nabla_x \vf(X_t,Y_{t}^K) + \lambda ( \nabla_x \vg(X_t,Y_{t}^K) - \nabla_x \vg(X_t,Z_{t}^K) )$ \\[1mm]
\STATE \quad Aggregate and broadcast $\bar h_t = \frac{1}{m} \vone   \vone^\top H_t$ \\[1mm]
\STATE \quad $X_{t+1} = X_t -  \eta_x  \vone  \bar h_t$ \\[1mm]
\STATE \textbf{end for} \\[1mm]
\end{algorithmic}
\end{algorithm*}


The classic paradigm for distributed algorithms is to compute the local gradients on each agent in parallel, and then aggregate them on the server. 
However, challenges arise when extending existing HVP-based methods for bilevel optimization from the single-machine setting to the distributed setting. Given a variable $\bar x$, one may want to calculate its hyper-gradient (\Eqref{hyper-grad}) according to this paradigm by:
\begin{align*}
    \hat \nabla \varphi(\bar x) = {\bf \Pi} \left( \nabla_x \vf(\vone x, \vone y^*(\bar x)) - \nabla_{xy}^2 \vf(\vone x,\vone y^*(\bar x)) [\nabla_{yy}^2 \vg(\vone \bar x, \vone y^*(\bar x))]^{-1} \nabla_y \vf(\vone \bar x, \vone y^*(\bar x))  \right),
\end{align*}
where $y^*(x) := \arg \min_{y \in \BR^{d_y}} g(x,y)$ and ${\bf \Pi} := \frac{1}{m} \vone \vone^\top$ denotes the aggregation operator on the server. 
But the nested structure of the hyper-gradient indicates that $ \hat \nabla \varphi(\bar x) \ne \nabla \varphi(\bar x)$. As a consequence, 
researchers need to pay extra effort to make HVP-based methods work in distributed bilevel problems~\citep{tarzanagh2022fednest,chen2022decentralized}.
It is a non-trivial issue to address, especially when the operator $ {\bf \Pi} \nabla_{yy}^2 \vg(\vone \bar x,\vone y^*(\bar x)) $ is forbidden to avoid an unacceptable $\fO(d_y^2)$ communication complexity.

In contrast, the distributed extension of F${}^2$BA naturally avoid this challenge since
\begin{align*}
    \nabla \fL_{\lambda}^*(\bar x) = {\bf \Pi} \left( \nabla_x \vf(\vone x,\vone y_{\lambda}^*(\bar x)) + \lambda (\nabla_x \vg(\vone \bar x, \vone y_{\lambda}^*(\bar x) ) - \nabla_x \vg (\vone \bar x, \vone y^*(\bar x)) ) \right),
\end{align*}
where $y_{\lambda}^*(x) := \arg \min_{y \in \BR^{d_y}} f(x,y) + \lambda g(x,y)$. It means that the previously mentioned classic aggregation paradigm for distributed optimization directly works for F${}^2$BA without any additional modification in the algorithm, as stated below.

\begin{prop} 
Running Algorithm \ref{alg:CF2BA} is equivalent to running  Algorithm \ref{alg:F2BA}  on the mean variables
\begin{align*}
    \bar x = \frac{1}{m} \vone^\top X,~ \bar y = \frac{1}{m} \vone^\top Y ~{\rm and}~ \bar z = \frac{1}{m} \vone^\top Z.
\end{align*}
\end{prop}

Then the convergence of Algorithm \ref{alg:CF2BA} directly follows Theorem \ref{thm:F2BA}.
\begin{cor} \label{cor:CF2BA}
Suppose Assumption \ref{asm:sc} holds.
Algorithm \ref{alg:CF2BA} with the same parameters in Theorem \ref{thm:F2BA} can find  $X_t$ satisfying $ \Vert \nabla \varphi(\bar x_t) \Vert \le \epsilon $ within $\fO(\ell \kappa^4 \epsilon^{-2} \log(\nicefrac{\ell \kappa}{\epsilon})) $ {iterations} and  $\fO( \ell \kappa^4 \epsilon^{-2} \log(\nicefrac{\ell \kappa}{\epsilon})) $ communication rounds.
\end{cor}

Algorithm \ref{alg:CF2BA} is a near-optimal distributed algorithm since both the iteration and communication complexity match the $\Omega(\epsilon^{-2})$ lower bound (Theorem 1 by \citet{lu2021optimal}), up to logarithmic factors.
Compared with the HVP-based methods, the distributed F${}^2$BA is more
practical since it
neither requires a $\fO(d_y^2)$ communication complexity per iteration~\citep{yangdecentralized} nor an additional distributed sub-solver for matrix inverse~\citep{tarzanagh2022fednest,chen2022decentralized}. 

{We remark that all algorithms in our main text can be implemented in a similar way as we implement F${}^2$BA in the distributed scenario.
}

\section{Discussions on Closely Related Works}

In this section, we discuss our method with closely related works, including \citep{shen2023penalty,kwon2023fully}.
Both our work and theirs use a penalty method. 
The main difference in algorithm design is that their method uses a single-time-scale update in $(x,y)$ while we use a two-time-scale update to improve their $\fO( \lambda \epsilon^{-2} \log(\nicefrac{1}{\epsilon}))$ first-order complexity to $\fO(\epsilon^{-2} \log \lambda )$, where $\lambda \asymp \epsilon^{-1}$ due to Lemma \ref{lem:1st}.

In the following, we give a more detailed comparison of another aspects, including the optimality notion and assumptions of our work and theirs.

\subsection{Discussions on the Weaker Notions in \citep{shen2023penalty}} \label{apx:diss-shen}

\citet{shen2023penalty} also proposed a similar penalty-based algorithm, and proved that the penalty method can find an $\epsilon$-stationary point (formally defined below) of $\fL_{\lambda}(x,y)$ in $\fO(\lambda \epsilon^{-2} \log (\nicefrac{1}{\epsilon}))$ first-order oracle calls.
Below, we show that their result implies a $\fO(\epsilon^{-3} \log (\nicefrac{1}{\epsilon})$ first-order complexity to find an $\epsilon$-stationary point of $\varphi(x)$.

\begin{dfn} \label{dfn:sta-Lambda}
We say $(x,y)$ is an $\epsilon$-stationary point of $\fL_{\lambda}(x,y)$ if $\Vert \nabla \fL_{\lambda}(x,y) \Vert \le \epsilon$.
\end{dfn}

We show that their notion and ours can be translated in both directions.

\begin{lem} \label{lem:translate}
Suppose Assumption \ref{asm:sc} and \ref{asm:stochastic} hold. Set $\lambda$ as Theorem \ref{thm:F2BSA}.
\begin{enumerate}
    \item If a point $x$ is an $\epsilon$-stationary point of $\varphi(x)$ in terms of Definition \ref{dfn:sta-hyper}, then we can find a point $y$ such that $(x,y)$ is an $2\epsilon$-stationary point of $\fL_{\lambda}(x,y)$ in terms of Definition \ref{dfn:sta-Lambda} with $\fO(\kappa \log(\nicefrac{\kappa}{\epsilon}))$,
    $\fO(\kappa^2 \epsilon^{-2})$, or $\fO(\kappa^8 \epsilon^{-2})$ first-order oracle calls under the deterministic $(\sigma_f = \sigma_g=0)$, partially stochastic ($\sigma_f >0, \sigma_g=0$) and fully stochastic setting $(\sigma_f >0, \sigma_g>0)$ respectively.
    \item If a point $(x,y)$ is an $\epsilon$-stationary point of $\fL_{\lambda}(x,y)$ in terms of Definition \ref{dfn:sta-Lambda}, then it is a $\fO(\kappa \epsilon)$-stationary point of $\varphi(x)$ in terms of Definition \ref{dfn:sta-hyper}.
\end{enumerate}
\end{lem}

\begin{proof}
Note that $\varphi(x) = \min_{y \in \BR^{d_y}} \fL_{\lambda}(x,y)$. Then, by Danskin's lemma, we have $\nabla \varphi(x) = \nabla_x \fL_{\lambda}(x,y_{\lambda}^*(x))$. Therefore, we further have
\begin{align*}
    \Vert \nabla \varphi(x) \Vert &\le \Vert \nabla_x \fL_{\lambda}(x,y) \Vert + 2 \lambda L_g \Vert y_{\lambda}^*(x) - y^*(x) \Vert \\
    &\le \Vert \nabla_x \fL_{\lambda}(x,y) \Vert + \frac{4 L_g}{\mu} \Vert \nabla_y \fL_{\lambda}(x,y) \Vert,
\end{align*}
where we use the fact that $\fL_{\lambda}(x,y)$ is $(2\lambda L_g)$-gradient Lipschitz and $(\lambda \mu /2)$-strongly convex in $y$. This shows that if $(x,y)$ is an $\epsilon$-stationary point of $\fL_{\lambda}(x,y)$, then $x$ is also an $\fO(\kappa \epsilon)$-stationary point of $\varphi(x)$.
\textbf{Conversely}, if $x$ is an $\epsilon$-stationary point of $\varphi(x)$, the it suffices to find $y$ such that $ \Vert \nabla_y f(x,y) \Vert \le \fO(\epsilon/\kappa)$, then $(x,y)$ is a $2\epsilon$-stationary point of $\fL_{\lambda}(x,y)$. Under the deterministic setting, this goal can be achieved with $\fO(\kappa \log(\nicefrac{\kappa}{\epsilon}))$ first-order oracle calls by Theorem \ref{thm:GD-SC}. Under the stochastic setting, this goal can be achieved with $\fO(\kappa^2 (\sigma_f^2 + \lambda^2 \sigma_g^2)  \epsilon^{-2} )$ first-order oracle calls by \citep[Theorem 3]{allen2018make}.
\end{proof}

Note that one has to pay an additional $\fO(\kappa)$ factor when translating a stationary point of $\fL_{\lambda}(x,y)$ using the notation of \citet{shen2023penalty} to that of $\varphi(x)$ using our notation. Therefore, our notation is stronger than that of \citet{shen2023penalty}.

\subsection{Discussion on the Additional Assumptions in \citep{kwon2023fully}}  \label{apx:diss-kwon}
F${}^2$SA ~\citep[Algorithm 1]{kwon2023fully} additionally uses the following assumption.

\begin{asm} \label{asm:F2SA}[Assumption 4 and 5 in \citet{kwon2023fully}]
Besides Assumption \ref{asm:sc}, they also assume that
\begin{enumerate}
    \item $f(x,y)$ is $C_f$-Lipschitz in $x$;
    \item $g(x,y)$ is $C_g$-Lipschitz in $y$;
    \item $f(x,y)$ is $\rho_f$-Hessian Lipschitz in $(x,y)$.
\end{enumerate}
\end{asm}

The role of Assumption \ref{asm:F2SA} is to ensure the Lipschitz continuity of $y_{\lambda}^*(x)$ in $\lambda$ such that the algorithm can use an increasing $\lambda$. 
However, this additional assumption is unnecessary if we use a fixed $\lambda$. \citet{kwon2023fully} sets $\lambda_{t+1} = \lambda_t + \delta_t$ with the requirement 
\begin{align*}
    \frac{\delta_t}{\lambda_t} \le \frac{1}{16} \min \left\{1, \frac{T \mu}{L_g} \right\}.
\end{align*}
See \citep[Theorem 4, Condition (3b)]{kwon2023fully}. Therefore, setting $\delta_t = 0$ also meets the requirement.
In this case, when we discuss the algorithm by \citet{kwon2023fully}, we refer to this version with a fixed $\lambda$. This allows a fair comparison with our algorithm. After both algorithms use the same penalty $\lambda$, the only difference lies in our step size, which is the key to our improved analysis.

\bibliography{sample}

\end{document}